\documentclass{amsart}
\usepackage{pgf}
\usepackage{amssymb}
\usepackage{dsfont}
\usepackage{txfonts}
\usepackage{graphicx}
\usepackage{amsmath}
\usepackage[colorlinks=true]{hyperref}
\usepackage{amsthm}
\usepackage{units}

\newtheorem{theo}{Theorem}[section]
\newtheorem{definition}[theo]{Definition}
\newenvironment{defi}{\begin{definition}\rm}{\end{definition}}
\newtheorem{remarque}[theo]{Remark}
\newenvironment{remark}{\begin{remarque}\rm}{\end{remarque}}
\newtheorem{exemple}[theo]{Example}
\newenvironment{ex}{\begin{exemple}\rm}{\end{exemple}}
\newtheorem{lemma}[theo]{Lemma}

\newtheorem{coro}[theo]{Corollary}
\newtheorem{nota}[theo]{Notation}
\newenvironment{notation}{\begin{nota}\rm}{\end{nota}}
\newtheorem{prf1}{\it {Idea of the proof}}

\newtheorem{prf}{\it{Proof}}

\newenvironment{demo}{\begin{prf}\rm}{\hfill$\Box$\end{prf}}

\newtheorem*{theorem*}{Theorem}

\def\varLim@#1#2{%
\vtop{\m@th\ialign{##\cr
\hfil$#1\operator@font Lim$\hfil\cr
\noalign{\nointerlineskip\kern1.5\ex@}#2\cr
\noalign{\nointerlineskip\kern-\ex@}\cr}}%
}
\def\varinjLim{%
\mathop{\mathpalette\varLim@{\rightarrowfill@\textstyle}}\nmlimits@
}

\makeatletter
\def\moverlay{\mathpalette\mov@rlay}
\def\mov@rlay#1#2{\leavevmode\vtop{%
		\baselineskip\z@skip \lineskiplimit-\maxdimen
		\ialign{\hfil$\m@th#1##$\hfil\cr#2\crcr}}}
\newcommand{\charfusion}[3][\mathord]{
	#1{\ifx#1\mathop\vphantom{#2}\fi
		\mathpalette\mov@rlay{#2\cr#3}
	}
	\ifx#1\mathop\expandafter\displaylimits\fi}
\makeatother

\newcommand{\bigcupdot}{\charfusion[\mathop]{\bigcup}{\cdot}}

\def\N{\mathbb{ N}}
\def\Z{\mathbb{ Z}}

\title{About algebraic Puiseux series in several variables.} 

\author{Michel Hickel and Micka\"{e}l Matusinski}


\begin{document}

\begin{abstract}
 We deal with the algebraicity of an iterated Puiseux series in several variables in terms of the properties of its coefficients. Our aim is to generalize to several variables the results from \cite{hickel-matu:puiseux-alg}. We show that the algebraicity of such a series for given bounded degrees is determined by a finite number of explicit universal polynomial formulas. Conversely, given a vanishing polynomial, there is a closed-form formula for the coefficients of the series in terms of the coefficients of the polynomial and of a bounded initial part of the series. 
\end{abstract}

\maketitle

\section{Introduction.}
Let $K$ be a field of characteristic zero and $\overline{K}$ its algebraic closure. Let $\underline{x}:=(x_1,\ldots,x_r)$ be an $r$-tuple of indeterminates where $r\geq 2$. Let $K[\underline{x}]$ and $K[[\underline{x}]]$ denote repectively the domain of polynomials and of formal power series in $r$ variables with coefficients in $K$, and their fraction fields $K(\underline{x})$ and $K((\underline{x}))$. Both fields embed naturally into\\ $K((x_r))((x_{r-1}))\cdots((x_1))$.
By iteration of the classical Newton-Puiseux theorem (see e.g. \cite[Theorem 3.1]{walker_alg-curves} and \cite[p. 314, Proposition]{rib-vdd_ratio-funct-field}), one can derive a description of an algebraic closure of $K((x_r))((x_{r-1}))\cdots((x_1))$ in terms of iterated fractional power series (see \cite[Theorem 3]{rayner_puiseux-multivar}\cite[p.151]{sathaye:newt-puiseux-exp_abh-moh-semigr}):
\begin{theorem*}
The following field, where $L$ ranges over the finite extensions of $K$ in $\overline{K}$: $$\mathcal{L}_r:= \displaystyle\varinjlim_{p\in\mathbb{N}^*} \displaystyle\varinjlim_L L((x_r^{1/p}))((x_{r-1}^{1/p}))\cdots ((x_1^{1/p}))$$  is the algebraic closure of $K((x_r))((x_{r-1}))\cdots((x_1))$. 
\end{theorem*}
Within this framework, there are several results concerning those iterated fractional power series which are solutions  of polynomial equations with coefficients either in $K(\underline{x})$ or $K((\underline{x}))$. More precisely, the authors provide necessary constraints on the supports of such series (see \cite[Theorem 3.16]{mcdonald_puiseux-multivar}, \cite[ Th\'eor\`eme 2]{gonzalez-perez_singul-quasi-ord}, \cite[Theorem 13]{soto-vicente:polyhedral-cones}  \cite[Theorem 1]{aroca-ilardi:puiseux-multivar}, \cite[Theorem 1]{soto-vicente_puiseux-multivar}). One can deduce from these results that:
\begin{theorem*}
The following field, where $L$ ranges over the finite extensions of $K$ in $\overline{K}$:
$$\mathcal{K}_r\ :=\ \displaystyle \varinjlim_{(p,\underline{q})\in \mathbb{N}^*\times\mathbb{N}^{r-1}} \displaystyle\varinjlim_L\ \ 
L\left(\left(\,\left( \displaystyle\frac{x_1}{x_2^{q_1}}\right)^{1/p},\ldots,  \left( \displaystyle\frac{x_{r-1}}{x_r^{q_{r-1}}}\right)^{1/p} ,x_r^{1/p}\right)\right)$$
is an algebraically closed extension of $K(\underline{x})$ and $K((\underline{x}))$ in $\mathcal{L}_r$.
\end{theorem*}

In fact, $\mathcal{K}_r$ can be described as:
$$ \displaystyle\varinjlim_{(p,\underline{q})\in \mathbb{N}^*\times\mathbb{N}^{r-1}}  \displaystyle\varinjlim_L
 \displaystyle\varinjlim_{\underline{n}\in \mathbb{Z}^r}\ \  \left(\frac{x_1}{x_2^{q_1}}\right)^{n_1/p}\cdots \left(\frac{x_{r-1}}{x_r^{q_{r-1}}}\right)^{n_{r-1}/p} x_r^{n_r/p}.L\left[\left[\left(\frac{x_1}{x_2^{q_1}}\right)^{1/p},\ldots, \left(\frac{x_{r-1}}{x_r^{q_{r-1}}}\right)^{1/p} ,x_r^{1/p}\right]\right],$$
(see for instance Lemma \ref{lemme:Kr}).

Let us call the elements of $\mathcal{K}_r$ \emph{rational polyhedral Puiseux series}  (since the support with respect to the variables $x_i$'s of such a series is included in the translation of some rational convex polyhedral cone). We are interested in those rational polyhedral Puiseux series that are algebraic, say the rational polyhedral Puiseux series which verify a polynomial equation $P(\underline{x},y)=0$ with coefficients which are themselves polynomials in $\underline{x}$:  $P(\underline{x},y)\in K[\underline{x}][y]\setminus\{0\}$. More precisely, we resume our previous work on algebraic Puiseux series in one variable \cite{hickel-matu:puiseux-alg}, by dealing with the following analogous questions:\\
\noindent$\bullet$ \textbf{Reconstruction of a vanishing polynomial for a given algebraic rational polyhedral Puiseux series.} The algebraicity of a rational polyhedral Puiseux series can be encoded by the vanishing of certain determinants derived from the coefficients of the series. These determinants are a straightforward generalization of what we called the Wilczynski polynomials in the one-variable case in \cite{hickel-matu:puiseux-alg}. We extend this approach by showing how to reconstruct the coefficients of a vanishing polynomial by means of some  of these (generalized) Wylczynski polynomials (see Section \ref{section:Wilc}). More precisely, we show that, for  given bounded degrees, there are finitely many universal polynomial formulas allowing to check the algebraicity of a series and to perform this reconstruction (see Theorem \ref{theo:wilc}). The results of this Section \ref{section:Wilc} hold for $K$ of arbitrary characteristic.\\
\noindent$\bullet$ \textbf{Description of the  coefficients of an algebraic rational polyhedral Puiseux series in terms of the coefficients of a vanishing polynomial.} For $y_0$ a  rational polyhedral Puiseux series solution of a given nonzero polynomial equation $P(\underline{x},y)=0$, our aim consists in determining a closed-form expression of the coefficients of $y_0$ in terms of the coefficients of $P$ and the coefficients of an initial part $z_{\underline{k}}$ of $y_0$ of  controled length $\underline{k}$.  In this direction, we prove a singular generalization of the multivariate version (see \cite[Theorems 3.5 and 3.6]{sokal:implicit-function}) of Flajolet-Soria formula \cite{flajolet-soria:coeff-alg-series} to the case of a series satisfying a strongly reduced Henselian equation in the sense of Section \ref{section:flajo-soria}. Then, we show  that the remaining part $y_0-z_{\underline{k}}$ of $y_0$ satisfies a strongly reduced Henselian equation canonically derived from $P$ (Theorem \ref{theo:FS}), and deduce the closed form expression (Corollary \ref{coro:FS}).

As a corollary of the latter result and of Theorem \ref{theo:wilc}, we obtain that, for given bounded degrees, there is a finite number of universal families of rational fractions such that, for any such $y_0$, the coefficients of the remaining part of $y_0-z_{\underline{k}}$ can be computed as the evaluation of such a family at the coefficients of $z_{\underline{k}}$ (Corollary \ref{coro:param-ratio}). As a direct consequence, we derive a proof of the multivariate version of Eisenstein Theorem due to K. V. Safonov \cite[Theorem 5]{safonov:algebraic-power-series} (see Corollary \ref{coro:eisenstein}).


\section{Preliminaries}\label{section:preliminaries} 

Let us denote $\mathbb{N}:=\mathbb{Z}_{\geq 0}$ and $\mathbb{N}^*:=\mathbb{N}\setminus\{0\}=\mathbb{Z}_{>0}$. 
For any set $\mathcal{E}$, we  write $|\mathcal{E}|:=\mathrm{Card}(\mathcal{E})$. We denote systematically the vectors as underlined letters, e.g. $\underline{x}:=(x_1,\ldots,x_r)$, $\underline{n}:=(n_1,\ldots,n_r)$, and in particular $\underline{0}:=(0,\ldots,0)$. Moreover, $\underline{x}^{\underline{n}}:=x_1^{n_1}\cdots x_r^{n_r}$. The floor function will be written $\lfloor q \rfloor$ for $q\in\mathbb{Q}$.

\begin{notation}\label{nota:FS}
For any vector of nonnegative  integers $\underline{M}=\displaystyle\left(m_{\underline{i},j}\right)_{\underline{i},j}$ and any vector of scalars $\underline{A}=\displaystyle\left(a_{\underline{i},j}\right)_{\underline{i},j}$ indexed by finitely many $\underline{i}\in\mathbb{N}^r$ and $j\in\mathbb{N}$, we set:
\begin{itemize}
\item $\underline{M}!:=\displaystyle\prod_{\underline{i},j}m_{\underline{i},j}!$;
\item $\underline{A}^{\underline{M}}:=\displaystyle\prod_{\underline{i},j}a_{\underline{i},j}^{m_{\underline{i},j}}$;
\item  $\ |\underline{M}|:=\displaystyle\sum_{\underline{i},j}m_{\underline{i},j}$, $\ ||\underline{M}||:= \displaystyle\sum_{\underline{i},j}m_{\underline{i},j}\, j$ and $\ G(\underline{M}) := \displaystyle\sum_{\underline{i},j}m_{\underline{i},j}\, \underline{i}$.
\end{itemize}
In the case where $\underline{k}=(k_0,\ldots,k_l)$, we set 
$  \|\underline{k}\| :=\displaystyle\sum_{j=0}^l{k_j}\,j$. In the case where $\underline{k}=(k_{\underline{i}})_{\underline{i}\in \Delta}$ where $\Delta$ is a finite subset of $\mathbb{N}^r$, we set $G(\underline{k}):=\displaystyle\sum_{\underline{i}\in \Delta}{k_i}\,\underline{i}$.
\end{notation}

We will consider the following orders on tuples in $\mathbb{Z}^r$:
\begin{description}
\item[The lexicographic order] $\underline{n} \leq_{\textrm{lex}} \underline{m}$ $:\Leftrightarrow$ $n_1<m_1$ or $(n_1=m_1\ \textrm{ and } n_2<m_2)$ or $\cdots$ or $(n_1=m_1,\ n_2=m_2,\ \ldots \ \textrm{ and } n_r<m_r)$. 
\item[The graded lexicographic order] $\underline{n} \leq_{\textrm{grlex}} \underline{m}$ $:\Leftrightarrow$ $|\underline{n }|<|\underline{m}|$ or $(|\underline{n }|=|\underline{m}|\ \textrm{ and } \underline{n} \leq_{\textrm{lex}} \underline{m})$.
\item[The product (partial) order] $\underline{n} \leq \underline{m}$  $:\Leftrightarrow$ $n_1\leq m_1 \ \textrm{ and } n_2\leq m_2\ \cdots \ \textrm{ and } n_r\leq m_r$.
\end{description}
 Note that we will apply also the lexicographic order on $\mathbb{Q}^r$. Similarly, one has the \textbf{anti-lexicographic order} denoted by $\leq_{\textrm{alex}}$.\\

To view the fields $K(\underline{x})$ and $K((\underline{x}))$ as embedded into $K((x_r))((x_{r-1}))\cdots((x_1))$ means that the rational fractions or formal meromorphic fractions can be represented as iterated formal Laurent series, i.e. Laurent series in $x_1$ whose coefficients are Laurent series in $x_2$, whose coefficients... etc. This corresponds to the following approach. As in \cite{rayner_puiseux-multivar, sathaye:newt-puiseux-exp_abh-moh-semigr}, we identify $K((x_r))((x_{r-1}))\cdots((x_1))$ with the field of generalized power series (in the sense of \cite{hahn:nichtarchim}; see also \cite{rib:series-fields-alg-closed})  with coefficients in $K$ and exponents in $\mathbb{Z}^r$ ordered lexicographically, usually denoted by $K\left(\left(X^{\mathbb{Z}^r}\right)\right)^{\mathrm{lex}}$. By definition, such a generalized series is a formal expression $s=\displaystyle\sum_{\underline{n}\in \mathbb{Z}^r}c_{\underline{n}}X^{\underline{n}}$ (say a map $\mathbb{Z}^r\rightarrow K$) whose support $\textrm{Supp}(s):=\{\underline{n}\in \mathbb{Z}^r\ |\ c_{\underline{n}}\neq 0\}$ is well-ordered. The field $K\left(\left(X^{\mathbb{Z}^r}\right)\right)^{\mathrm{lex}}$ comes naturally equipped with the following valuation of rank $r$:
$$\begin{array}{lccl}v:&K\left(\left(X^{\mathbb{Z}^r}\right)\right)^{\mathrm{lex}}&\rightarrow&(\mathbb{Z}^r\cup\{\infty\},\leq_{\textrm{lex}})\\
&s\neq 0 &\mapsto& \min(\textrm{support}(s))\\
&0&\mapsto & \infty
\end{array}$$
The identification of $K\left(\left(X^{\mathbb{Z}^r}\right)\right)$ and $K((x_r))((x_{r-1}))\cdots((x_1))$ reduces to the identification $$X^{(1,0,\ldots,0)}=x_1\ ,\ \ X^{(0,1,\ldots,0)}=x_2 \ ,\ \ \ldots\ ,\  \  X^{(0,\ldots,0,1)}=x_r.$$
Note also that this corresponds to the fact that the power series in the rings $K[\underline{x}]$ and $K[[\underline{x}]]$ are viewed as expanded along $(\mathbb{Z}^r,\leq_{\textrm{lex}})$. \\
Similarly, the field $\mathcal{L}_r$ is a union of fields of generalized series $L\left(\left(X^{(\mathbb{Z}^r)/p}\right)\right)^{\mathrm{lex}}$ and comes naturally equipped with the valuation of rank $r$:
$$\begin{array}{lccl}v:&\mathcal{L}_r&\rightarrow&(\mathbb{Q}^r\cup\{\infty\},\leq_{\textrm{lex}})\\
&s\neq 0&\mapsto& \min(\textrm{support}(s))\\
&0&\mapsto & \infty.
\end{array}$$

We will need another representation of the elements in $K(\underline{x})$ and $K((\underline{x}))$, via the embedding of these fields into the field $K\left(\left(X^{\mathbb{Z}^r}\right)\right)^{\mathrm{grlex}}$ with valuation:
$$\begin{array}{lccl}w:&K\left(\left(X^{\mathbb{Z}^r}\right)\right)^{\mathrm{grlex}}&\rightarrow&(\mathbb{Z}^r\cup\{\infty\},\leq_{\textrm{grlex}})\\
&s\neq 0&\mapsto& \min(\textrm{support}(s))\\
&0&\mapsto & \infty.
\end{array}$$
and the same identification:
$$X^{(1,0,\ldots,0)}=x_1\ ,\ \ X^{(0,1,\ldots,0)}=x_2 \ ,\ \ \ldots\ ,\  \  X^{(0,\ldots,0,1)}=x_r.$$
For a polynomial $P(y)=\displaystyle\sum_{j=0}^da_jY^j\in K\left(\left(X^{\mathbb{Z}^r}\right)\right)^{\mathrm{grlex}}[y]$, we denote:
$$ w(P(y)):=\min_{j=0,\ldots,d}\{w(a_j)\}.$$

We will also use the following notations to keep track of the variables used to write the monomials. Given a ring $R$, we denote by $ R((x_1^{\mathbb{Z}},\ldots,x_r^\mathbb{Z}))^{\textrm{lex}}$ and $ R((x_1^{\mathbb{Z}},\ldots,x_r^\mathbb{Z}))^{\textrm{grlex}}$ the corresponding rings of generalized series $\displaystyle\sum_{\underline{n}\in \mathbb{Z}^r}c_{\underline{n}} \underline{x}^{\underline{n}}$ with coefficients $c_{\underline{n}}$ in $R$. Accordingly, let us write $ R((x_1^{\mathbb{Z}},\ldots,x_r^\mathbb{Z}))^{\textrm{lex}}_{\textrm{Mod}}$ and $ R((x_1^{\mathbb{Z}},\ldots,x_r^\mathbb{Z}))^{\textrm{grlex}}_{\textrm{Mod}}$ the subrings of series whose actual exponents are all bounded by below by some constant for the product order. Note that these subrings are both isomorphic to the ring $\displaystyle\bigcup_{\underline{n}\in\mathbb{Z}^r}\underline{x}^{\underline{n}}R[[\underline{x}]] $.
 Let us write also $R((x_1^{\mathbb{Z}},\ldots,x_r^\mathbb{Z}))^{\textrm{lex}}_{\geq_{\textrm{lex}}\underline{0}}$ and $R((x_1^{\mathbb{Z}},\ldots,x_r^\mathbb{Z}))^{\textrm{grlex}}_{\geq_{\textrm{grlex}}\underline{0}}$ the subrings of series $y$ with $v(y)\geq_{\textrm{lex}}\underline{0}$, respectively $w(y)\geq_{\textrm{grlex}}\underline{0}$. 

\begin{lemma}[Monomialization Lemma]\label{lemme:monomialisation}
Let $f$ be non zero in $K[[u_1,\ldots,u_r]]$. There exists $s_1,\ldots,s_{r-1}\in \mathbb{N}$ such that, if we set
\begin{equation}\label{equ:transfo}
\left\{\begin{array}{lcl}
v_1&:=&\displaystyle\frac{u_1}{u_2^{s_1}}\\
&\vdots&\\
v_{r-1}&:=&\displaystyle\frac{u_{r-1}}{u_{r}^{s_{r-1}}}\\
v_r&:=&u_r
\end{array}\right.
\end{equation} 
then $f(u_1,\ldots,u_r)=\underline{v}^{\underline{\alpha}}g(v_1,\ldots,v_r)$ where $\underline{\alpha}\in\mathbb{N}^r$ and $g$ is an invertible element of $K[[v_1,\ldots,v_r]]$.
\end{lemma}
\begin{demo}

Let us write $f=\underline{u}^{\underline{\beta}}\,h$ where $\underline{\beta}=v(f)$ and  $v(h)=\underline{0}$ (where $v$ is the lexicographic valuation with respect to the variables $\underline{u}$). Note that $h$ can be written as $h=h_0+h_1$ where $h_0\in K((u_2,\ldots,u_r))_{\geq_{\textrm{lex}}\underline{0},\textrm{Mod}}$ with $v(h_0)=\underline{0}$, and $h_1\in u_1K[[u_1]]((u_2^{\mathbb{Z}},\ldots,u_r^\mathbb{Z}))_{\textrm{Mod}}$. Let $s_1$ be a positive integer such that:
$$ s_1\geq \sup\{1\ ;\ (1-m_2)/m_1,\ \underline{m}\in\textrm{supp}\ h_1\}.$$
Let $v_1:=u_1/u_2^{s_1}$. For every monomial in $h_1$, one has $u_1^{m_1}u_2^{m_2}\ldots u_r^{m_r}=v_1^{m_1}u_2^{m_2+s_1m_1}\ldots u_r^{m_r}$. Hence, $m_2+s_1m_1\geq 1$ by definition of $s_1$. So $(m_2+s_1m_1,\ldots,m_r)>_{\textrm{lex}}0$, meaning that $h_1\in K[[v_1]]((u_2^{\mathbb{Z}},\ldots,u_r^\mathbb{Z}))_{\geq_{\textrm{lex}}\underline{0},\textrm{Mod}}$ and $v(h_1)>_{\textrm{lex}}\underline{0}$ (where $v$ is now the lexicographic valuation with respect to the variables $(v_1,u_2,\ldots,u_r)$). So $h\in K[[v_1]]((u_2^{\mathbb{Z}},\ldots,u_r^\mathbb{Z}))_{\geq_{\textrm{lex}}\underline{0},\textrm{Mod}}$ and $v(h)=\underline{0}$. \\
Suppose now that we have obtained $h\in  K[[v_1,\ldots,v_p]]((u_{p+1}^{\mathbb{Z}},\ldots,u_r^\mathbb{Z}))_{\geq_{\textrm{lex}}\underline{0},\textrm{Mod}}$ and $v(h)=\underline{0}$ (where $v$ is now the lexicographic valuation with respect to the variables\\
 $(v_1,\ldots,v_p,u_{p+1},\ldots,u_r)$). As before, there exists a positive integer $s_{p+1}$ such that, if we set $v_{p+1}:=u_{p+1}/u_{p+2}^{s_{p+1}}$, then $h\in K[[v_1,\ldots,v_{p+1}]]((u_{p+2}^{\mathbb{Z}},\ldots,u_r^\mathbb{Z}))_{\geq_{\textrm{lex}}\underline{0},\textrm{Mod}}$ and $v(h)=\underline{0}$ (where $v$ is now the lexicographic valuation with respect to the variables $(v_1,\ldots,v_{p+1},u_{p+2},\ldots,u_r)$).\\
By iteration of this process, we obtain $h \in K[[v_1,\ldots,v_{r-1}]]((u_r^\mathbb{Z}))_{\geq_{\textrm{lex}}\underline{0},\textrm{Mod}}$  and $v(h)=\underline{0}$ (where $v$ is now the lexicographic valuation with respect to the variables $(v_1,\ldots,v_{r-1}, u_r)$), which means that $h\in K[[v_1,\ldots,v_{r-1},u_r]]$ with $h$ invertible. Since $\underline{u}^{\underline{\beta}}=\underline{v}^{\underline{\alpha}}$ for some $\underline{\alpha}\in\N^r$, the lemma follows.
\end{demo}

We will use the following particular representation of $\mathcal{K}_r$.
\begin{lemma}\label{lemme:Kr}
$ \ \ \ \ \mathcal{K}_r=\\
 \displaystyle\varinjlim_{(p,\underline{q})\in \mathbb{N}^*\times\mathbb{N}^{r-1}}  \displaystyle\varinjlim_L
 \displaystyle\varinjlim_{\underline{n}\in \mathbb{Z}^r}\ \  \left(\frac{x_1}{x_2^{q_1}}\right)^{n_1/p}\cdots \left(\frac{x_{r-1}}{x_r^{q_{r-1}}}\right)^{n_{r-1}/p} x_r^{n_r/p}.L\left[\left[\left(\frac{x_1}{x_2^{q_1}}\right)^{1/p},\ldots, \left(\frac{x_{r-1}}{x_r^{q_{r-1}}}\right)^{1/p} ,x_r^{1/p}\right]\right]^*$\\

\noindent where $L\left[\left[\displaystyle\left(\frac{x_1}{x_2^{q_1}}\right)^{1/p},\ldots, \left(\displaystyle\frac{x_{r-1}}{x_r^{q_{r-1}}}\right)^{1/p} ,x_r^{1/p}\right]\right]^*$ denotes the group of invertible elements in\\ $L\left[\left[\left(\displaystyle\frac{x_1}{x_2^{q_1}}\right)^{1/p},\ldots, \left(\displaystyle\frac{x_{r-1}}{x_r^{q_{r-1}}}\right)^{1/p} ,x_r^{1/p}\right]\right]$.
\end{lemma}
\begin{demo}
Let $y_0\in \mathcal{K}_r$. There exist $(p,\underline{q})\in \mathbb{N}^*\times\mathbb{N}^{r-1}$ and $L$ with $[L:K]<+\infty$ such that $y_0\in L\left(\left(\left(\displaystyle\frac{x_1}{x_2^{q_1}}\right)^{1/p},\ldots, \left(\displaystyle\frac{x_{r-1}}{x_r^{q_{r-1}}}\right)^{1/p} ,x_r^{1/p}\right)\right)$. Let us denote $\underline{u}=(u_1,\ldots,u_r):=\left(\left(\displaystyle\frac{x_1}{x_2^{q_1}}\right)^{1/p},\ldots, \left(\displaystyle\frac{x_{r-1}}{x_r^{q_{r-1}}}\right)^{1/p} ,x_r^{1/p}\right)$. So $y_0=\displaystyle\frac{f}{g}$ for some $f,g\in L[[\underline{u}]]$. 
By the preceding lemma, we can monomialize the product $f.g$, so $f$ and $g$ simultaneously, by a suitable transformation (\ref{equ:transfo}). Note that this transformation maps $L\left(\left(\left(\displaystyle\frac{x_1}{x_2^{q_1}}\right)^{1/p},\ldots, \left(\displaystyle\frac{x_{r-1}}{x_r^{q_{r-1}}}\right)^{1/p} ,x_r^{1/p}\right)\right)$ into some $L\left(\left(\left(\displaystyle\frac{x_1}{x_2^{t_1}}\right)^{1/p},\ldots, \left(\displaystyle\frac{x_{r-1}}{x_r^{t_{r-1}}}\right)^{1/p} ,x_r^{1/p}\right)\right)$.
\end{demo}

Let  $\tilde{y}_0\in \mathcal{K}_r$ be a non zero rational polyhedral Puiseux series. By Lemma \ref{lemme:Kr} there are $(p,\underline{q})\in \mathbb{N}^*\times\mathbb{N}^{r-1}$ such that, if we set:
\begin{equation}\label{equ:eclt1} \left(u_1,\ldots,u_{r-1},u_r\right):=\left(\left(\frac{x_1}{x_2^{q_1}}\right)^{1/p},\ldots, \left(\frac{x_{r-1}}{x_r^{q_{r-1}}}\right)^{1/p} ,x_r^{1/p}\right),\end{equation}
 then we can rewrite $\tilde{y}_0 =\displaystyle\sum_{\underline{n}\geq \underline{n}^0} \tilde{c}_{\underline{n}}\underline{u}^{\underline{n}},\ \tilde{c}_{\underline{n}^0}\neq 0$. Let us denote $c_{\underline{n}}:=\tilde{c}_{\underline{n}+\underline{n}^0-(0,\ldots,0,1)}$, and: 
  $$\tilde{y}_0=\underline{u}^{\underline{n}^0} u_r^{-1}\displaystyle\sum_{\underline{n}\geq (0,\ldots,0,1)} c_{\underline{n}}\underline{u}^{\underline{n}}=\underline{u}^{\underline{n}^0} u_r^{-1}y_0\ \ \textrm{with}\  c_{(0,\ldots,0,1)}\neq 0.$$ 
By the change of variable (\ref{equ:eclt1}), we have:
$$ x_k=u_k^pu_{k+1}^{pq_{k}}u_{k+2}^{pq_{k}q_{k+1}}\cdots u_{r}^{pq_{k}q_{k+1}\cdots q_{r-1}},\ \ \ \ \ \ \ \  k=1,\ldots,r$$
The series $\tilde{y}_0$ is a root of a polynomial $\tilde{P}(\underline{x},y)=\displaystyle\sum_{\underline{i},j} \tilde{a}_{\underline{i},j}\underline{x}^{\underline{i}}y^j$ of degree $d_y$ in $y$ if and only if the series $y_0=\displaystyle\sum_{\underline{n}\geq (0,\ldots,0,1)} c_{\underline{n}}\underline{u}^{\underline{n}}$ is a root of  $$\underline{u}^{\underline{m}}\tilde{P}\left(  u_1^pu_{2}^{pq_{1}}\cdots u_{r}^{pq_{1}q_{2}\cdots q_{r-1}}\    ,\ \ldots\ ,\ u_r^{p} ,\ \underline{u}^{\underline{n}^0}u_r^{-1}y\right),$$
 the latter being a polynomial for $\underline{m}$ such that 
\begin{equation}\label{equ:m}
m_k=\max\left\{0\, ;\, -n_k^0d_y\right\},\ k=1,\ldots,r-1,\ \ \ \textrm{ and }\ \ \  m_r=\max\left\{0\, ;\, \left(1-n_r^0\right)d_y\right\}.\end{equation}
The existence of a nonzero polynomial $\tilde{P}$ cancelling $\tilde{y}_0$ is equivalent to the one of a polynomial $P(\underline{u},y)=\displaystyle\sum_{\underline{i},j} a_{\underline{i},j}\underline{u}^{\underline{i}}y^j$ cancelling $y_0$, but  with constraints on the support of $P$. Let us make these constraints explicit in the case $r=2$:
$$\begin{array}{lcl}
P(u_1,u_2,y)&=&u_1^{m_1}u_2^{m_2}\tilde{P}\left(  u_1^pu_{2}^{pq_{1}},\ u_2^{p} ,\  u_1^{n_1^0}u_2^{n_2^0-1}y\right)\\
&=& \displaystyle\sum_{\underline{i},j} \tilde{a}_{\underline{i},j}u_1^{pi_1+jn_1^0+m_1}u_2^{p(i_2+q_1i_1)+j(n_2^0-1)+m_2}y^j\\
&=&\displaystyle\sum_{\underline{k},j}a_{\underline{k},j}u_1^{k_1}u_2^{k_2}y^j
\end{array}$$
The necessary conditions for $(\underline{k},j)$ to belong to the support of $P$ are:\\

$(k_1,k_2)=$\\

\noindent$\left\{\begin{array}{cl}
\left(jn_1^0 \mod p\ ,\ q_1k_1+j(n_2^0-1-q_1n_1^0)\mod p\right)  & \textrm{ if } n_1^0\geq 0\textrm{ and } n_2^0\geq1\\
\left(jn_1^0 \mod p\ ,\ q_1k_1+j(n_2^0-1-q_1n_1^0)-d_y(n_2^0-1)\mod p\right)  & \textrm{ if } n_1^0\geq 0\textrm{ and } n_2^0<1\\
\left((j-d_y)n_1^0 \mod p\ ,\ q_1k_1+j(n_2^0-1)-(j-d_y)q_1n_1^0\mod p\right)  & \textrm{ if } n_1^0< 0\textrm{ and } n_2^0\geq1\\
\left((j-d_y)n_1^0 \mod p\ ,\ q_1k_1+(j-d_y)(n_2^0-1-q_1n_1^0)\mod p\right) & \textrm{ if } n_1^0<0\textrm{ and } n_2^0<1
\end{array}\right.$\\

In the general case with $r$ variables, we claim that one can derive similar constraints on the support of a vanishing polynomial $P$ for $y_0$, depending only on $d_y$, $p,q_1,\ldots,q_{r-1}$ and $\underline{n}^0$. The algebraicity of $\tilde{y}_0$ is equivalent to that of $y_0$ but \textit{ with such constraints} on the support of the vanishing polynomial. This leads us to the following definition:

\begin{defi}\label{defi:alg-relative}
Let $\mathcal{F}$ and $\mathcal{G}$ be two strictly increasing finite sequences of pairs $(\underline{i},j)\in\left(\mathbb{N}^r\times\mathbb{N}\right)$ ordered anti-lexicographically: 
$$(\underline{i}_1,j_1) \leq_{\textrm{alex}} (\underline{i}_2,j_2)\Leftrightarrow  j_1 < j_2\textrm{ or } (j_1 = j_2\ \textrm{and}\ \underline{i}_1 \leq_{\textrm{grlex}} \underline{i}_2).$$
We suppose additionally that  $\mathcal{F}\geq_{\textrm{alex}} \left(\underline{0},1\right)>_{\textrm{alex}}\mathcal{G}>_{\textrm{alex}}\left(\underline{0},0\right)$ (thus the elements of $\mathcal{G}$ are  ordered pairs of the form $(\underline{i},0)$, $|\underline{i}|> 0$, and those of  $\mathcal{F}$ are of the form  $(\underline{i},j),\ j\geq 1$). We say that a series $y_0=\displaystyle\sum_{\underline{n}\geq_{\mathrm{grlex}} (0,\ldots,0,1)} c_{\underline{n}}\underline{x}^{\underline{n}}\in K[[\underline{x}]]$, $c_{(0,\ldots,0,1)}\neq 0$, is \textbf{algebraic relatively to $(\mathcal{F},\mathcal{G})$} if there exists a polynomial $P(\underline{x},y)=\displaystyle\sum_{(\underline{i},j)\in\mathcal{F}\cup\mathcal{G}} a_{\underline{i},j}\underline{x}^{\underline{i}}y^j\in K[\underline{x},y]\setminus\{0\}$ such that $P(\underline{x},y_0)=0$.
\end{defi}

\begin{ex}\label{ex:eclt}
For $r=2$, let us consider the following general equation of degrees $d_x=1$ in $\underline{x}$ and $d_y=2$ in $y$: 
$$\tilde{P}=\tilde{a}_{0,0,0}+\tilde{a}_{0,1,0}x_2+\tilde{a}_{1,0,0}x_1+\left(\tilde{a}_{0,0,1}+\tilde{a}_{0,1,1}x_2+\tilde{a}_{1,0,1}x_1\right)y+ \left(\tilde{a}_{0,0,2}+\tilde{a}_{0,1,2}x_2+\tilde{a}_{1,0,2}x_1\right)y^2.$$
 For instance, if $\tilde{a}_{0,0,0}=\tilde{a}_{0,0,1}=0$ and $\tilde{a}_{0,1,0}.\tilde{a}_{0,0,2}\neq 0$  then one can expand the two  solutions of this equation in ${x_2}^{1/2}\cdot L\left[\left[\left(\displaystyle\frac{x_1}{x_2}\right)^{1/2}, {x_2}^{1/2}\right]\right]^*$ for $L=K\left[\sqrt{-\tilde{a}_{0,1,0}/\tilde{a}_{0,0,2}}\right]$. With $\tilde{c}_{0,1}=\sqrt{-\tilde{a}_{0,1,0}/\tilde{a}_{0,0,2}}$, the solutions are:
$$\tilde{y}_0=\tilde{c}_{0,1}{x_2}^{1/2}+\cdots\ \textrm{ and } \ \tilde{\tilde{y}}_0=-\tilde{c}_{0,1}{x_2}^{1/2}+\cdots.$$
Note that for both series we have $\underline{n}^0=(0,1)$ and therefore $\underline{m}=(0,0)$ for $\underline{m}$ as in (\ref{equ:m}).
By application of the following change of variables:
$$ \left(u_1,u_2\right):=\left(\left(\displaystyle\frac{x_1}{x_2}\right)^{1/2},{x_2}^{1/2}\right)\ \Leftrightarrow\ \left(x_1,x_2\right)=\left({u_1}^2{u_2}^2\, , {u_2}^2\right),$$
we derive from $\tilde{P}$ the following equation:
$$P=a_{0,2,0}{u_2}^2+a_{2,2,0}{u_1}^2{u_2}^2+\left(a_{0,2,1}{u_2}^2+ a_{2,2,1}{u_1}^2{u_2}^2\right)y+ \left(a_{0,0,2}+a_{0,2,2}{u_2}^2+a_{2,2,2}{u_1}^2{u_2}^2\right)y^2.$$
that has its solutions in $u_2\cdot L[[u_1,u_2]]^*$, with $c_{0,1}=\tilde{c}_{0,1}$:
$$y_0=c_{0,1}u_2+\cdots\ \textrm{ and } \ \hat{y}_0=-c_{0,1}u_2+\cdots.$$
 The corresponding sets $\mathcal{F}$ and $\mathcal{G}$ that contain the support of $P$ are:
$$\mathcal{F}=\left\{(0,2,1),(2,2,1),(0,0,2),(0,2,2),(2,2,2)\right\}\ \textrm{ and }\ \mathcal{G}=\left\{(0,2,0),(2,2,0)\right\}.$$
\end{ex}

We will also need the following arithmetical lemma in Section \ref{section:hensel} and  at the end of Section \ref{section:flajo-soria}:

\begin{lemma}\label{lemma:arithm}
Let $m\in\mathbb{N}^*$ and $\underline{k}=(k_0,\ldots,k_d)\in\mathbb{N}^{d+1}$ for $d\in\mathbb{N}^*$ such that $|\underline{k}|=m$ and $\|\underline{k}\|=m-1$. Then:
$$\displaystyle\frac{1}{m}\cdot\displaystyle\frac{m!}{\underline{k}!}\in \mathbb{N}.$$
\end{lemma}
\begin{demo}
For any prime number $p$, $\nu_p$ denotes the $p$-adic valuation on $\mathbb{Q}$. Let us show that for any $p$:
$$\nu_p\left(\displaystyle\frac{1}{m}\cdot\displaystyle\frac{m!}{\underline{k}!}\right)= \nu_p\left(\displaystyle\frac{(m-1)!}{\underline{k}!}\right)\geq 0.$$
By a classical result of A.-M. Legendre \cite{legendre:theorie-nbres} or \cite[Lemma 4]{singmaster:binomial-multinomial}, for every $n\in\mathbb{N}$ and any prime $p$, one has:
$$ \nu_p(n!)=\displaystyle\sum_{i\geq 1} \left\lfloor \displaystyle\frac{n}{p^i}\right\rfloor.$$
For any prime $p$ and $i\in\mathbb{N}^*$, let  us write the Euclidian divisions:
$$\left\{\begin{array}{ccll}
m-1&=&q_i\cdot p^i+r_i, & 0\leq r_i< p^i;\\
k_j&=& q_{i,j}\cdot p^i+r_{i,j},& 0\leq r_{i,j}< p^i;\\
\displaystyle\sum_{j=0,\ldots,d}r_{i,j}&=& \kappa_i \cdot p^i + \rho_i,& 0\leq \rho_i < p^i.
\end{array}\right.$$
One has that:
$$\begin{array}{lcl}
 \nu_p\left(\displaystyle\frac{(m-1)!}{\underline{k}!}\right)&=& \nu_p((m-1)!)-\displaystyle\nu_p(\underline{k}!)\\
 &=& \displaystyle\sum_{i\geq 1} \left( \left\lfloor \displaystyle\frac{m-1}{p^i}\right\rfloor-\displaystyle\sum_{j=0,\ldots,d}\left\lfloor \displaystyle\frac{k_j}{p^i}\right\rfloor\right)\\
 &=& \displaystyle\sum_{i\geq 1} \left(q_i-\displaystyle\sum_{j=0,\ldots,d}q_{i,j}\right).
\end{array}$$
For any $p$ and $i$, there are two cases. If $r_i< p^i-1$, then:
$$ \left\lfloor \displaystyle\frac{m}{p^i}\right\rfloor=q_i.$$
Since $\displaystyle\sum_{j=0,\ldots,d}k_j=m$, we obtain:
$$ \displaystyle\sum_{j=0,\ldots,d}q_{i,j}+\kappa_i=\left\lfloor \displaystyle\frac{m}{p^i}\right\rfloor=q_i.$$
Hence, $q_i\geq \displaystyle\sum_{j=0,\ldots,d}q_{i,j}$.\\
If $r_i= p^i-1$, then $m=(q_i+1)p^i$. So:
$$\left(\displaystyle\sum_{j=0,\ldots,d}q_{i,j}+\kappa_i\right)p^i+\rho_i=(q_i+1)p^i.$$
Therefore, $\rho_i=0$ and $\displaystyle\sum_{j=0,\ldots,d}q_{i,j}+\kappa_i=q_i+1$. Either $\kappa_i\geq 1$, so $\displaystyle\sum_{j=0,\ldots,d}q_{i,j}\leq q_i$. Or $\kappa_i=0$, so each $r_{i,j}=0$: $p^i$ divides $k_j$ for any $j$. But this would imply that $p^i$ divides $m-1=\displaystyle\sum_{j=0,\ldots,d}j\,k_j$: a contradiction with the fact that $m$ and $m-1$ are coprime.\\
We obtain that $q_i\geq \displaystyle\sum_{j=0,\ldots,d}q_{i,j}$ for any $i$ and $p$, which gives the desired result.
\end{demo}

\section{Characterizing the algebraicity of a formal multivariate power series}\label{section:Wilc}

Here we resume and extend to the multivariate case the remarks from  \cite{wilczynski:alg-power-series}. Note that in the present section, the field $K$ can be of any  characteristic. 

The purpose of the following discussion is to translate the vanishing of a polynomial $P$ at a formal series $y_0$ in terms of the vanishing of minors of an infinite matrix. As we have seen in the previous section, one can always assume that $y_0=\displaystyle\sum_{\underline{n}\in \mathbb{N}^r} c_{\underline{n}}\underline{u}^{\underline{n}}$ is such that $c_{\underline{0}}=0$, $c_{(0,\ldots,0,1)}\neq 0$. (In fact we could even assume that  $\underline{n}\geq (0,\ldots,0,1)$ but we will not use this restriction).\\

Let us consider a series $Y_0=\displaystyle\sum_{\underline{n}\geq_{\mathrm{grlex}} (0,\ldots,0,1)} C_{\underline{n}}x^{\underline{n}}\in K[(C_{\underline{n}})_{\underline{n}\in\mathbb{N}^r}][[\underline{x}]]$ where $\underline{x}$ and the $C_{\underline{n}}$'s are variables. We denote the multinomial expansion of the $j$th power ${Y_0}^j$  of $Y_0$ by:
$${Y_0}^j=\displaystyle\sum_{\underline{n}\geq_{\mathrm{grlex}} (0,\ldots,0,1)} C_{\underline{n}}^{(j)}\underline{x}^{\underline{n}}.$$
where $C_{\underline{n}}^{(j)}\in  K[(C_{\underline{n}})_{\underline{n}\in\mathbb{N}^r}]$.  
Of course, one has that $C_{\underline{n}}^{(j)}=0$ for $|\underline{n}|<j$ and $C_{(0,\ldots,0,j)}^{(j)}={C_{(0,\ldots,0,1)}}^j\neq 0$. For $j=0$, we set ${Y_0}^0:=1$. We remark that for any $\underline{n}$ and any $j\leq |\underline{n}|$, $C_{\underline{n}}^{(j)}$ is a homogeneous polynomial of degree $j$ in the $C_{\underline{m}}$'s for $\underline{m}\in\mathbb{N}^r$, $\underline{m}\leq_{\mathrm{grlex}} \underline{n}-(j-1)(0,\ldots,0,1)$, with coefficients in $\mathbb{N}^*$ (indeed, each monomial occurring in $C_{\underline{n}}^{(j)}$ is of the form $C_{\underline{i}_1}\ldots C_{\underline{i}_j}$ with $\underline{i}_k\geq_{\mathrm{grlex}} (0,\ldots,0,1)$ and $\underline{i}_1+\cdots+\underline{i}_j=\underline{n}$, so $\underline{i}_k\leq_{\mathrm{grlex}} \underline{n}-(j-1)(0,\ldots,0,1)$ for any $k$).

Now suppose we are given a series $y_0=\displaystyle\sum_{\underline{n}\geq_{\mathrm{grlex}} (0,\ldots,0,1)} c_{\underline{n}}\underline{x}^{\underline{n}}\in K[[\underline{x}]]$ with  $c_{(0,\ldots,0,1)}\neq 0$. For any $j\in\mathbb{N}$, we denote the multinomial expansion of ${y_0}^j$ by:
$${y_0}^j=\displaystyle\sum_{\underline{n}\geq_{\mathrm{grlex}} (0,\ldots,0,1)} c_{\underline{n}}^{(j)}\underline{x}^{\underline{n}}.$$
So, $c_{\underline{n}}^{(j)}=C_{\underline{n}}^{(j)}(c_{(0,\ldots,0,1)},\ldots,c_{\underline{n}-(j-1)(0,\ldots,0,1)})$.

\begin{defi}\label{defi:mat_Wilc}
\begin{enumerate}
    \item 
Given an ordered pair $(\underline{i},j)\in\mathbb{N}^r\times\mathbb{N}$, we call \textbf{Wilczynski vector} $V_{\underline{i},j}$ the infinite vector with components $\gamma_{\underline{n}}^{\underline{i},j}$ with $\underline{n}\in\mathbb{N}^r$ ordered with $\leq_{\textrm{grlex}}$:\\
- if $j\geq 1$:
$$V_{i,j}:=\left(\gamma_{\underline{n}}^{\underline{i},j}\right)_{\underline{n}\in\mathbb{N}^r} \textrm{ with } \gamma_{\underline{n}}^{\underline{i},j}=\left\{\begin{array}{ll}
=c_{\underline{n}-\underline{i}}^{(j)}& \textrm{ if }\underline{n}\geq\underline{i}\\
=0 & \textrm{ if }\underline{n}<\underline{i}
\end{array}\right.$$
- otherwise: 1  in the $\underline{i}$th position and 0 for the other coefficients,
$$V_{\underline{i},0}:=(0,\ldots,1,0,0,\ldots,0,\ldots).$$
\item Let $\mathcal{F}$ and $\mathcal{G}$ be two sequences as in Definition \ref{defi:alg-relative}. We associate to $\mathcal{F}$ and $\mathcal{G}$ the \textbf{ (infinite)  Wilczynski matrix } whose columns are the corresponding vectors $V_{\underline{i},j}$:
$$M_{\mathcal{F},\mathcal{G}}:=(V_{\underline{i},j})_{(\underline{i},j)\in\mathcal{F}\cup \mathcal{G}}\, ,$$
$\mathcal{F}\cup\mathcal{G}$ being ordered anti-lexicographically.\\ 
We define also the \textbf{reduced Wilczynski matrix}, $M_{\mathcal{F},\mathcal{G}}^{red}$: it is the matrix obtained from $M_{\mathcal{F},\mathcal{G}}$ by removing the columns indexed in $\mathcal{G}$,  and also removing the corresponding rows (suppress the $\underline{i}$th row for any $(\underline{i},0)\in\mathcal{G}$). This amounts exactly to remove the rows containing the coefficient 1 for some Wilczynski vector indexed in $\mathcal{G}$. 
\end{enumerate}
\end{defi}

\begin{lemma}[generalized Wilczynski]\label{lemme:wilcz}
The series $y_0$ is algebraic relatively to $(\mathcal{F},\mathcal{G})$ if and only if all the minors of order $|\mathcal{F}\cup\mathcal{G}|$ of the  Wilczynski matrix $M_{\mathcal{F},\mathcal{G}}$ vanish, or also if and only if all the minors of order  $|\mathcal{F}|$ of the reduced Wilczynski matrix $M_{\mathcal{F},\mathcal{G}}^{red}$ vanish.
\end{lemma}
\begin{demo}
Given a   nontrivial polynomial $P(\underline{x},y)=\displaystyle\sum_{(\underline{i},j) \in\mathcal{F}\cup\mathcal{G}}a_{\underline{i},j}\underline{x}^{\underline{i}}y^j$,  we compute:
$$ P(\underline{x},y_0)=\displaystyle\sum_{(\underline{i},j) \in\mathcal{F}}a_{\underline{i},j}\underline{x}^{\underline{i}}\left(\displaystyle\sum_{\underline{n}\geq_{\mathrm{grlex}} (0,\ldots,0,1)} c_{\underline{n}}^{(j)}\underline{x}^{\underline{n}}\right)+\displaystyle\sum_{(\underline{i},0) \in\mathcal{G}}a_{\underline{i},0}\underline{x}^{\underline{i}}.$$
The coefficients of the expansion of  $P(\underline{x},y_0)$ with respect to the powers of $\underline{x}$ in increasing order $\leq_{\mathrm{grlex}}$ are exactly the components of the infinite vector resulting from the following operation:
$$M_{\mathcal{F},\mathcal{G}}\cdot (a_{\underline{i},j})_{(\underline{i},j)\in\mathcal{F}\cup\mathcal{G}}.$$
The series $y_0$ is a root of a nonzero polynomial with support included into $\mathcal{F}\cup\mathcal{G}$  if and only if there is a non zero solution $(a_{i,j})_{(i,j)\in\mathcal{F}\cup\mathcal{G}}$ of the following equation:
$$ M_{\mathcal{F},\mathcal{G}}\cdot (a_{\underline{i},j})_{(\underline{i},j)\in\mathcal{F}\cup\mathcal{G}}=0.$$
This means that the rank of  $M_{\mathcal{F},\mathcal{G}}$ is less than  $|\mathcal{F}\cup\mathcal{G}|$, the number of columns of  $M_{\mathcal{F},\mathcal{G}}$. The latter condition is characterized as in finite dimension by the vanishing of all the minors of maximal order (see \cite[Lemma 1]{hickel-matu:puiseux-alg}).\\ 
Let us now remark that, in the infinite vector $M_{\mathcal{F},\mathcal{G}}\cdot (a_{\underline{i},j})_{(\underline{i},j)\in\mathcal{F}\cup\mathcal{G}}$, if we remove the components indexed by  $\underline{i}$ for $(\underline{i},0)\in\mathcal{G}$, then we get exactly the infinite vector $M_{\mathcal{F},\mathcal{G}}^{red}\cdot (a_{\underline{i},j})_{(\underline{i},j)\in\mathcal{F}}$. The vanishing of the latter means precisely that the rank of  $M_{\mathcal{F},\mathcal{G}}^{red}$ is less than $|\mathcal{F}|$. 
Conversely, if the columns of  $M_{\mathcal{F},\mathcal{G}}^{red}$ are dependent for certain $\mathcal{F}$ and $\mathcal{G}$, we denote by $(a_{\underline{i},j})_{(\underline{i},j)\in\mathcal{F}}$ a corresponding sequence of coefficients of a nontrivial vanishing linear combination of the column vectors. Then it suffices to note that the remaining  coefficients $a_{\underline{k},0}$ for  $(\underline{k},0)\in\mathcal{G}$ are each uniquely determined  as follows: 
\begin{equation}\label{equ:terme-cst}
a_{\underline{k},0}=-\displaystyle\sum_{(\underline{i},j)\in\mathcal{F},\, \underline{i}< \underline{k}} a_{\underline{i},j}c_{\underline{k}-\underline{i}}^{(j)}\,.
\end{equation} 
\end{demo}

We deal with the implicitization problem for algebraic power series: for fixed bounded degrees in $\underline{x}$ and $y$,  given the expression of an algebraic series, can we reconstruct a vanishing polynomial? if yes, how?

\begin{defi}\label{defi:poly-wilc}
Let us consider the abstract version $\textbf{M}_{\mathcal{F},\mathcal{G}}$ and $\textbf{M}_{\mathcal{F},\mathcal{G}}^{red}$ associated to the abstract series $Y_0=\displaystyle\sum_{\underline{n}\geq_{\mathrm{grlex}} (0,\ldots,0,1)} C_{\underline{n}}\underline{x}^{\underline{n}}\in K[(C_{\underline{n}})_{\underline{n}\in\mathbb{N}^r}][[\underline{x}]]$ and to two sequences $\mathcal{F}$ and $\mathcal{G}$ of  ordered pairs $(\underline{i},j)$ as in Definition \ref{defi:alg-relative} of the  Wilczynski matrices. We call \textbf{Wilczynski polynomial} any  polynomial in the   variables $C_{\underline{n}}$ of $Y_0$ obtained as a minor of $\textbf{M}_{\mathcal{F},\mathcal{G}}^{red}$. We denote such  Wilczynski polynomial by $Q_{\underline{K},\underline{I}}$, where  $\underline{I}:=((\underline{i}_1,j_1),\ldots,(\underline{i}_l,j_l))$ is a subsequence of  $\mathcal{F}$ indicating the $l$ columns of $\textbf{M}_{\mathcal{F},\mathcal{G}}^{red}$, and  $\underline{K}:=(\underline{k}_1,\underline{k}_2,\cdots,\underline{k}_l)$ a strictly increasing sequence of elements of $\left(\mathbb{N}^r,\leq_{\mathrm{grlex}}\right)$ indicating the $l$ rows of  $\textbf{M}_{\mathcal{F},\mathcal{G}}^{red}$ used to form the minor of $\textbf{M}_{\mathcal{F},\mathcal{G}}^{red}$. One  has that $l\in\mathbb{N}^*$, $l\leq |\mathcal{F}|$, $l$ being the order of that minor, that we will also call the  \textbf{order} of the Wilczynski polynomial $Q_{\underline{K},\underline{I}}$. Note also  that a  Wilczynski polynomial $Q_{\underline{K},\underline{I}}$ is  either homogeneous of degree equal to $\displaystyle\sum_{(\underline{i},j)\in \underline{I}}j\ $  or identically 0 (indeed, the multinomial  coefficients  $ C_{\underline{k}}^{(j)}$ in a column indexed by $(\underline{i},j)$ of $\textbf{M}_{\mathcal{F},\mathcal{G}}^{red}$ are either homogeneous of degree $j$ (case $|\underline{k}|\geq j$) or identically 0 (case $|\underline{k}|<j$)). By convention, we call  \textbf{Wilczynski polynomial of order 0} any nonzero constant polynomial. 
\end{defi}

By  Lemma \ref{lemme:wilcz}, the algebraicity of $y_0$ for certain $\mathcal{F}$ and $\mathcal{G}$ is equivalent to the vanishing of all the  $Q_{\underline{K},\mathcal{F}}$ of order $l=|\mathcal{F}|$, for the specific values of the  given $c_{\underline{n}}$, coefficients of $y_0$. \\

\begin{ex}\label{ex:wilc} We resume Example \ref{ex:eclt}, using for simplicity the variables $\underline{x}$ instead of the variables $\underline{u}$. Let $y_0=\displaystyle\sum_{\underline{n}\geq_{\mathrm{grlex}} (0,1)} c_{\underline{n}}\underline{x}^{\underline{n}}\in K[[\underline{x}]]$ be a series with  $c_{0,1}\neq 0$. 
We consider the conditions for $y_0$ to be a root of a polynomial of type:
$$\begin{array}{lcr}
P(\underline{x},y)&=&a_{0,2,0}{x_2}^2+a_{2,2,0}{x_1}^2{x_2}^2+ \left(a_{0,2,1}{x_2}^2+a_{2,2,1}{x_1}^2{x_2}^2\right)y\\&&+ \left(a_{0,0,2}+a_{0,2,2}{x_2}^2+a_{2,2,2}{x_1}^2{x_2}^2\right)y^2.
\end{array}$$
Thus, $\mathcal{F}=\left\{(0,2,1),(2,2,1),(0,0,2),(0,2,2),(2,2,2)\right\}\ \textrm{ and }\ \mathcal{G}=\left\{(0,2,0),(2,2,0)\right\}$.
 The corresponding Wilczynski matrix and the reduced matrix are given in Section \ref{section:appendice}. We give five nontrivial  Wilczynski polynomials of maximal order 5,  which are equal to $5\times 5$  minors of $\textbf{M}^{red}$. So one has that $\underline{I}=\mathcal{F}$ as index for $Q_{\underline{K},\underline{I}}$:\\

\noindent$\begin{array} {rl} Q_{\underline{K},\mathcal{F}}:=&-2\,{C_{{0,1}}}^{7}C_{{1,0}}\ \textrm{ for }\underline{K}=((1,1),(0,3),(0,4),(2,3),(2,4)),  \\
Q_{\underline{K},\mathcal{F}}:=&-2\,{C_{{0,1}}}^{3} \left( C_{{1,0}}C_{{1,1}}+C_{{0,1}}C_{{2,0}}
 \right)  \left( {C_{{0,1}}}^{2}C_{{2,0}}-C_{{0,2}}{C_{{1,0}}}^{2}
 \right)\\
&\ \ \ \ \  \textrm{ for }\underline{K}=((0,3),(2,1),(0,4),(2,4),(4,2))   \\
Q_{\underline{K},\mathcal{F}}:=&-2\,{C_{{0,1}}}^{5} \left( -C_{{0,1}}C_{{1,0}}C_{{0,3}}+{C_{{0,2}}}^{2
}C_{{1,0}}+{C_{{0,1}}}^{2}C_{{1,2}} \right)\\
 &\ \ \ \ \  \textrm{ for }\underline{K}=((0,3),(0,4),(1,3),(2,3),(2,4)),   \\
Q_{\underline{K},\mathcal{F}}:=& -2\,{C_{{0,1}}}^{4} \left( -C_{{0,1}}{C_{{1,0}}}^{2}C_{{0,3}}+{C_{{1,0
}}}^{2}{C_{{0,2}}}^{2}+2\,C_{{0,1}}C_{{1,0}}C_{{0,2}}C_{{1,1}}+{C_{{0,
1}}}^{2}C_{{1,0}}C_{{1,2}}\right.\\
&\left.\ \  \ \ \  \ \ \  \ \ \  \ \ \  \ \ \  \ -{C_{{0,1}}}^{2}{C_{{1,1}}}^{2} \right) \\
&\ \ \ \ \  \textrm{ for }\underline{K}=((1,2),(0,4),(1,3),(2,3),(2,4)).\\
Q_{\underline{K},\mathcal{F}}:=&-2\,{C_{{0,1}}}^{6} \left( C_{{1,0}}C_{{2,1}}+C_{{1,1}}C_{{2,0}}+C_{{0
,1}}C_{{3,0}} \right) \\
&\ \ \ \ \  \textrm{ for }\underline{K}=((0,3),(0,4),(3,1),(2,3),(2,4)).
\end{array}$\\

\noindent The series $y_0$ is a root of a polynomial $P(x,y)$ as above if and only if all the Wilczynski polynomials of order 5 vanish at the $c_{\underline{n}}$. Since $c_{{0,1}}\neq 0$, this implies in particular that:
\begin{equation}\label{equ:exemple}c_{{1,0}}=c_{{2,0}}=c_{{1,2}}=c_{{1,1}}=c_{{3,0}}=0.\end{equation}
\end{ex}

Recall that the dimension of the space of polynomials in $r$ variables of degree at most $d$ is equal to the binomial number $\displaystyle\binom{d+r}{r}$.

\begin{theo}\label{theo:wilc}
Let $\mathcal{F}$ and $\mathcal{G}$ be two finite sequences of ordered pairs as in Definition \ref{defi:alg-relative}. We set  $d_y:=\max\{j,\ (i,j)\in\mathcal{F}\}$, $d_x:=\max\{|\underline{i}|,\ (\underline{i},j)\in\mathcal{F}\cup\mathcal{G}\}$ and $N:=2d_xd_y$.  Then there exist a finite set $\Lambda$ and  a finite number of homogeneous polynomials: $$\left(\left\{a_{\underline{i},j}^{(\lambda)}\in\mathbb{Z}[C_{(0,\ldots,0,1)},\ldots,C_{(N,0,\ldots,0)}],\ \ (\underline{i},j)\in\mathcal{F}\cup\mathcal{G}\right\}\right)_{\lambda\in\Lambda}$$ 
(where the variables $C_{\underline{n}}$ are listed with indices ordered by $\leq_{\textrm{grlex}}$)

$$\textrm{ of total degree } \left\{\begin{array}{ll}
\deg a_{\underline{i},j}^{(\lambda)}\leq \frac{1}{2}d_y(d_y+1)\displaystyle\binom{d_x+r}{r}-1&\textrm{ for }(\underline{i},j)\in\mathcal{F}\\
\deg a_{\underline{i},0}^{(\lambda)}\leq \frac{1}{2}d_y(d_y+1)\displaystyle\binom{d_x+r}{r}-1+|\underline{i}|&\textrm{ for }(\underline{i},0)\in\mathcal{G}\end{array}\right.$$ such that, for any  series $y_0=\displaystyle\sum_{\underline{n}\geq_{\mathrm{grlex}} (0,\ldots,0,1)} c_{\underline{n}}\underline{x}^{\underline{n}}\in K\left[\left[\underline{x}\right]\right]$ with $c_{(0,\ldots,0,1)}\neq 0$ algebraic relatively to $(\mathcal{F},\mathcal{G})$, there is $\lambda\in\Lambda$ such that the polynomial in $ K\left[\underline{x },y\right]$:
\begin{equation}\label{equ:reconstruction}
 P^{(\lambda)}(x,y)=\displaystyle\sum_{(\underline{i},j) \in\mathcal{F}}a_{\underline{i},j}^{(\lambda)}(c_{(0,\ldots,0,1)},\ldots,c_{(N,0,\ldots,0)})\underline{x}^{\underline{i}}y^j+\displaystyle\sum_{(\underline{i},0) \in\mathcal{G}}a_{\underline{i},0}^{(\lambda)}(c_{(0,\ldots,0,1)},\ldots,c_{(N,0,\ldots,0)})\underline{x}^{\underline{i}}
\end{equation}
is nonzero and vanishes at $y_0$. 
\end{theo}

\begin{demo} First, we give the reconstruction process. Then we will show its finiteness.

Let $y_0=\displaystyle\sum_{\underline{n}\ge_{\mathrm{grlex}} (0,\ldots,0,1)} c_{\underline{n}}\underline{x}^{\underline{n}}\in K\left[\left[\underline{x}\right]\right]$ with $c_{(0,\ldots,0,1)}\neq 0$ be algebraic relatively to $(\mathcal{F},\mathcal{G})$. We show how to reconstruct a nonzero vanishing polynomial   $P(\underline{x},y)$ of $y_0$.

We consider a minimal family $\mathcal{F}'\subseteq\mathcal{F}$ such that $y_0$ is algebraic relatively to $(\mathcal{F}',\mathcal{G})$. Let $Q(\underline{x},y)=\displaystyle\sum_{(\underline{i},j) \in\mathcal{F}'}b_{\underline{i},j}\underline{x}^{\underline{i}}y^j+\displaystyle\sum_{(\underline{i},0) \in\mathcal{G}}b_{\underline{i},0}\underline{x}^{\underline{i}}$ be a nonzero polynomial that vanishes at $y_0$. Let $m:=| \mathcal{F}'|$.
If $m=1$, $Q(\underline{x},y)$ is of the form:
$$Q(\underline{x},y)=b_{\underline{i},j}\underline{x}^{\underline{i}}y^j + \displaystyle\sum_{(\underline{i},0) \in\mathcal{G}}b_{\underline{i},0}\underline{x}^{\underline{i}},$$
with $b_{\underline{i},j}\neq 0$. So we must have that $b_{\underline{n},0}=0$ for $|\underline{n}-\underline{i}|<j$, and the series $y_0$ verifies: 
$$\displaystyle\sum_{(\underline{n},0) \in\mathcal{G}}b_{\underline{n },0}\underline{x}^{\underline{n}}=-b_{\underline{i},j}\underline{x}^{\underline{i}}y_0^j=\displaystyle\sum_{\underline{n}\geq \underline{i}} -b_{\underline{i},j}c_{\underline{n}-\underline{i}}^{(j)}\underline{x}^{\underline{n}}.$$
By Lemma \ref{lemme:wilcz}, the minors of order 1 of  $M_{(\underline{i},j),\mathcal{G}}^{red}$,  being equal to $c_{\underline{n}-\underline{i}}^{(j)}$ for $(\underline{n},0)\notin \mathcal{G}$, are all zero. We fix the coefficient $a_{\underline{i},j}$ arbitrarily in $\mathbb{Z}\setminus \{0\}$: it is a  constant Wilczynski polynomial. Then the other coefficients are uniquely determined in accordance with Relation \eqref{equ:terme-cst} by the equation: 
$$a_{\underline{n},0}\left(c_{(0,\ldots,0,1)},c_{(0,\ldots,1,0)},\ldots\right):=-a_{\underline{i},j}c_{\underline{n}-\underline{i}}^{(j)},\ \ (\underline{n},0)\in\mathcal{G}.$$
Thus $a_{\underline{n},0}$ is a polynomial of degree $j$ in the $C_{\underline{k}}$, $\underline{k}\leq_{\mathrm{grlex}} \underline{n}-\underline{i}-(j-1)(0,\ldots,0,1)$, which verifies indeed that $j\leq d_y \leq \frac{1}{2}d_y(d_y+1)\displaystyle\binom{d_x+r}{r}\leq \frac{1}{2}d_y(d_y+1)\displaystyle\binom{d_x+r}{r}-1+|\underline{n}|$. \\
Suppose now that $m=| \mathcal{F}'|\geq 2$.
By Lemma \ref{lemme:wilcz}, the minors of order $m$ of $M_{\mathcal{F}',\mathcal{G}}^{red}$ all vanish, and, because $\mathcal{F}'$ is minimal,     
there exists a nonzero minor of order $m-1$ of this matrix, i.e. a  Wilczynski polynomial evaluated at the $c_{\underline{n}}$'s: 
\begin{equation}\label{equ:nonzerominor}Q_{\underline{K}_0,\underline{I}_0}\left(c_{(0,\ldots,0,1)},c_{(0,\ldots,1,0)},\ldots\right)\neq 0.
\end{equation}
Let $(\underline{i}_0,j_0)\in\mathcal{F}'$ be such that $\mathcal{F}'=\underline{I}_0\cup\{(\underline{i}_0,j_0)\}$ and $p_0$ be the position of $(\underline{i}_0,j_0)$ in $\mathcal{F}'$. Denote by $M_{\underline{K}_0,\underline{I}_0}$ the square matrix  whose determinant is $Q_{\underline{K}_0,\underline{I}_0}\left(c_{(0,\ldots,0,1)},c_{(0,\ldots,1,0)},\ldots\right)$, and $W_{\underline{K}_0,(\underline{i}_0,j_0)}$ the truncated $p_0$-th column that has been removed from $M_{\mathcal{F}',\mathcal{G}}^{red}$ to form this minor. We get a system of equations with a non-vanishing determinant and $b_{\underline{i}_0,j_0}\neq 0$: 
\begin{equation}\label{equ:cramer1}M_{\underline{K}_0,\underline{I}_0}\cdot  (b_{\underline{i},j})_{(\underline{i},j)\neq (\underline{i}_0,j_0)}=-b_{\underline{i}_0,j_0}W_{\underline{K}_0,(\underline{i}_0,j_0)}.
\end{equation}
Let us build polynomials $a_{\underline{i},j}\in\Z[C_{(0,\ldots,0,1)},C_{(0,\ldots,1,0)},\ldots]$ verifying:
\begin{equation}\label{equ:cramer2}M_{\underline{K}_0,\underline{I}_0}\cdot  (a_{\underline{i},j}\left(c_{(0,\ldots,0,1)},c_{(0,\ldots,1,0)},\ldots\right))_{(\underline{i},j)\neq (\underline{i}_0,j_0)}=-a_{\underline{i}_0,j_0}\left(c_{(0,\ldots,0,1)},c_{(0,\ldots,1,0)},\ldots\right)W_{\underline{K}_0,(\underline{i}_0,j_0)},
\end{equation}
by taking $a_{\underline{i}_0,j_0}\left(c_{(0,\ldots,0,1)},c_{(0,\ldots,1,0)},\ldots\right):=(-1)^{p_0}Q_{\underline{K}_0,\underline{I}_0}\left(c_{(0,\ldots,0,1)},c_{(0,\ldots,1,0)},\ldots\right)$ and by computing the other $a_{\underline{i},j}\left(c_{(0,\ldots,0,1)},c_{(0,\ldots,1,0)},\ldots\right)$ by Cramer's rule. Thus the\\ $a_{\underline{i},j}\left(c_{(0,\ldots,0,1)},c_{(0,\ldots,1,0)},\ldots\right)$'s are all minors of order $m-1$ of $M_{\mathcal{F}',\mathcal{G}}^{red}$, and so, up to the sign, evaluations at the $c_{\underline{n}}$'s of Wilczynski polynomials $ Q_{\underline{K}_0,\underline{I}}$ of order $m-1$.  If $\underline{K}_0=(\underline{k}_{0,1},\ldots,\underline{k}_{0,m-1})$, we set: 
\begin{equation}\label{equ:N} \underline{n}_{y_0}:=\underline{k}_{0,m-1}.\end{equation}
The $a_{\underline{i},j}$ are  homogeneous polynomials of $\mathbb{Z}\left[C_{(0,\ldots,0,1)},\ldots,C_{\underline{n}_{y_0}}\right]$ (where the variables $C_{\underline{n}}$ are listed with indices ordered by $\leq_{\textrm{grlex}}$). Their degree  verifies:
$$\begin{array}{lcl}
\deg a_{\underline{i},j}=\deg Q_{\underline{K}_0,\underline{I}}&=&\displaystyle\sum_{(\underline{i},j)\in \underline{I},\  C_{\underline{k}}^{(j)}\neq 0}j\\
&\leq& -1+\displaystyle\sum_{(\underline{i},j)\in\mathcal{F}' }j\\
&\leq& -1+\displaystyle\binom{d_x+r}{r}\displaystyle\sum_{j=1}^{d_y}j\\
&=& \frac{1}{2}d_y(d_y+1)\displaystyle\binom{d_x+r}{r}-1.
\end{array}$$
The coefficients $a_{\underline{n},0}\left(c_{(0,\ldots,0,1)},c_{(0,\ldots,1,0)},\ldots\right)$ for $(\underline{n},0)\in\mathcal{G}$ are obtained via Relations (\ref{equ:terme-cst}):
\begin{equation}\label{equ:CI1}
a_{\underline{n},0}\left(c_{(0,\ldots,0,1)},c_{(0,\ldots,1,0)},\ldots\right)=-\displaystyle\sum_{(\underline{i},j)\in\mathcal{F}', \underline{n}>\underline{i}} a_{\underline{i},j}\left(c_{(0,\ldots,0,1)},c_{(0,\ldots,1,0)},\ldots\right)c_{\underline{n}-\underline{i}}^{(j)}.
\end{equation} 
Note that the coefficients $b_{\underline{n},0}$ for $(\underline{n},0)\in\mathcal{G}$ of $Q$ also satisfy:
\begin{equation}\label{equ:CI2}
b_{\underline{n},0}=-\displaystyle\sum_{(\underline{i},j)\in\mathcal{F}',\underline{ n}>\underline{i}} b_{\underline{i},j}c_{\underline{n}-\underline{i}}^{(j)}.
\end{equation} 
Let us set $a_{\underline{i},j}:=0$ for $( \underline{i},j)\in\mathcal{F}\setminus\mathcal{F}'$. Knowing that  $C_{\underline{n}-\underline{i}}^{(j)}\nequiv 0 \Rightarrow |\underline{n}-\underline{i}|\geq j$, and in this case $\deg C_{\underline{n}-\underline{i}}^{(j)}=j$, we deduce that $\deg a_{\underline{n},0}\leq |\underline{n}|+\max_{(\underline{i},j)\in\mathcal{F}'}(\deg a_{\underline{i},j})$ as desired. As the right-hand sides of Systems (\ref{equ:cramer1}) and (\ref{equ:cramer2}) are proportional, there is $\mu:=\displaystyle\frac{a_{\underline{i}_0,j_0}\left(c_{(0,\ldots,0,1)},c_{(0,\ldots,1,0)},\ldots\right)}{b_{\underline{i}_0,j_0}}  \in K\setminus\{0\}$ such that $a_{\underline{i},j}\left(c_{(0,\ldots,0,1)},c_{(0,\ldots,1,0)},\ldots\right)=\mu\, b_{\underline{i},j}$ for any $(\underline{i},j)\in\mathcal{F}'$. By Systems (\ref{equ:CI1}) and (\ref{equ:CI2}), one has also $a_{\underline{n},0}\left(c_{(0,\ldots,0,1)},c_{(0,\ldots,1,0)},\ldots\right)=\mu\, b_{\underline{n},0}$ for $(\underline{n},0)\in\mathcal{G}$. The polynomial $$P(\underline{x},y)=\displaystyle\sum_{(\underline{i},j)\in\mathcal{F}'\cup\mathcal{G}} a_{\underline{i},j}\left(c_{(0,\ldots,0,1)},c_{(0,\ldots,1,0)},\ldots\right) \underline{x}^{\underline{i}}y^j$$ is proportional to $Q$ (i.e. $P=\mu Q$), so it is nonzero and vanishes at $y_0$.\\

To obtain Theorem \ref{theo:wilc}, it suffices now to show that there exists a \underline{uniform} bound  $N_{d_x,d_y}$ for $|\underline{n}_{y_0}|$ as defined in Formula (\ref{equ:N}), which is a measure of  the depth in  $M_{\mathcal{F},\mathcal{G}}^{red}$ to which we get the reconstruction process, that is, the depth at which we find a first nonzero minor. We reach this in the two following lemmas. 

\begin{lemma}\label{lemme:ordreQ}
Let $d_x,\, d_y\in \mathbb{N}^*$. For any series   $y_0=\displaystyle\sum_{\underline{n}\geq_{\mathrm{grlex}} (0,\ldots,0,1)} c_{\underline{n}}\underline{x}^{\underline{n}}\in K\left[\left[\underline{x}\right]\right]$ with $c_{(0,\ldots,0,1)}\neq 0$, verifying an equation $P(\underline{x},y_0)=0$ where $P(\underline{x},y)\in K\left[\underline{x},y\right]\setminus\{0 \}$, $\deg_{\underline{x}}P\leq d_x,\ \deg_yP\leq d_y$, and for any polynomial $Q(\underline{x},y)\in K\left[\underline{x},y\right],\ \deg_{\underline{x}}Q\leq d_x,\ \deg_yQ\leq d_y$, such that $Q(\underline{x},y_0)\neq 0$, one has that $\mathrm{ord}_{\underline{x}}Q(\underline{x},y_0)\leq 2d_xd_y$.
\end{lemma} 
\begin{demo}
Let $y_0$ be a series as in the statement of Lemma \ref{lemme:ordreQ}. We consider the prime ideal $\mathfrak{I}_0:=\left\{R(\underline{x},y)\in K\left[\underline{x},y\right]\ |\ R(\underline{x},y_0)=0\right\}$.
Since $\mathfrak{I}_0\neq (0)$, $$\dim \left(K[\underline{x},y]/\mathfrak{I}_0\right)=\mathrm{trdeg}_K\mathrm{Frac}\left(K[\underline{x},y]/\mathfrak{I}_0\right)\leq r.$$ But, in $\mathrm{Frac}\left(K[\underline{x},y]/\mathfrak{I}_0\right)$, the elements $\overline{x}_1,\ldots,\overline{x}_r$ are algebraically independant (if not, we would have $P(\overline{x}_1,\ldots,\overline{x}_r)=\overline{0}$ for some non trivial $P\in K[\underline{X}]$, i.e. $P(x_1,\ldots,x_r)\in \mathfrak{I}_0$, a contradiction). Thus, $\mathfrak{I}_0$ is a height one prime ideal of the factorial ring $K\left[\underline{x},y\right]$. It is generated by an irreducible polynomial $P_0(\underline{x},y)\in K\left[\underline{x},y\right]$. We set $d_{0,x}:=\deg_{\underline{x}} P_0$ and $d_{0,y}:=\deg_y P_0$. Note also that, by factoriality of $K\left[\underline{x},y\right]$, $P_0$ is also irreducible as an element of $K\left(\underline{x}\right)[y]$.\\
Let $P$ be as in  the statement of Lemma \ref{lemme:ordreQ}. One has that $P=SP_0$ for some $S\in K\left[\underline{x},y\right]$. Hence $d_{0,x}\leq d_x$ and $d_{0,y}\leq d_y$. Let $Q\in K\left[\underline{x},y\right]$ be such that $Q(\underline{x},y_0)\neq 0$ with $\deg_{\underline{x}} Q\leq d_x$, $\deg_yQ\leq d_y$. So $P_0$ and $Q$ are coprime in $K\left(\underline{x}\right)[y]$. Their resultant $r(\underline{x})$ is nonzero. One has the following B\'ezout relation in $K\left[\underline{x}\right][y]$:
$$A(\underline{x},y)P_0(\underline{x},y)+B(\underline{x},y)Q(\underline{x},y)=r(\underline{x}).$$
We evaluate at $y=y_0$:
$$0+B(\underline{x},y_0)Q(\underline{x},y_0)=r(\underline{x}).$$
So $\mathrm{ord}_{\underline{x}} Q(\underline{x},y_0)\leq \deg_{\underline{x}} r(\underline{x})$. But, the resultant is a  determinant of order at most $d_y+d_{0,y}\leq 2\,d_y$ whose entries are  polynomials in $K\left[\underline{x}\right]$ of degree at most $\max\{d_x,d_{0,x}\}= d_x$. So, $ \deg_{\underline{x}} r(\underline{x})\leq 2\,d_xd_y$. Hence, one has that: $\mathrm{ord}_{\underline{x}} Q(\underline{x},y_0)\leq  2\,d_xd_y$.
\end{demo}

\begin{lemma}\label{lemme:profondeur}
Let $\mathcal{F}''\subsetneq \mathcal{F}$. If $y_0$ is not algebraic relatively to $(\mathcal{F}'',\mathcal{G})$, we denote $l:=|\mathcal{F}''|$ and $\underline{p}:=\min_{\leq_{\textrm{grlex}}}\left\{\underline{k}_l\ |\ Q_{\underline{K},\mathcal{F}''}\left(c_{(0,\ldots,0,1)},c_{(0,\ldots,1,0)},\ldots\right)\neq 0,\ \underline{K}=(\underline{k}_1,\ldots,\underline{k}_l)\right\}$. Then, for any polynomial $Q(\underline{x},y)=\displaystyle\sum_{(\underline{i},j)\in\mathcal{F}''\cup\mathcal{G}}b_{\underline{i},j} \underline{x}^{\underline{i}}y^j$  with $b_{\underline{i},j}\neq 0$ for some $(\underline{i},j)\in\mathcal{F}''$, we have:
$$\mathrm{ord}_{\underline{x}}Q(\underline{x},y_0)\leq |\underline{p}|\leq 2\,d_xd_y,$$
and the value $\underline{p}$ is reached for a certain polynomial $Q_0$. 
\end{lemma}
\begin{demo}
Denote by $\underline{\tilde{p}}$ the predecessor of $\underline{p}$ for $\leq_{\textrm{grlex}}$. By the definition of $\underline{p}$, for any  $\underline{K}=(\underline{k}_1,\ldots,\underline{k}_l)$ with $\underline{k}_l<_{\textrm{grlex}}\underline{p}$, we have that $Q_{\underline{K},\mathcal{F}''}\left(c_{(0,\ldots,0,1)},c_{(0,\ldots,1,0)},\ldots\right)=0$.  This means that the rank of the column vectors $V_{\underline{i},j,\underline{\tilde{p}}}$ that are the restrictions of those of $M_{\mathcal{F}'',\mathcal{G}}^{red}$  up to the row $\underline{\tilde{p}}$, is less than $l=|\mathcal{F}''|$. There are coefficients $(a_{\underline{i},j})_{(\underline{i},j)\in \mathcal{F}''\cup\mathcal{G}}$ not all zero such that $\displaystyle\sum_{(\underline{i},j)\in\mathcal{F}''}a_{\underline{i},j}V_{\underline{i},j,\underline{\tilde{p}}}=0$. By computing the coefficients $a_{\underline{n},0}$ for $(\underline{n},0)\in\mathcal{G}$  via Relations (\ref{equ:terme-cst}):
\begin{equation}\label{equ:CI3}
a_{\underline{n},0}=-\displaystyle\sum_{(\underline{i},j)\in\mathcal{F}'', \underline{n}>\underline{i}} a_{\underline{i},j}c_{\underline{n}-\underline{i}}^{(j)},
\end{equation} 
we obtain the vanishing of the first terms of $Q_0(\underline{x},y_0):=\displaystyle\sum_{(\underline{i},j)\in\mathcal{F}''\cup\mathcal{G}} a_{\underline{i},j} \underline{x}^{\underline{i}}(y_0)^j$ up to  $\tilde{\underline{p}}$.
 Thus, $\mathrm{ord}_{\underline{x}}Q_0(\underline{x},y_0)\geq |\underline{p}|$, and so $|\underline{p}|\leq 2\,d_xd_y$. On the other hand, again by the definition of $\underline{p}$, the column vectors up to the row $\underline{p}$ are, in turn, of rank $l=|\mathcal{F}''|$. 
 This means that the rank of the matrix $M_{\mathcal{F}'',\mathcal{G},\underline{p}}^{red}$ consisting of the first rows of $M_{\mathcal{F}'',\mathcal{G}}^{red}$ up to  $\underline{p}$ is $l$. 
Thus, for any nonzero vector $(b_{\underline{i},j})_{(\underline{i},j)\in\mathcal{F}''}$, we have:
$$M_{\mathcal{F}'',\mathcal{G},\underline{p}}^{red}\,\cdot\, (b_{\underline{i},j})_{(\underline{i},j)\in\mathcal{F}''}\neq 0.$$
But the components of this nonzero vector  are exactly the coefficients $e_{\underline{k}}$, $(\underline{k},0)\notin \mathcal{G}$ and $\underline{k}\leq_{\mathrm{grlex}} \underline{p}$, of the expansion of $\displaystyle\sum_{(\underline{i},j)\in\mathcal{F}''}b_{\underline{i},j}\,\underline{x}^{\underline{i}}\,(y_0)^j$. 
Now, these terms of the latter series do not overlap with the terms of $\displaystyle\sum_{(\underline{i},0)\in\mathcal{G}}b_{\underline{i},0}\,\underline{x}^{\underline{i}}$. Therefore, for a given polynomial $Q(\underline{x},y)=\displaystyle\sum_{(\underline{i},j)\in\mathcal{F}''\cup\mathcal{G}} b_{\underline{i},j}\underline{x}^{\underline{i}}y^j$ with $b_{\underline{i},j}\neq 0$ for some $(\underline{i},j)\in\mathcal{F}''$, the series $Q(\underline{x},y_0)$ has a nonzero term $e_{\underline{k}}$ with $\underline{k}\leq_{\mathrm{grlex}} \underline{p}$, $(\underline{k},0)\notin\mathcal{G}$. Hence, $\mathrm{ord}_xQ(\underline{x},y_0)\leq |\underline{p}|$.
\end{demo}

We achieve the proof of Theorem \ref{theo:wilc} via Lemmas  \ref{lemme:ordreQ} and \ref{lemme:profondeur} by considering for a given algebraic series $y_0$ a family  $\mathcal{F}'\subset \mathcal{F}$ minimal among the families such that  $y_0$ is algebraic relatively to $(\mathcal{F}',\mathcal{G})$.  We consider an associated nonzero Wilczynski polynomial $Q_{\underline{K}_0,\underline{I}_0}$  as in (\ref{equ:nonzerominor}) with  $\underline{n}_{y_0}$ as defined in Formula (\ref{equ:N}) minimal.   Taking $\mathcal{F}''=\underline{I}_0$, Lemma \ref{lemme:profondeur} applies and $\underline{n}_{y_0}=\underline{p}$. So $|\underline{n}_{y_0}|\leq |\underline{p}|\leq N=2\,d_xd_y$.\\
Recall that the coefficients $a_{\underline{i},j}$ constructed in the first part of the proof are homogenous polynomials in $\mathbb{Z}[C_{(0,\ldots,0,1)},\ldots,C_{\underline{n}_{y_0}}]\subseteq\mathbb{Z}[C_{(0,\ldots,0,1)},\ldots,C_{(N,0,\ldots,0)}]$. To complete the proof of Theorem \ref{theo:wilc}, let us describe a finite set $\Lambda$ which enumerates all possible reconstruction formulas. Let $M_N$ be the matrix obtained from $M_{\mathcal{F},\mathcal{G}}^{red}$ by taking its first rows indexed by $\underline{n}$ such that $|\underline{n}|\leq 2d_xd_y$. Set $D:=\displaystyle\binom{2d_xd_y+r}{r}- |\mathcal{G}|$. So $M_N$ is a $D\times |\mathcal{F}|$-matrix. Let $\nu$ be the number of minors of order less or equal to $\min\{D\,,\, |\mathcal{F}|-1\}$ of $M_N$. We fix a finite set $\Lambda$ of cardinality $|\mathcal{F}|+\nu$. Its first $|\mathcal{F}|$ elements are the indices of reconstruction formulas (\ref{equ:reconstruction}) as built in the first part of the proof of Theorem \ref{theo:wilc} (case $m=|\mathcal{F}'|=1$). The other $\nu$ elements are used to enumerate reconstruction formulas (\ref{equ:reconstruction}) in the case described in the second part of the proof (case $m=|\mathcal{F}'|\geq 2$). 
 \end{demo}

\noindent\textbf{Construction of the coefficients $a_{\underline{i},j}^{(\lambda)}$ for a given $y_0$}.\\
Let $y_0$ be algebraic relatively to $(\mathcal{F},\mathcal{G})$ as in Definition \ref{defi:alg-relative}. Let $N=2\,d_xd_y$ and $D:=\displaystyle\binom{2d_xd_y+r}{r}- |\mathcal{G}|$ as in Theorem \ref{theo:wilc} and its proof. Recall that  $M_N$ denotes the matrix consisting in the $D$ first rows of $M_{\mathcal{F},\mathcal{G}}^{red}$. Let $\rho$ be the rank of $M_N$, and $\theta:=\rho+1$. The minors of $M_N$ of order $\theta$ are all zero and there exists a minor of order $\theta-1=\rho$ which is nonzero. There are two cases. If $\rho=0$, we choose  $(\underline{i},j)\in\mathcal{F}$ and we fix the coefficients $a_{\underline{i},j}:=1$ and $a_{\underline{l},m}=0$ for $(\underline{l},m)\in\mathcal{F}$, $(\underline{l},m)\neq (\underline{i},j)$. Then we derive the  coefficients $a_{\underline{i},0}$ for $(\underline{i},0)\in\mathcal{G}$ from Relations (\ref{equ:terme-cst}). The polynomials $P$ thus obtained are all annihilators of $y_0$. \\
If $\rho\geq 1$, we consider all the Wilczynski polynomials $Q_{\underline{K},\underline{I}}$ of order $\rho$ that do not vanish when evaluated at $c_{(0,\ldots,0,1)},\ldots,c_{(N,0,\ldots,0)}$. Each of them allows to reconstruct  a family of coefficients $a_{\underline{i},j}^{(\lambda)}$, $(\underline{i},j)\in\mathcal{F}$ as described after (\ref{equ:cramer2}), and subsequently  $a_{\underline{i},0}^{(\lambda)}$, $(\underline{i},0)\in\mathcal{G}$ via Relations (\ref{equ:terme-cst}). The corresponding polynomials $P^{(\lambda)}$ are annihilators of $y_0$ if and only if $$\mathrm{ord}_{\underline{x}}\  P^{(\lambda)}\left(\underline{x},\, \displaystyle\sum_{\underline{k}=(0,\ldots,0,1)}^{(N,0,\ldots,0)}c_{\underline{k}}\underline{x}^{\underline{k}}\right)\ >\ N.$$

\begin{remark}\label{rem:Lambda}
With the hypothesis and notations of Theorem \ref{theo:wilc} and its proof, let us denote $f:=|\mathcal{F}|\leq \displaystyle\binom{d_x+r}{r}\, d_y$ and $g:=\min\{D,f-1\}$. Then $|\Lambda|$ is bounded by $f+\displaystyle\sum_{t=1}^g{{f-1}\choose{t}} {{D}\choose{t}}$, which is itself roughly bounded by $f+(2^{f-1}-1)(2^D-1)$.
\end{remark}

\begin{ex}\label{ex:wilcz2}
We resume Example \ref{ex:wilc}, and note that, for\\ $\underline{I}=((0,2,1),(2,2,1),(0,2,2),(2,2,2))$ and $\underline{K}=((0,3),(0,4),(2,3),(2,4))$, we have that:
$$Q_{\underline{K},\underline{I}}=
\left|\begin{array}{cccc}c_{{0,1}}  & 0 &0&0\\
c_{{0,2}} & 0&{c_{{0,1}}}^2&0 \\
c_{{2,1}}&c_{{0,1}}&   2\,c_{{0,1}}c_{{2,0}}+2\,c_{{1,0}}c_{{1,1}}  &0\\
c_{{2,2}}&c_{{0,2}}&2\,c_{{1,0}}c_{{1,2}}+{c_{{1,1}}}^{2}+2\,c_{{0,2}}c_{{2,0}}+2\,c_{{0,1}}c_{{2,1}}&{c_{{0,1}}}^2
\end{array}\right|=-{c_{0,1}}^6\neq 0.$$
So we set $a_{0,0,2}:=(-1)^3(-{c_{0,1}}^6)={c_{0,1}}^6$ and, applying  Cramer's rule:\\
$\begin{array}{l}
a_{0,2,1}:=\\
\ \ \ \ (-1)^1
\left|\begin{array}{cccc} 0&2\,c_{{0,1}}c_{{0,2}} &0&0\\
0& 2\,c_{{0,1}}c_{{0,3}}+{c_{{0,2}}}^{2} &{c_{{0,1}}}^2&0 \\
c_{{0,1}} & c_{{2,3}}^{(2)}  &  2\,c_{{0,1}}c_{{2,0}}+2\,c_{{1,0}}c_{{1,1}}   &0\\
c_{{0,2}}&c_{{2,4}}^{(2)}&  2\,c_{{1,0}}c_{{1,2}}+{c_{{1,1}}}^{2}+2\,c_{{0,2}}c_{{2,0}}+2\,c_{{0,1}}c_{{2,1}}   &{c_{{0,1}}}^2
\end{array}\right|\\
\ \ \ \ \ \ \ \ \ \ =-2{c_{0, 1}}^6 c_{0, 2} \vspace{3pt}\\
a_{2,2,1}:= \\
\ \ \ \ (-1)^2
\left|\begin{array}{cccc} c_{{0,1}}&2\,c_{{0,1}}c_{{0,2}} &0&0\\
c_{{0,2}}& 2\,c_{{0,1}}c_{{0,3}}+{c_{{0,2}}}^{2} &{c_{{0,1}}}^2&0 \\
c_{{2,1}} & c_{{2,3}}^{(2)}  &  2\,c_{{0,1}}c_{{2,0}}+2\,c_{{1,0}}c_{{1,1}}   &0\\
c_{{2,2}}&c_{{2,4}}^{(2)}&  2\,c_{{1,0}}c_{{1,2}}+{c_{{1,1}}}^{2}+2\,c_{{0,2}}c_{{2,0}}+2\,c_{{0,1}}c_{{2,1}}   &{c_{{0,1}}}^2
\end{array}\right|\\
\ \ \ \ \ \ \ \ \ \ =-2\,{c_{{0,1}}}^{3} \left( -2\,c_{{0,1}}c_{{0,3}}c_{{1,0}}c_{{1,1}}-{c
_{{0,1}}}^{2}c_{{0,3}}c_{{2,0}}+{c_{{0,2}}}^{2}c_{{1,0}}c_{{1,1}}+{c_{
{0,2}}}^{2}c_{{0,1}}c_{{2,0}}\right.\\
\ \ \  \ \ \ \ \ \ \ \ \ \ \ \ \  \left.+{c_{{0,1}}}^{2}c_{{1,1}}c_{{1,2}}+{c_{{0
,1}}}^{2}c_{{1,0}}c_{{1,3}}+{c_{{0,1}}}^{3}c_{{2,2}} \right) 
\vspace{3pt}\\
 a_{0,2,2}:=
  (-1)^4 \left|\begin{array}{cccc}c_{{0,1}}  & 0 &2\,c_{{0,1}}c_{{0,2}}&0\\
c_{{0,2}} & 0&2\,c_{{0,1}}c_{{0,3}}+{c_{{0,2}}}^{2}&0 \\
c_{{2,1}}&c_{{0,1}}&   c_{{2,3}}^{(2)} &0\\
c_{{2,2}}&c_{{0,2}}& c_{{2,4}}^{(2)}&{c_{{0,1}}}^2
\end{array}\right|=- {c_{0, 1}}^4\left(2c_{0, 1}c_{0, 3}-{c_{0, 2}}^2\right)
\end{array}\\
\begin{array}{l}
 a_{2,2,2}:=\\
\ \ \ \ (-1)^5 \left|\begin{array}{cccc}c_{{0,1}}  & 0 &0&2\,c_{{0,1}}c_{{0,2}}\\
c_{{0,2}} & 0&{c_{{0,1}}}^2&2\,c_{{0,1}}c_{{0,3}}+{c_{{0,2}}}^{2} \\
c_{{2,1}}&c_{{0,1}}&   2\,c_{{0,1}}c_{{2,0}}+2\,c_{{1,0}}c_{{1,1}}  &c_{{2,3}}^{(2)}\\
c_{{2,2}}&c_{{0,2}}&2\,c_{{1,0}}c_{{1,2}}+{c_{{1,1}}}^{2}+2\,c_{{0,2}}c_{{2,0}}+2\,c_{{0,1}}c_{{2,1}}&c_{{2,4}}^{(2)}
\end{array}\right|\\
\ \ \ \ \ \ \ \ \ \ =-c_{{0,1}} \left( 2\,{c_{{0,1}}}^{4}c_{{2,3}}+2\,{c_{{0,1}}}^{3}c_{{1,0}}c_{{1,4}}+2\ {c_{{0,1}}}^{3}c_{{2,0}}c_{{0,4}} +{c_{{0,1}}}^{3}{c_{{1,2}}}^{2}+2\,{c_{{0,1}}}^{3}c_{{1,1}}c_{{1,3}}\right. \\
\ \ \ \ \ \ \ \ \ \ \ \  \left. -2\,{c_{{0,1}}}^{3}c_{{0,3}}c_{{2,1}}-2\,{c_{{0,1}}}^{2}c_{{0,3}}c_{{0,2}}c_{{2,0}}-2\,{c_{{0,
1}}}^{2}c_{{0,3}}{c_{{1,1}}}^{2}-4\,{c_{{0,1}}}^{2}c_{{0,3}}c_{{1,0}}c_{{1,2}}\right.\\
\ \ \ \ \ \ \ \ \ \ \ \ \left.+c_{{0,1}}{c_{{0,2}}}^{2}{c_{{1,1}}}^{2}+2\,c_{{0,1}}{c_{{0,2}}}^{2}c_{{1,0}}c_{{1,2}}+2\,{c_{{0,2}}}^{2}{c_{{0,1}}}^{2}c_{{2,1}}+4\,c_{{0,2}}c_{{0,1}}c_{{0,3}}c_{{1,0}}c_{{1,1}}\right.\\
\ \ \ \ \ \ \ \ \ \ \ \ \left.-2\,{c_{{0,2}}}^{3}c_{{1,0}}c_{{1,1}} -2\,c_{{0,2}}{c_{{0,1}}}^{2}c_{{1,1}}c_{{1,2}}-2\,c_{{0,2}}{c_{{0,1}}}^{2}c_{{1,0}}c_{{1,3}} -2\,{c_{{0,1}}}^{3}c_{{0,2}}c_{{2,
2}} \right) .
\end{array} $\\
We deduce from Formulas (\ref{equ:terme-cst}) that:
$$\left\{\begin{array}{lcl}
a_{0,2,0}&=&-a_{0,2,1}. 0-a_{2,0,1}. 0-a_{0,0,2}. {c_{{0,1}}}^2-a_{0,2,2}. 0-a_{2,0,2}. 0\ =\ -{c_{{0,1}}}^8\\
a_{2,2,0}&=&-a_{0,2,1}. c_{{2,0}}-a_{2,0,1}. 0-a_{0,0,2}. \left( 2\,c_{{1,0}}c_{{1,2}}+{c_{{1,1}}}^{2}+2\,c_{{0,2}}c_{{2,0}}+2\,c_{{0,1}}c_{{2,1}}  \right)\\
&&-a_{0,2,2}. {c_{{1,0}}}^{2}-a_{2,0,2}. 0\\
&=&-{c_{{0,1}}}^{4} \left( {c_{{0,1}}}^{2}{c_{{1,1}}}^{2}+2\,{c_{{0,1}}}^
{2}c_{{1,0}}c_{{1,2}}+2\,{c_{{0,1}}}^{3}c_{{2,1}}-2\,c_{{0,1}}{c_{{1,0
}}}^{2}c_{{0,3}}+{c_{{1,0}}}^{2}{c_{{0,2}}}^{2} \right). 
\end{array}\right. $$
Since the series $y_0=\displaystyle\sum_{\underline{n}\geq_{\mathrm{grlex}} 0,1} c_{\underline{n}}\underline{x}^{\underline{n}}\in K[[\underline{x}]]$,  $c_{0,1}\neq 0$, is algebraic relatively to\\ $\mathcal{F}=\left\{(0,2,1),(2,2,1),(0,0,2),(0,2,2),(2,2,2)\right\}$  and $\mathcal{G}=\left\{(0,2,0),(2,2,0)\right\}$, we may apply  the conditions (\ref{equ:exemple}). Therefore, we obtain a vanishing polynomial of $y_0$ of the form:
$$\begin{array}{lcl}P(\underline{x},y)&=&-{c_{{0,1}}}^8{x_2}^2-2{c_{{0,1}}}^7c_{2, 1}{x_1}^{2}{x_2}^2+\left(-2{c_{{0,1}}}^6 c_{0, 2}{x_2}^2-2{c_{{0,1}}}^6c_{2, 2}{x_1}^2{x_2}^2\right)y\\
&&+ \left({c_{{0,1}}}^6-{c_{{0,1}}}^{4} \left( 2\,c_{{0,1}}c_{{0,3}}-{c_{{0,2}}}^{2} \right){x_2}^2-c_{{0,1}} \left( 2\,{c_{{0,1}}}^{4}c_{{2,3}}-2\,{c_{{0,1}}}^{3}c_{{0,3}}c_{{2,1}}\right.\right.\\
&&\left.\left.  \ \ \ \ \ +2\,{c_{{0,2}}}^{2}{c_{{0,1}}}^{2}c_{{2,1}}- 2\,{c_{{0,1}}}^{3}c_{{0,2}}c_{{2,2}} \right) 
{x_1}^2{x_2}^2\right)y^2\\
&=&{c_{{0,1}}}^{3} \left[ -{c_{{0,1}}}^{5}{x_{{2}}}^{2}-2\,{c_{{0,1}}}^{4
}c_{{2,1}}{x_{{1}}}^{2}{x_{{2}}}^{2}+\left(-2\,{c_{{0,1}}}^{3}c_{{0,2}}{x_{{
2}}}^{2}-2\,{c_{{0,1}}}^{3}c_{{2,2}}{x_{{1}}}^{2}{x_{{2}}}^{2}\right)y\right.\\
&&\left.+\left({c_{{0
,1}}}^{3}+\left(c_{{0,1}}{c_{{0,2}}}^{2}-2\,{c_{{0,1}}}^{2}c_{{0,3}}\right){x_{{2}}}^{2}\right.\right.\\ 
&&\left.\left.\ \  \ \ \ +\left(2\,c_{{0,1}}c_{{0,3}}c_{{2,1}}-2\,{c_{{0,1}}}^{2}c_{{2,3}}-2\,{c_{{0,2}}
}^{2}c_{{2,1}}+2\,c_{{0,1}}c_{{0,2}}c_{{2,2}}\right){x_{{1}}}^{2}{x_{{2}}}^{2}\right){y}^{2} \right] .\end{array}$$
\end{ex}

\section{A generalization of the Flajolet-Soria Formula.}\label{section:hensel}
Let us assume from now on that $K$ has characteristic zero. In the monovariate context, let $Q(x,y)=\displaystyle\sum_{i,j}a_{i,j}x^iy^j\,\in K[x,y]$ with $Q(0,0)=\displaystyle\frac{\partial Q}{\partial y}(0,0)=0$ and $Q(x,0)\neq 0$. 
In \cite{flajolet-soria:coeff-alg-series}, P. Flajolet and M. Soria give the following formula for the coefficients of the unique formal solution $y_0=\displaystyle\sum_{n\geq 1}c_nx^n$ of the implicit equation $y=Q(x,y)$: 

\begin{theo}[Flajolet-Soria's Formula \cite{flajolet-soria:coeff-alg-series}]\label{theo:formule-FS0}
$$ c_n=\displaystyle\sum_{m=1}^{2n-1}\frac{1}{m}\displaystyle\sum_{|\underline{k}|=m,\ ||\underline{k}||=m-1,\ G(\underline{k})=n}\frac{m!}{\prod_{i,j}k_{i,j}!}\prod_{i,j}a_{i,j}^{k_{i,j}},$$
where $\underline{k}=\displaystyle(k_{i,j})_{i,j}$, $\ |\underline{k}|=\displaystyle\sum_{i,j}k_{i,j}$,  $\ ||\underline{k}|| = \displaystyle\sum_{i,j}j\, k_{i,j}$ and $\ G(\underline{k}) = \displaystyle\sum_{i,j}i\, k_{i,j}$.
\end{theo}

Note that in the particular case where the coefficients of $Q$ verify $a_{0,j}=0$ for all $j$, one has $m\leq n$ in the summation. 

One can derive immediately from  Theorems 3.5 and 3.6 in \cite{sokal:implicit-function}   a multivariate version of the Flajolet-Soria Formula in the case where $Q(\underline{x},y)\in K\left[\underline{x},y\right]$. The purpose of the present section is to generalize the latter result to the case where $Q\left(\underline{x},y\right)\in K\left[x_1,x_1^{-1},\ldots,x_r,x_r^{-1}\right]\left[y\right]$.\\


We will need a special version of Hensel's Lemma for multivariate power series elements of  $K((x_1^{\mathbb{Z}},\ldots,x_r^\mathbb{Z}))^{\textrm{grlex}}$. Recall that the latter denotes the field of generalized series $\left(K\left(\left(X^{\mathbb{Z}^r}\right)\right)^{\mathrm{grlex}},\, w\right)$  where $w$ is the graded lexicographic valuation as described in Section \ref{section:preliminaries}. Generalized series fields are known to be Henselian \cite[Theorem 4.1.3 and Remark 4.1.8]{engler-prestel:valued-fields}. For the convenience of the reader, we give a short proof in our particular context.

\begin{defi}\label{defi:equ-hensel-red}
We call \textbf{strongly reduced Henselian equation} any equation of the following type:
$$y=Q\left(\underline{x},y\right)\ \textrm{ with }\ Q\left(\underline{x},y\right)\in K\left[x_1,x_1^{-1},\ldots,x_r,x_r^{-1}\right]\left[y\right],$$
such that $w\left(Q\left(\underline{x},y\right)\right)>_{\textrm{grlex}}\underline{0}$ and $Q\left(\underline{x},0\right)\nequiv 0$.
\end{defi}

\begin{theo}[Hensel's lemma]\label{theo:hensel}
Any strongly reduced Henselian equation admits a unique solution $y_0= \displaystyle\sum_{\underline{n}>_{\textrm{grlex}}\underline{0}}c_{\underline{n}}\underline{x}^{\underline{n}}\in  K((x_1^{\mathbb{Z}},\ldots,x_r^\mathbb{Z}))^{\mathrm{grlex}}$. 
\end{theo}
\begin{demo} Let 
\begin{equation}\label{equ:hensel}y=Q\left(\underline{x},y\right)
\end{equation}
 be a strongly reduced Henselian equation and let  $y_0=\displaystyle\sum_{\underline{n}>_{\textrm{grlex}}\underline{0}}c_{\underline{n}}\underline{x}^{\underline{n}} \in K((x_1^{\mathbb{Z}},\ldots,x_r^{\mathbb{Z}}))^{\textrm{grlex}}$. For $\underline{n}\in\mathbb{Z}^r$, $\underline{n}>_{\textrm{grlex}}0$, let us denote $\tilde{z}_{\underline{n}}:= \displaystyle\sum_{\underline{m}<_{\textrm{grlex}}\underline{n}} c_{\underline{m}}\underline{x}^{\underline{m}}$.
 We get started with the following key lemma:
\begin{lemma}\label{lemma:hensel} The following are equivalent:
\begin{enumerate}
\item a series $y_0$ is a solution of (\ref{equ:hensel});
\item for any $\underline{n}\in\mathbb{Z}^r$, $\underline{n}>_{\mathrm{grlex}}\underline{0}$, 
 $$ w\left(\tilde{z}_{\underline{n}}-Q\left(\underline{x},\tilde{z}_{\underline{n}}\right)\right)=w\left(y_0-\tilde{z}_{\underline{n}}\right);$$
\item for any $\underline{n}\in\mathbb{Z}^r$, $\underline{n}>_{\textrm{grlex}}\underline{0}$, 
 $$ w\left(\tilde{z}_{\underline{n}}-Q\left(\underline{x},\tilde{z}_{\underline{n}}\right)\right)\geq_{\mathrm{grlex}}\underline{n}.$$
\end{enumerate}
\end{lemma}
\begin{demo}
For $\underline{n}>_{\textrm{grlex}}\underline{0}$, let us denote $\tilde{y}_{\underline{n}}:=y_0-\tilde{z}_{\underline{n}}=\displaystyle\sum_{\underline{m}\geq_{\textrm{grlex}}\underline{n}} c_{\underline{m}}\underline{x}^{\underline{m}}$. We apply Taylor's Formula to $P(\underline{x},y):=y-Q(\underline{x},y)$ at $\tilde{z}_{\underline{n}}$:
$$P\left(\underline{x},\tilde{z}_{\underline{n}}+y\right) =\tilde{z}_{\underline{n}}-Q\left(\underline{x},\tilde{z}_{\underline{n}}\right)+\left(1-\displaystyle\frac{\partial Q}{\partial y}\left(\underline{x},\tilde{z}_{\underline{n}}\right)\right)y +y^2R\left(\underline{x},y\right),$$
where $R\left(\underline{x},y\right)\in K((x_1^{\mathbb{Z}},\ldots,x_r^{\mathbb{Z}}))^{\textrm{grlex}}[y]$ with $w\left(R(\underline{x},y)\right)>_{\textrm{grlex}}\underline{0}$. The series
$y_0$ is a solution of (\ref{equ:hensel}) iff for any $\underline{n}$, $\tilde{y}_{\underline{n}}$ is a root of $P\left(\underline{x},\tilde{z}_{\underline{n}}+y\right)=0$, i.e.:
\begin{equation}\label{equ:henselian-trunc}
\tilde{z}_{\underline{n}}-Q\left(\underline{x},\tilde{z}_{\underline{n}}\right)+\left(1-\displaystyle\frac{\partial Q}{\partial y}\left(\underline{x},\tilde{z}_{\underline{n}}\right)\right)\tilde{y}_{\underline{n}}+ \tilde{y}_{\underline{n}}^2R\left(\underline{x},\tilde{y}_{\underline{n}}\right)=0.
\end{equation}  Now consider $y_0$ a solution of (\ref{equ:hensel}) and $\underline{n}\in\mathbb{Z}^r$, $\underline{n}>_{\mathrm{grlex}}\underline{0}$. Either $\tilde{y}_{\underline{n}}=0$, i.e. $y_0=\tilde{z}_{\underline{n}}$: (2) holds trivially. Or  $\tilde{y}_{\underline{n}}\neq 0$, so we have:
$$\underline{n}\leq_{\textrm{grlex}} w\left(\left(1-\displaystyle\frac{\partial Q}{\partial y}\left(\underline{x},\tilde{z}_{\underline{n}}\right)\right)\tilde{y}_{\underline{n}}\right) =w\left(\tilde{y}_{\underline{n}}\right)<_{\textrm{grlex}} 2w\left(\tilde{y}_{\underline{n}}\right)<_{\textrm{grlex}} w\left(\tilde{y}_{\underline{n}}^2R\left(\underline{x},\tilde{y}_{\underline{n}}\right)\right).$$
So we must have $w\left(\tilde{z}_{\underline{n}}-Q\left(\underline{x},\tilde{z}_{\underline{n}}\right)\right)=w\left(\tilde{y}_{\underline{n}}\right)$. \\
Now, $(2)\,\Rightarrow\, (3)$ since $w\left(\tilde{y}_{\underline{n}}\right)\geq_{\textrm{grlex}}\underline{n}$.\\
 Finally, suppose that for any $\underline{n}$,  $w\left(\tilde{z}_{\underline{n}}-Q\left(\underline{x},\tilde{z}_{\underline{n}}\right)\right)\geq_{\textrm{grlex}}\underline{n}$. If $y_0-Q\left(\underline{x},y_0\right)\neq 0$, denote  $\underline{n}_0:= w\left(y_0-Q\left(\underline{x},y_0\right)\right)$. For $\underline{n}>_{\textrm{grlex}}\underline{n}_0$, one has $$\underline{n}_0=w\left(\tilde{z}_{\underline{n}}-Q\left(\underline{x},\tilde{z}_{\underline{n}}\right)\right)\geq_{\textrm{grlex}}\underline{n}.$$ A contradiction.
\end{demo}
Let us return to the proof of Theorem \ref{theo:hensel}. Note that, if $y_0$ is a solution of (\ref{equ:hensel}), then its support needs to be included in the monoid $\mathcal{S}$ generated by the $\underline{i}$'s from the nonzero coefficients $a_{\underline{i},j}$ of $Q(\underline{x},y)$. If not, consider the smallest index $\underline{n}$ for $\leq_{\mathrm{grlex}}$ which is not in  $\mathcal{S}$. Property (2) of Lemma \ref{lemma:hensel} gives a contradiction for this index.
Since $\mathcal{S}$ is a finitely generated totally ordered monoid in $(\mathbb{Z}^r)_{\geq_\mathrm{grlex}\underline{0}}$, by \cite[Corollary 1.2]{evans-al:tot-ord-commut-monoids},  it is a  well-ordered set.
Let us prove  by transfinite induction on $\underline{n}\in \mathcal{S}$  the existence and uniqueness of a sequence of series  $\tilde{z}_{\underline{n}}$ as in the statement of the previous lemma.  Suppose that for some $\underline{n}\in\mathcal{S}$,  we are given a series $\tilde{z}_{\underline{n}}$ with support included in $\mathcal{S}$ and $<_{\textrm{grlex}}\underline{n}$, such that $w\left(\tilde{z}_{\underline{n}}-Q\left(\underline{x},\tilde{z}_{\underline{n}}\right)\right)\geq_{\textrm{grlex}}\underline{n}$. Then by Taylor's formula as in the proof of the previous lemma, denoting by $\underline{m}$ the successor of $\underline{n}$ in $\mathcal{S}$ for $\leq_{\textrm{grlex}}$:
$$P\left(\underline{x},\tilde{z}_{\underline{m}}\right)=P\left(\underline{x},\tilde{z}_{\underline{n}}+c_{\underline{n}}\underline{x}^{\underline{n}}\right) =\tilde{z}_{\underline{n}}-Q\left(\underline{x},\tilde{z}_{\underline{n}}\right)+\left(1-\displaystyle\frac{\partial Q}{\partial y}\left(\underline{x},\tilde{z}_{\underline{n}}\right)\right)c_{\underline{n}}\underline{x}^{\underline{n}} +c_{\underline{n}}^2\underline{x}^{2\underline{n}}R\left(\underline{x},\tilde{z}_{\underline{n}}\right).$$
Therefore, one has:
$$ w\left(P\left(\underline{x},\tilde{z}_{\underline{m}}\right)\right)=w\left(\tilde{z}_{\underline{m}}-Q\left(\underline{x},\tilde{z}_{\underline{m}}\right)\right)\geq_{\textrm{grlex}}\underline{m}>_{\textrm{grlex}}\underline{n}$$
if and only if $c_{\underline{n}}$ is equal to the coefficient of $\underline{x}^{\underline{n}}$ in $Q\left(\underline{x},\tilde{z}_{\underline{n}}\right)$. This determines $\tilde{z}_{\underline{m}}$ in a unique way as desired.
\end{demo}

We prove now our generalized version of the Flajolet-Soria Formula \cite{flajolet-soria:coeff-alg-series}. Our proof, as the one in \cite{sokal:implicit-function}, uses the classical Lagrange Inversion Formula in one variable. We will use Notation \ref{nota:FS}.

\begin{theo}[Generalized multivariate Flajolet-Soria Formula]\label{theo:formule-FS}\indent \\
Let $y=Q\left(\underline{x},y\right)=\displaystyle\sum_{\underline{i},j}a_{\underline{i},j} \underline{x}^{\underline{i}}y^j$ be a strongly reduced Henselian equation.  Define $\underline{\iota}_0=(\iota_{0,1},\ldots,\iota_{0,r})$ by: $$-\iota_{0,k}:=\min\left\{0,\, i_k\, /\, a_{\underline{i},j}\neq 0,\, \underline{i} = (i_1,\ldots,i_k,\ldots,i_r)\right\},\ \ \ k=1,\ldots,r.$$
Then the  coefficients $c_{\underline{n}}$ of the unique solution $y_0=\displaystyle\sum_{\underline{n}>_{\mathrm{grlex}}\underline{0} } c_{\underline{n}}\underline{x}^{\underline{n}}\in K((x_1^{\mathbb{Z}},\ldots,x_r^\mathbb{Z}))^{\mathrm{grlex}}$ are given by:
$$ c_{\underline{n}}=\displaystyle\sum_{m=1}^{\mu_{\underline{n}}}\frac{1}{m}\displaystyle\sum_{|\underline{M}|=m,\ ||\underline{M}||=m-1,\ G(\underline{M})=\underline{n}}\frac{m!}{\underline{M}!}\underline{A}^{\underline{M}}$$
where $\mu_{\underline{n}}$ is the greatest integer $m$ such that there exists an $\underline{M}$ with $|\underline{M}|=m,\ ||\underline{M}||=m-1$ and $G(\underline{M})=\underline{n}$. Moreover, for  $\underline{n}=(n_1,\ldots,n_r)$,  $\mu_{\underline{n}}\leq \displaystyle\sum_{k=1}^r\lambda_k\, n_k$ with:
$$\lambda_k=\left\{ \begin{array}{ll}
\displaystyle\prod_{j=k+1}^{r-1}(1+\iota_{0,j})+\displaystyle\prod_{j=1}^{r-1}(1+\iota_{0,j}) & \textrm{ if } k<r-1; \\
1+\displaystyle\prod_{j=1}^{r-1}(1+\iota_{0,j})& \textrm{ if } k=r-1;\\
\displaystyle\prod_{j=1}^{r-1}(1+\iota_{0,j})& \textrm{ if } k=r.\end{array}  \right.$$
\end{theo}
\begin{remark}\label{rem:entier}
By Lemma \ref{lemma:arithm}, note that in fact $\displaystyle\frac{1}{m}\cdot \displaystyle\frac{m!}{\underline{M}!}\in\mathbb{N}$. If we set $m_j:=\displaystyle\sum_{\underline{i}}m_{\underline{i },j}$ and $\underline{N}=(m_j)_j$, then $|\underline{N}|=m$, $\|\underline{N}\|=m-1$ and:
$$\displaystyle\frac{1}{m}\cdot \displaystyle\frac{m!}{\underline{M}!}= \displaystyle\frac{1}{m}\cdot \displaystyle\frac{m!}{\underline{N}!} \cdot \displaystyle\frac{N!}{\underline{M}!},$$
where $\displaystyle\frac{N!}{\underline{M}!}$ is a product of multinomial coefficients and $\displaystyle\frac{1}{m}\cdot \displaystyle\frac{m!}{\underline{N}!}$ is an integer by Lemma  \ref{lemma:arithm}. 
Thus, each $c_n$ is the evaluation at the $a_{i,j}$'s of a polynomial with coefficients in $\mathbb{Z}$.
\end{remark}
\begin{demo}
For a given strongly reduced Henselian equation $y=Q\left(\underline{x},y\right)$, one can expand:
$$f\left(\underline{x},y\right):=\displaystyle\frac{y}{Q\left(\underline{x},y\right)}=\displaystyle\sum_{n\geq 1}b_n(\underline{x})y^n\,\in K(\underline{x})[[y]]\ \mathrm{with}\ b_1\neq 0,$$
which admits a unique formal inverse in $K(\underline{x})[[y]]$:
$$\tilde{f}\left(\underline{x},y\right)= \displaystyle\sum_{m\geq  1}d_m(\underline{x}) y^m.$$
The Lagrange Inversion Theorem (see e.g. \cite[Theorem 2]{henrici:lagr-burmann} with $\mathcal{F}=K(\underline{x})$ and $P=f(\underline{x},y)$) applies: for any $m$, $d_m(\underline{x})$ is equal to the coefficient of $y^{m-1}$ in  $\left[Q\left(\underline{x},y\right)\right]^m$, divided by $m$. Hence, according to the multinomial expansion of $\left[Q\left(\underline{x},y\right)\right]^m=\left[\displaystyle\sum_{\underline{i},j}a_{\underline{i},j} \underline{x}^{\underline{i}}y^j\right]^m$:
$$d_m(\underline{x})=\displaystyle\frac{1}{m}\displaystyle\sum_{|\underline{M}|=m,\  ||\underline{M}||=m-1}\frac{m!}{\underline{M}!}\underline{A}^{\underline{M}}\underline{x}^{G(\underline{M})}.$$
Note that the powers $\underline{n}$ of $\underline{x}$ that  appear in $d_m$ are nonzero elements of the monoid generated by the exponents $\underline{i} $ of the monomials $\underline{x}^{\underline{i}}y^j$ appearing  in $Q\left(\underline{x},y\right)$, so they are $>_{\mathrm{grlex}}0$.
Now, it will suffice to show that, for any fixed $\underline{n}$,  the number $\displaystyle\sum_{k=1}^r\lambda_k\, n_k$ is indeed a bound for the number $\mu_{\underline{n}}$ of $m$'s for which $d_m$ can contribute to the coefficient of $\underline{x}^{\underline{n}}$. Indeed, this will show that $\tilde{f}\left(\underline{x},y\right)\in K[y]((x_1^{\mathbb{Z}},\ldots,x_r^\mathbb{Z}))^{\textrm{grlex}}$. But, by definition of $\tilde{f}$, one has that:
$$\tilde{f}\left(\underline{x},y\right)=y\,Q\left(\underline{x},\tilde{f}\left(\underline{x},y\right)\right)\,\in K((x_1^{\mathbb{Z}},\ldots,x_r^\mathbb{Z}))^{\textrm{grlex}}[[y]].$$
Hence, both members of this equality are in fact in $ K[y]((x_1^{\mathbb{Z}},\ldots,x_r^\mathbb{Z}))^{\textrm{grlex}}$.
So, for $y=1$, we get that $\tilde{f}\left(\underline{x},1\right)\in K((x_1^{\mathbb{Z}},\ldots,x_r^\mathbb{Z}))^{\textrm{grlex}}$ is a solution with $w\left(\tilde{f}\left(\underline{x},1\right)\right)>_{\mathrm{grlex}}\underline{0}$ of the equation: $$f(\underline{x},y)=\displaystyle\frac{y}{Q\left(\underline{x},y\right)}=1\, \Leftrightarrow\, y=Q(\underline{x},y).$$
It is equal to the unique solution $y_0$ of Theorem \ref{theo:hensel}: 
$$y_0=\tilde{f}\left(\underline{x},1\right)= \displaystyle\sum_{m\geq  1}d_m(\underline{x}).$$
  We consider the relation:
$$G(\underline{M})=\underline{n}\ \Leftrightarrow\ \left\{\begin{array}{lcl}
\displaystyle\sum_{\underline{i},j}m_{\underline{i},j}\,i_1&=&n_1;\\
&\vdots&\\
\displaystyle\sum_{\underline{i},j}m_{\underline{i},j}\,i_r&=&n_r.
\end{array}\right.$$
Let us decompose $m=|M|=\displaystyle\sum_{\underline{i},j}m_{\underline{i},j}$ as follows:
$$|M|=\displaystyle\sum_{|\underline{i}|>0}m_{\underline{i},j}+\displaystyle\sum_{|\underline{i}|=0,\, i_1>0}m_{\underline{i},j}+\cdots+ \displaystyle\sum_{|\underline{i}|=0=i_1=\cdots=i_{r-2},\, i_{r-1}>0}m_{\underline{i},j}.$$
So, the relation $G(\underline{M})=\underline{n}$ can be written as:
\begin{equation}\label{equ:borneFS}
\left\{\begin{array}{ccl}
\displaystyle\sum_{|\underline{i}|>0}m_{\underline{i},j}\,i_1+\displaystyle\sum_{|\underline{i}|=0,\, i_1>0}m_{\underline{i},j}\,i_1&=&n_1;\\
&\vdots&\\
\displaystyle\sum_{|\underline{i}|>0}m_{\underline{i},j}\,i_k+\displaystyle\sum_{|\underline{i}|=0,\, i_1>0}m_{\underline{i},j}\,i_k+\cdots+ \displaystyle\sum_{|\underline{i}|=0=i_1=\cdots=i_{k-1},\, i_{k}>0}m_{\underline{i},j}\,i_k&=&n_k;\\
&\vdots&\\
\displaystyle\sum_{\underline{i},j}m_{\underline{i},j}\,i_r&=&n_r.
\end{array}\right.\end{equation}
Firstly, let us show by induction on $k\in\{0,\ldots,r-1\}$ that:
$$\begin{array}{l}
\displaystyle\sum_{|\underline{i}|=0=i_1=\cdots=i_{k-1},\, i_{k}>0}m_{\underline{i},j}\ \ \ \ \leq\ \ \ \  \left[\iota_{0,k}\left(\displaystyle\prod_{p=2}^{k-1}(1+\iota_{0,p}) + \displaystyle\prod_{p=1}^{k-1}(1+\iota_{0,p}) \right)\right]n_1 \\
\ \ \ \ \ \ \ \ \ \ \ \ + \left[\iota_{0,k}\left(\displaystyle\prod_{p=3}^{k-1}(1+\iota_{0,p})  + \displaystyle\prod_{p=1}^{k-1}(1+\iota_{0,p}) \right)\right]n_2
+\cdots+\left[1+\iota_{0,k}\displaystyle\prod_{p=1}^{k-1}(1+\iota_{0,p}) \right]n_k\\
\ \ \ \ \ \ \ \ \ \ \ \ \ \ \ \ \ \ \ \ \ \ \ \ \ \ \ \ \ \ \ \ \ \ \ \ +\left[\iota_{0,k}\displaystyle\prod_{p=1}^{k-1}(1+\iota_{0,p})\right]n_{k+1}
+\cdots+\left[\iota_{0,k}\displaystyle\prod_{p=1}^{k-1}(1+\iota_{0,p})\right]n_r ,
\end{array}$$
the initial step $k=0$ being:
$$ \displaystyle\sum_{|\underline{i}|>0}m_{\underline{i},j}\leq n_1+\ldots+n_r.        $$
This case $k=0$ follows directly from (\ref{equ:borneFS}), by summing its $r$ relations:
$$ \displaystyle\sum_{|\underline{i}|>0}m_{\underline{i},j}\leq\displaystyle\sum_{|\underline{i}|>0}m_{\underline{i},j}|\underline{i}|\leq n_1+\ldots+n_r.         $$
Suppose that we have the desired property until some rank $k-1$. Recall that for any $\underline{i}$, $i_k\geq -\iota_{0,k}$. By the $k$'th equation in (\ref{equ:borneFS}), we have:\\

$\begin{array}{l}
\displaystyle\sum_{|\underline{i}|=0=i_1=\cdots=i_{k-1},\, i_{k}>0}m_{\underline{i},j}\, \leq \, \displaystyle\sum_{|\underline{i}|=0=i_1=\cdots=i_{k-1},\, i_{k}>0}m_{\underline{i},j}\,i_k
\end{array}$

$\ \ \ \ \ \ \ \ \ \ \ \ \ \ \ \ \ \ \ \ \ \  \begin{array}{lcl}
&\leq & n_k-\left(
\displaystyle\sum_{|\underline{i}|>0}m_{\underline{i},j}\,i_k+\displaystyle\sum_{|\underline{i}|=0,\, i_1>0}m_{\underline{i},j}\,i_k+\cdots+ \displaystyle\sum_{|\underline{i}|=0=i_1=\cdots=i_{k-2},\, i_{k-1}>0}m_{\underline{i},j}\,i_k\right)\\
&\leq & n_k+\iota_{0,k}\left(
\displaystyle\sum_{|\underline{i}|>0}m_{\underline{i},j} +\displaystyle\sum_{|\underline{i}|=0,\, i_1>0}m_{\underline{i},j} +\cdots+ \displaystyle\sum_{|\underline{i}|=0=i_1=\cdots=i_{k-2},\, i_{k-1}>0}m_{\underline{i},j} \right).
\end{array}$\\

We apply the induction hypothesis to these $k$ sums and obtain an inequality of type:
$$\displaystyle\sum_{|\underline{i}|=0=i_1=\cdots=i_{k-1},\, i_{k}>0}m_{\underline{i},j}\leq \alpha_{k,1}\,n_1+\cdots+\alpha_{k,r}\,n_r.$$
For $q>k$, let us compute:
$$\begin{array}{lcl}
\alpha_{k,q}&=&\iota_{0,k}\left( 1+  \iota_{0,1}+ \iota_{0,2}(1+\iota_{0,1})+\iota_{0,3}(1+\iota_{0,1})(1+\iota_{0,2})+\cdots + \iota_{0,k-1} \displaystyle\prod_{p=1}^{k-2}(1+\iota_{0,p})     \right)\\
&=& \iota_{0,k}\displaystyle\prod_{p=1}^{k-1}(1+\iota_{0,p}).
\end{array}$$
For $q=k$, we have the same computation, plus the contribution of the isolated term $n_k$. Hence:
$$\alpha_{k,k}=1+\iota_{0,k}\displaystyle\prod_{p=1}^{k-1}(1+\iota_{0,p}).$$
For $q<k$, we have a part of the terms leading again by the same computation to the formula $\iota_{0,k} \displaystyle\prod_{p=1}^{k-1}(1+\iota_{0,p})$. The other part consists of terms starting to appear at the rank $q$ and whose sum can be computed as:
$$\iota_{0,k}\left( 1+  \iota_{0,q+1}+ \iota_{0,q+2}(1+\iota_{0,q+1})+\cdots + \iota_{0,k-1} \displaystyle\prod_{p=q+1}^{k-2}(1+\iota_{0,p})     \right)\\
= \iota_{0,k} \displaystyle\prod_{p=q+1}^{k-1}(1+\iota_{0,p}).$$
So we obtain as desired:
$$\alpha_{k,q}= \iota_{0,k}\left[ \displaystyle\prod_{p=q+1}^{k-1}(1+\iota_{0,p})+ \displaystyle\prod_{p=1}^{k-1}(1+\iota_{0,p})\right].  $$
Subsequently, we obtain an inequality for  $m=|M|=\displaystyle\sum_{\underline{i},j}m_{\underline{i},j}$ of type:
$$\begin{array}{lcl}
m&=&\displaystyle\sum_{|\underline{i}|>0}m_{\underline{i},j}+\displaystyle\sum_{|\underline{i}|=0,\, i_1>0}m_{\underline{i},j}+\cdots+ \displaystyle\sum_{|\underline{i}|=0=i_1=\cdots=i_{r-2},\, i_{r-1}>0}m_{\underline{i},j}\\
&\leq & \beta_1\, n_1+\cdots +\beta_r\,n_r,
\end{array}$$
with  $\beta_k= 1+\displaystyle\sum_{l=1}^{r-1}\alpha_{l,k}$ for any $k$.   For $k=r$, let us compute in a similar way as before for $\alpha_{k,q}$:
$$\begin{array}{lcl}
\beta_r&=&1+\iota_{0,1}+\iota_{0,2}(1+\iota_{0,1})+\cdots +\iota_{0,k}\displaystyle\prod_{p=1}^{k-1}(1+\iota_{0,p})+\cdots +\iota_{0,r-1}\displaystyle\prod_{p=1}^{r-2}(1+\iota_{0,p})\\
&=& \displaystyle\prod_{p=1}^{r-1}(1+\iota_{0,p})=\lambda_r.
\end{array}$$
For $k=r-1$, we have the same computation plus 1 coming from the term $\alpha_{r-1,r-1}$. Hence:
$$ \beta_{r-1}=1+ \displaystyle\prod_{p=1}^{r-1}(1+\iota_{0,p})=\lambda_{r-1}.$$
For $k\in \{1,\ldots,r-2\}$, we have a part of the terms leading again by the same computation to the formula $\displaystyle\prod_{p=1}^{r-1}(1+\iota_{0,p})$. The other part consists of terms starting to appear at the rank $k$ and whose sum can be computed as:
$$1+\iota_{0,k+1}+\iota_{0,k+2}(1+\iota_{0,k+1})+\cdots+\iota_{0,r-1}\displaystyle\prod_{p=k+1}^{r-2}(1+\iota_{0,p})=\displaystyle\prod_{p=k+1}^{r-1}(1+\iota_{0,p})$$
Altogether, we obtain as desired:
$$\beta_k=\displaystyle\prod_{p=k+1}^{r-1}(1+\iota_{0,p})+\displaystyle\prod_{p=1}^{r-1}(1+\iota_{0,p})=\lambda_k.$$
\end{demo}

\begin{remark}\label{rem:FS-contraint}\indent 
\begin{enumerate}
\item Note that  for any $k\in\{1,\ldots,r-1\}$,  $\lambda_k=\lambda_r\left(\displaystyle\frac{1}{(1+\iota_{0,1})\cdots(1+\iota_{0,k})}+1\right)$, so $\lambda_1\geq \lambda_k>\lambda_r$. Thus, we obtain that:
$$ \mu_{\underline{n}}\leq \lambda_1|\underline{n}|.$$
Moreover, in the particular case where $\underline{\iota}_0=\underline{0}$ -- i.e. when $Q(\underline{x},y)\in K[\underline{x},y]$ and $y_0\in K[[\underline{x}]]$ as in \cite{sokal:implicit-function} -- we have $\lambda_k=2$ for $k\in\{1,\ldots,r-1\}$ and $\lambda_r=1$. Thus  we obtain:
$$ \mu_{\underline{n}}\leq 2|\underline{n}|-n_r\leq 2|\underline{n}|.$$
Note that :
$$|\underline{n}| \leq 2|\underline{n}|-n_r\leq 2|\underline{n}|$$
which can be related in this context with the effective bounds  $2|\underline{n}|-1$ (case\\ $w(Q(\underline{x},y))\geq_{\mathrm{grlex}}\underline{0}$) and $|\underline{n}|$ (case $w(Q(\underline{x},y))>_{\mathrm{grlex}}\underline{0}$) given in \cite[Remark 2.4]{sokal:implicit-function}.
\item With the notation from Theorem \ref{theo:formule-FS}, any strongly reduced Henselian equation $y=Q(\underline{x},y)$ can be written:
$$\underline{x}^{\underline{\iota}_0}y=\tilde{Q}(\underline{x},y)$$
$\textrm{ with }\tilde{Q}(\underline{x},y)\in K[\underline{x},y]$ and $w(\tilde{Q}(\underline{x},y))>_{\mathrm{grlex}}\underline{\iota}_0$. 
Any element $\underline{n}$ of $\mathrm{Supp}\, y_0$, being in the monoid $\mathcal{S}$ of the proof of Theorem \ref{theo:hensel}, is of the form:
$$\underline{n}=\underline{m}-k\,\underline{\iota}_0\ \ \mathrm{with}\ \underline{m}\in\mathbb{N}^r,\ k\in\mathbb{N}\ \mathrm{and}\ k\,|\underline{\iota}_0|\leq |\underline{m}|.$$
\end{enumerate}
\end{remark}

\begin{ex}\label{ex:FS-gene}
Let us consider the following example of strongly reduced Henselian equation:
$$\begin{array}{lcl}
 y &=& a_{1,-1,2}x_1{x_2}^{-1} y^2 + a_{-1,2,0}{x_1}^{-1}{x_2}^2 +a_{0,1,1}x_2y+ a_{-1,3,0}{x_1}^{-1}{x_2}^3 +a_{0,2,1}{x_2}^2y\\
 &&+\left(a_{1, 1, 0}+ a_{1,1,2}y^2\right)x_1 x_2  +a_{1,2,0} x_1{x_2}^2+a_{2,1,1}y{x_1}^2x_2\\
&&+ a_{1,3,0} x_1{x_2}^3 +a_{2,2,1} y{x_1}^2{x_2}^2+a_{3,1,2}y^2{x_1}^3x_2.
\end{array}$$
The support of the solution is included in the monoid $\mathcal{S}$ generated by the exponents of $(x_1,x_2)$, which is equal to the pairs $\underline{n}=(n_1,n_2)\in\mathbb{Z}^2$ with $n_2=-n_1+ l$ and  $n_1\geq -l$ for  $l\in\mathbb{N}$. We have $\underline{\iota}_0=(1,1)$, so $(\lambda_1,\lambda_2)=(3,2)$ and $\mu_{\underline{n}}\leq 3n_1+2n_2=n_1+2l$. We are in position to compute the first coefficients of the unique solution $y_0$. Let us give the details for the computation of the first terms, for $l=0$. In this case, to compute $c_{n_1,-n_1}$, $n_1>0$, we consider $m$ such that  $1\leq m\leq \mu_{n_1,-n_1}\leq n_1$, and $\underline{M}=(m_{\underline{i},j})_{\underline{i},j}$ such that:
$$\left\{\begin{array}{lcl}
|\underline{M}|=m & \Leftrightarrow &  \displaystyle\sum_{\underline{i},j}m_{\underline{i},j}=m\leq n_1;\\
\|\underline{M}\|=m-1 & \Leftrightarrow &  \displaystyle\sum_{\underline{i},j}m_{\underline{i},j}j=m-1\leq n_1-1;\\
G(\underline{M})=\underline{n} &\Leftrightarrow &  \left\{\begin{array}{lcl}
\displaystyle\sum_{\underline{i},j}m_{\underline{i},j}\,i_1&=&n_1>0;\\
\displaystyle\sum_{\underline{i},j}m_{\underline{i},j}\,i_2&=&-n_1<0.
\end{array}\right.
\end{array}\right.$$
The last condition implies that $m_{1,-1,2}\geq n_1$. 
But, according to the second condition, this gives $n_1-1\geq \|\underline{M}\|\geq 2\, m_{1,-1,2} \geq 2\,n_1$, a contradiction. Hence, $c_{n_1,-n_1}=0$ for any $n_1>0$.\\
In the case $l=1$, we consider the corresponding conditions to compute $c_{n_1,-n_1+1}$ for $n_1\geq -1$. We obtain that $1\leq m\leq \mu_{n_1,-n_1+1}\leq n_1+2$. Suming the two conditions in $G(\underline{M})=(n_1,-n_1+1)$, we get $m_{-1,2,0}+m_{0,1,1}=1$ and $m_{\underline{i},j}=0$ for any $\underline{i}$ such that $i_1+i_2\geq 2$. So we are left with the following linear system:
$$\left\{\begin{array}{cccccccccc}
(L_1)&m_{1,-1,2}&+&m_{-1,2,0}&+&m_{0,1,1}&=& m&\leq& n_1+2\\
(L_2)&2\,m_{1,-1,2}&+&&&m_{0,1,1}&=& m-1&\leq& n_1+1\\
(L_3)&m_{1,-1,2}&-&m_{-1,2,0}& &&=& n_1&&\\
(L_4)&-m_{1,-1,2}&+&2\,m_{-1,2,0}&+&m_{0,1,1}&=& -n_1+1&&\\
\end{array}\right.$$
By comparing $(L_2)-(L_3)$ and $(L_1)$, we get that $m=m-1-n_1$, so $n_1=-1$. Consequently, by $(L_1)$, $m=1$, and by $(L_2)$, $m_{1,-1,2}=m_{0,1,1}=0$. Since $m_{-1,2,0}+m_{0,1,1}=1$, we obtain $m_{-1,2,0}=1$ which indeed gives the only solution. Finally,  $c_{n_1,-n_1+1}=0$ for any $n_1\geq 0$ and: 
$$ c_{-1,2}=\displaystyle\frac{1}{1}\displaystyle\frac{1!}{1!0!}{a_{-1,2,0}}^1=a_{-1,2,0}.$$
Similarly, we claim that one can determine that:
$$\begin{array}{lcll}
c_{-2,4}&=&0,&\mu_{\underline{n}}\leq 2;\\
c_{-1,3}&=&a_{{-1,3,0}}+a_{{0,1,1}}a_{{-1,2,0}}+a_{{1,-1,2}}{a_{{-1,2,0}}}^{2},& \mu_{\underline{n}}\leq 3;\\
c_{0,2}&=&0,&  \mu_{\underline{n}}\leq 4;\\
c_{1,1}&=&a_{{1,1,0}},&  \mu_{\underline{n}}\leq 5;\\
c_{n_1,-n_1+2}&=&0\ \ \ \ \mathrm{for}\ \ \ \ n_1\geq 0,\ n_1\neq 1&\mu_{\underline{n}}\leq n_1+4;\\
c_{n_1,-n_1+3}&=&0\ \ \ \ \mathrm{for}\ \ \ \ -3\leq n_1\leq -2,&\mu_{\underline{n}}\leq n_1+6;\\
c_{-1,4}&=&a_{{0,2,1}}a_{{-1,2,0}}+a_{{0,1,1}}a_{{-1,3,0}}+2\,a_{{1,-1,2}}a_{{-1,2,0}}a_{{-1,3,0}} &\\
&&+{a_{{0,1,1}}}^{2}a_{{-1,2,0}}+3\,a_{{0,1,1}}a_{{1,-1,2}}{a_{{-1,2,0}}}^{2}+2\,{a_{{1,-1,2}}}^{2}{a_{{-1,2,0}}}^{3},& \mu_{\underline{n}}\leq 5;\\
&\vdots&\end{array}$$
\end{ex}

\section{Closed-form expression of an algebraic multivariate series.}\label{section:flajo-soria}

The field $K$ of coefficients has still characteristic zero. Our purpose is to determine the coefficients of an algebraic series in terms of the coefficients of a vanishing polynomial. We consider the following polynomial of degrees bounded by $d_x$ in $\underline{x}$ and by $d_y$ in $y$: 
$$ \begin{array}{lcl}
P(\underline{x},y)&=&\displaystyle\sum_{|\underline{i}|=0}^{d_x}\displaystyle\sum_{j=0}^{d_y}a_{\underline{i},j}\underline{x}^{\underline{i}}y^j ,\ \textrm{  with } P(\underline{x},y)\in K[\underline{x},y]\setminus\{0\}\\
&=& \displaystyle\sum_{|\underline{i}|=0}^{d_x}\pi_{\underline{i}}^P(y)\underline{x}^{\underline{i}}\\
&=& \displaystyle\sum_{j=0}^{d_y}a_{j}^P(\underline{x})y^j,
\end{array}$$
and a formal power series:
$$y_0=\displaystyle\sum_{\underline{n}\geq_{\mathrm{grlex}} (0,\ldots,0,1)}c_ {\underline{n}}\underline{x}^{\underline{n}},  \textrm{  with } y_0\in K[[\underline{x}]],\   c_{(0,\ldots,0,1)}\neq 0.$$
The field $K((\underline{x}))$ is endowed with the  graded lexicographic valuation $w$.\\

\begin{notation}\label{nota:P_k}
For any $\underline{k}\in\mathbb{N}^r$ and for any $Q(\underline{x},y)=\displaystyle\sum_{j=0}^da_j^Q(\underline{x})y^j\in K((x_1^{\mathbb{Z}},\ldots,x_r^{\mathbb{Z}}))^{\textrm{grlex}}[y]$, we denote: 
\begin{itemize}
\item $S(\underline{k})$ the successor element of $\underline{k}$ in $(\mathbb{N}^r,\leq_{\textrm{grlex}})$;
\item $w(Q):=\min\left\{w \left(a_j^Q(\underline{x})\right),\ j=0,..,d\right\}$;
\item $z_0:=0$ and for $\underline{k}\geq_{\textrm{grlex}} (0,\ldots,0,1)$, $z_{\underline{k}}:=\displaystyle\sum_{\underline{n}=(0,\ldots,0,1)}^{\underline{k}}c_{\underline{n}}\underline{x }^{\underline{n}}$;
\item $y_{\underline{k}}:=y_0-z_{\underline{k}}=\displaystyle\sum_{\underline{n}\geq_{\textrm{grlex}}S( \underline{k})}c_{\underline{n}}\underline{x}^{\underline{n}}$;
\item $Q_{\underline{k}}(\underline{x},y):=Q(\underline{x},z_{\underline{k}}+\underline{x}^{S(\underline{k})}y) =\displaystyle\sum_{\underline{i}\geq_{\textrm{grlex}}\underline{i}_{\underline{k}}}\pi^Q_{\underline{k},\underline{i}}(y)\underline{x}^{\underline{i}}$ where $\underline{i}_{\underline{k}}:=w( Q_{\underline{k}})$.
 Note that the sequence $(\underline{i}_{\underline{k}})_{\underline{k}\in\mathbb{N}^r}$ is nondecreasing since $Q_{S(\underline{k})}(\underline{x},y)=Q_{\underline{k}}(\underline{x},c_{S(\underline{k})}+\underline{x}^{\underline{n}}y)$ for  $\underline{n}=S^2({\underline{k}})-S({\underline{k}})>_{\textrm{grlex}}\underline{0}$, $\underline{n}\in \mathbb{Z}^r$.
\end{itemize}
\end{notation}

As for the one variable case i.e. $r=1$ (see e.g. \cite{walker_alg-curves}), we consider $y_0$ solution of the equation $P=0$ via an adaptation in several variables of the  algorithmic method of Newton-Puiseux, also with two stages:
\begin{enumerate}
\item a first stage of separation of the  solutions, which illustrates the following fact: $y_0$ may share an initial part with other roots of $P$.  But, if $y_0$ is a simple root of $P$, this step concerns \emph{only finitely many} of the first terms of $y_0$ since $w\left(\partial P/\partial y\,(\underline{x},y_0)\right)$ is finite.
\item a second stage of unique "automatic" resolution: for $y_0$ a simple root of $P$, once it has been separated from the other solutions, we will show that the remaining part of $y_0$ is a root of a strongly reduced  Henselian equation, in the sense of Definition \ref{defi:equ-hensel-red}, naturally derived from $P$ and an initial part of $y_0$.
\end{enumerate}

\begin{lemma}\label{lemme:double-simple}
\begin{itemize}
\item[(i)] The series $y_0$ is a root of  $P(\underline{x},y)$ if and only if the sequence $(\underline{i}_{\underline{k}})_{\underline{k}\in\mathbb{N}^r}$ where $\underline{i}_{\underline{k}}:=w( P_{\underline{k}})$ is strictly increasing.
\item[(ii)] The series $y_0$ is a simple root of $P(\underline{x},y)$ if and only if the sequence $(\underline{i}_{\underline{k}})_{\underline{k}\in\mathbb{N}^r}$ is strictly increasing and there exists a lowest multi-index $\underline{k}_0$ such that $\underline{i}_{S(\underline{k}_0)}=\underline{i}_{\underline{k}_0}-S(\underline{k}_0)+S^2(\underline{k}_0)$. In that case, one has that  $\underline{i}_{S(\underline{k})}=\underline{i}_{\underline{k}}-S(\underline{k})+S^2(\underline{k})=\underline{i}_{\underline{k}_0}-S(\underline{k}_0)+S^2(\underline{k})$  for any $\underline{k}\geq_{\textrm{grlex}} \underline{k}_0$. 
\end{itemize}
\end{lemma}
\begin{proof}
(i) Note that for any $\underline{k}\in\mathbb{N}^r,\ $ $\underline{i}_{\underline{k}}\leq_{\textrm{grlex}} w(P_{\underline{k}}(\underline{x},0)=w(P(\underline{x},z_{\underline{k}}))$. Hence, if the sequence $(\underline{i}_{\underline{k}})_{\underline{k}\in\mathbb{N}^r}$ is strictly increasing in $(\mathbb{N}^r,\leq_{\textrm{grlex}})$, it tends to  $+\infty$ (i.e.  $\forall \underline{n}\in\mathbb{N}^r$, $\exists \underline{k}_0\in\mathbb{N}^r$, $\forall \underline{k}\geq_{\mathrm{grlex}} \underline{k}_0$, $\underline{i}_{\underline{k}}\geq_{\mathrm{grlex}}\underline{n}$),  and so does $w(P(\underline{x},z_{\underline{k}}))$. The series $y_0$ is indeed a root of $P(\underline{x},y)$. Conversely, suppose that there exist $\underline{k}<_{\textrm{grlex}}\underline{l}$ such that $\underline{i}_{\underline{k}}\geq_{\textrm{grlex}} \underline{i}_{\underline{l}}$.
Since the sequence $(\underline{i}_{\underline{n}})_{\underline{n}\in\mathbb{N}^r}$ is nondecreasing,  one has that $\underline{i}_{\underline{l}}\geq \underline{i}_{\underline{k}}$, so  $\underline{i}_{\underline{l}}=\underline{i}_{\underline{k}}$.
We apply the multivariate Taylor's formula to $P_{\underline{j}}(\underline{x},y)$ for $\underline{j}>_{\textrm{grlex}}\underline{k}$:
\begin{equation}\label{equ:taylor}
\begin{array}{lcl}
P_{\underline{j}}(\underline{x},y)&=&P_{\underline{k}}\left(\underline{x},c_{S(\underline{k})}+ c_{S^2(\underline{k})}\underline{x}^{S^2(\underline{k})-S(\underline{k})} +\cdots+c_{\underline{j}}\underline{x}^{\underline{j}-S(\underline{k})}+\underline{x}^{S(\underline{j})-S(\underline{k})}y\right)\\
&=& \displaystyle\sum_{\underline{i}\geq_{\textrm{grlex}}\underline{i}_{\underline{k}}} \pi^P_{\underline{k},\underline{i}} \left(c_{S(\underline{k})}+ c_{S^2(\underline{k})}\underline{x}^{S²(\underline{k})-S(\underline{k})} +\cdots+\underline{x}^{S(\underline{j})-S(\underline{k})}y\right) \underline{x}^{\underline{i}}\\
&=& \pi^P_{\underline{k},\underline{i}_{\underline{k}}}(c_{S(\underline{k})})x^{\underline{i}_{\underline{k}}}+b_{S(\underline{i}_{\underline{k}})} x^{S(\underline{i}_{\underline{k}})}+ \cdots.
\end{array}
\end{equation}
Note that $b_{S(\underline{i}_{\underline{k}})}= \pi^P_{\underline{k},S(\underline{i}_{\underline{k}})}(c_{S(\underline{k})})$ or  $b_{S(\underline{i}_{\underline{k}})}=  (\pi^P_{\underline{k},\underline{i}_{\underline{k}}} )'(c_{S(\underline{k})})\, c_{S^2(\underline{k})}+\pi^P_{\underline{k},S(\underline{i}_{\underline{k}})}(c_{S(\underline{k})})$ depending on whether $S(\underline{i}_{\underline{k}})<_{\textrm{grlex}} \underline{i}_{\underline{k}}+S^2(\underline{k})-S(\underline{k})$ or $S(\underline{i}_{\underline{k}})=\underline{i}_{\underline{k}}+S^2(\underline{k})-S(\underline{k})$.
For $\underline{j}=\underline{l}$, we deduce that $\pi^P_{\underline{k},\underline{i}_{\underline{k}}}(c_{S(\underline{k})})\neq 0$. This implies that for any $\underline{j}>_{\mathrm{grlex}}\underline{k}$, $\underline{i}_{\underline{j}}=\underline{i}_{\underline{k}}$
and    $w\left(P_{\underline{j}}(\underline{x},0)\right)=w\left(P(\underline{x},z_{\underline{j}})\right)=\underline{i}_{\underline{k}}$. Hence $w\left(P(\underline{x},y_0)\right)=\underline{i}_{\underline{k}}\neq +\infty$.
\vspace{2 pt} 

\noindent (ii) The series $y_0$ is a double root of $P$ if and only if it is a root of $P$ and $\partial P/\partial y$. Let $y_0$ be a root of $P$. Let us expand the multivariate Taylor's formula (\ref{equ:taylor}) for  $\underline{j}=S(\underline{k})$: 
\begin{equation}\label{equ:p_{k+1}}
\begin{array}{l}
\begin{array}{lcl}
P_{S(\underline{k})}(\underline{x},y)
&=&  \pi^P_{\underline{k},\underline{i}_{\underline{k}}}(c_{S(\underline{k})})\underline{x}^{\underline{i}_{\underline{k}}}+ \pi^P_{\underline{k},S(\underline{i}_{\underline{k}})}(c_{S(\underline{k})})\underline{x}^{S(\underline{i}_{\underline{k}})}+\cdots\\
&& +\left[(\pi^P_{\underline{k},\underline{i}_{\underline{k}}})'(c_{S(\underline{k})})\, y+\pi^P_{\underline{k},\underline{i}_{\underline{k}}+S^2(\underline{k})-S(\underline{k})}(c_{S(\underline{k})})\right]\underline{x}^{\underline{i}_{\underline{k}}+S^2(\underline{k})-S(\underline{k})}+\cdots +
\end{array}\\
\left[\displaystyle\frac{(\pi^P_{\underline{k},\underline{i}_{\underline{k}}})''(c_{S(\underline{k})})}{2}\, y^2+(\pi^P_{\underline{k},\underline{i}_{\underline{k}}+S^2(\underline{k})-S(\underline{k})})'(c_{S(\underline{k})})\,y+\pi^P_{\underline{k},\underline{i}_{\underline{k}}+2(S^2(\underline{k})-S(\underline{k}))}(c_{S(\underline{k})})\right]\underline{x}^{\underline{i}_{\underline{k}}+2(S^2(\underline{k})-S(\underline{k}))}+\cdots
\end{array}\\
\end{equation}

\noindent Note that if $S(i_{\underline{k}})=\underline{i}_{\underline{k}}+S^2(\underline{k})-S(\underline{k})$, then there are no intermediary terms between the first one and the one with valuation $\underline{i}_{\underline{k}}+S^2(\underline{k})-S(\underline{k})$.
We have by definition of $P_{\underline{k}}$:
\begin{center}
$\displaystyle\frac{\partial P_{\underline{k}}}{\partial y}(\underline{x},y)=\underline{x}^{S(\underline{k})}\left(\displaystyle\frac{\partial P}{\partial y}\right)_{\underline{k}}(\underline{x},y)=\displaystyle\sum_{\underline{i}\geq_{\textrm{grlex}} \underline{i}_{\underline{k}}}(\pi^P_{\underline{k},\underline{i}})'(y)\underline{x}^{\underline{i}}$
\end{center}
One has that $\pi^P_{\underline{k},\underline{i}_{\underline{k}}}(y)\nequiv 0$ and $\pi^P_{\underline{k},\underline{i}_{\underline{k}}}(c_{S(\underline{k})})=0$ (see the point (i) above), so $(\pi^P_{\underline{k},\underline{i}_{\underline{k}}})'(y)\nequiv 0$. Thus: \begin{equation}\label{equ:val-DP} w\left(\left(\displaystyle\frac{\partial P}{\partial y}\right)_{\underline{k}}\right)=\underline{i}_{\underline{k}}-S(\underline{k}).
\end{equation}
 We perform the Taylor's expansion of   $\left(\displaystyle\frac{\partial P}{\partial y}\right)_{S(\underline{k})}$:
$$\begin{array}{lcl}
\left(\displaystyle\frac{\partial P}{\partial y}\right)_{S(\underline{k})}(\underline{x},y)&=& \left(\displaystyle\frac{\partial P}{\partial y}\right)_{\underline{k}}\left(\underline{x},c_{S(\underline{k})}+\underline{x}^{S^2(\underline{k})-S(\underline{k})}y\right)\\
&=&( \pi^P_{\underline{k},\underline{i}_{\underline{k}}})'(c_{S(\underline{k})})x^{\underline{i}_{\underline{k}}-S(\underline{k})}+\cdots\\
&&+ \left[(\pi^P_{\underline{k},\underline{i}_{\underline{k}}})''(c_{S(\underline{k})})\, y+(\pi^P_{\underline{k},\underline{i}_{\underline{k}}+S^2(\underline{k})-S(\underline{k})})'(c_{S(\underline{k})})\right]\underline{x}^{\underline{i}_{\underline{k}}+S^2(\underline{k})-2S(\underline{k})}+\cdots.
\end{array}$$
By the point (i) applied to  $\displaystyle\frac{\partial P}{\partial y}$, if $y_0$ is a double root $P$, we must have $ (\pi^P_{\underline{k},\underline{i}_{\underline{k}}})'(c_{S(\underline{k})})=0$. Moreover, if  $\pi^P_{\underline{k},\underline{i}}(c_{S(\underline{k})})\neq 0$ for some $\underline{i}\in\left\{S(\underline{i}_{\underline{k}}),\,\ldots\,,\,\underline{i}_{\underline{k}}+S^2(\underline{k})-S(\underline{k})\right\}$, by Formula (\ref{equ:p_{k+1}}) we would have $i_{S(\underline{k})}\leq_{\textrm{grlex}}\underline{i}_{\underline{k}}+S^2(\underline{k})-S(\underline{k})$ and even $\underline{i}_{\underline{j}}\leq_{\textrm{grlex}}\underline{i}_{\underline{k}}+S^2(\underline{k})-S(\underline{k})$ for every  $\underline{j}>_{\textrm{grlex}}\underline{k}$ according to Formula (\ref{equ:taylor}): $y_0$ could not be a root of $P$. So, $\pi^P_{\underline{k},\underline{i}}(c_{S(\underline{k})})= 0$ for $\underline{i}=S(\underline{i}_{\underline{k}}),..,\underline{i}_{\underline{k}}+S^2(\underline{k})-S(\underline{k})$, and, accordingly, $i_{S(\underline{k})}>_{\textrm{grlex}} \underline{i}_{\underline{k}}+S^2(\underline{k})-S(\underline{k})$.\\
If $y_0$ is a simple root of $P$, from the point (i) and its proof there exists a lowest  $\underline{k}_0$ such that the sequence $(\underline{i}_{\underline{k}}-S(\underline{k}))_{\underline{k}\in\mathbb{N}^r}$ is no longer strictly increasing, that is to say, such that  $(\pi^P_{\underline{k}_0,\underline{i}_{\underline{k}_0}})'(c_{S(\underline{k}_0)})\neq 0$. For any $\underline{k}\geq_{\mathrm{grlex}} \underline{k}_0$, we consider the  Taylor's expansion of $\left(\displaystyle\frac{\partial P}{\partial y}\right)_{S(\underline{k})}=\left(\displaystyle\frac{\partial P}{\partial y}\right)_{\underline{k}_0}(c_{S(\underline{k}_0)}+\cdots+\underline{x}^{S^2(\underline{k})-S(\underline{k_0 })}y)$:
\begin{equation}\label{equ:partialP}
\begin{array}{l}
\left(\displaystyle\frac{\partial P}{\partial y}\right)_{S(\underline{k})}(\underline{x},y)\,=\,(\pi^P_{\underline{k}_0,\underline{i}_{\underline{k}_0}})'(c_{S(\underline{k}_0)})x^{\underline{i}_{\underline{k}_0}-S(\underline{k}_0)}+\cdots\\
\ \ \ \ \ \  \ \ \  \ \ \ +\left[(\pi^P_{\underline{k}_0,\underline{i}_{\underline{k}_0}})''(c_{S(\underline{k}_0)})c_{S^2(k_0)}+(\pi^P_{\underline{k}_0, \underline{i}_{\underline{k}_0}+S^2(\underline{k}_0)-S(\underline{k}_0)})' (c_{S(\underline{k}_0)})\right]x^{\underline{i}_{\underline{k}_0}+ S^2(\underline{k}_0)-S(\underline{k}_0)} +\cdots
\end{array}
\end{equation}
and we get that:
\begin{equation}\label{equ:taylor-deriv}
w\left(\displaystyle\frac{\partial P}{\partial y}\left(z_{S(\underline{k})},0\right) \right)=w\left(\left(\displaystyle\frac{\partial P}{\partial y}\right)_{S(\underline{k})}(\underline{x},0)\right)=w\left(\left(\displaystyle\frac{\partial P}{\partial y}\right)_{S(\underline{k})}\right)=\underline{i}_{\underline{k}_0}-S(\underline{k}_0).
\end{equation} 
By Equation (\ref{equ:val-DP}), we obtain that $w\left(\left(\displaystyle\frac{\partial P}{\partial y}\right)_{S(\underline{k})}\right)=\underline{i}_{S(\underline{k})}-S^2(\underline{k})$. So, $\underline{i}_{S(k)}=\underline{i}_{\underline{k}_0}-S(\underline{k}_0)+S^2(\underline{k})$. As every $\underline{k}>_{\mathrm{grlex}} \underline{k}_0$ is the successor of some $\underline{k}'\geq_{\mathrm{grlex}} \underline{k}_0$, we get that for every $\underline{k}\geq_{\mathrm{grlex}} \underline{k}_0$, $\underline{i}_{\underline{k}}-S(\underline{k})=\underline{i}_{\underline{k}_0}-S(\underline{k}_0)$. So, finally, $\underline{i}_{S(k)}=\underline{i}_{\underline{k}}-S(\underline{k})+S^2(\underline{k})$ as desired.
\end{proof}

Resuming the notations of Theorem \ref{theo:wilc} and of Lemma \ref{lemme:double-simple}, the  multi-index $\underline{k}_0$ represents the length of the initial part in the stage of separation of the solutions. In the following lemma, we bound it using Lemma \ref{lemme:ordreQ} or the discriminant $\Delta_P$ of $P$. 

\begin{lemma}\label{lemme:partie-princ}
Let $y_0=\displaystyle\sum_{\underline{n}\geq_{\mathrm{grlex}} (0,\ldots,0,1)}c_ {\underline{n}}\underline{x}^{\underline{n}},\   c_{(0,\ldots,0,1)}\neq 0$, be a simple root of a nonzero polynomial $P(\underline{x},y)$ with $\deg_x (P)\leq d_x$ and $\deg_y(P)\leq d_y$.
The  multi-index $\underline{k}_0$ of Lemma \ref{lemme:double-simple} verifies that:
$$|\underline{k}_0|\leq 2d_x d_y.$$
Moreover, if $P$ has only simple roots: 
$$|\underline{k}_0|\leq d_x(2\,d_y-1).$$
 \end{lemma}
 \begin{proof}
By Lemma \ref{lemme:ordreQ}, since $P(\underline{x},y_0)=0$ and $ \displaystyle\frac{\partial P}{\partial y}(\underline{x},y_0)\neq 0$, one has that:
\begin{equation}\label{equ:derivee}
\mathrm{ord}_{\underline{x}} \displaystyle\frac{\partial P}{\partial y}(\underline{x},y_0)\leq 2d_xd_y.
\end{equation}

But, by definition of $\underline{k}_0$ and by Formula (\ref{equ:taylor-deriv}): $$w\left( \displaystyle\frac{\partial P}{\partial y}\left(\underline{x},z_{S(\underline{k})}\right)\right)=w\left(\displaystyle\frac{\partial P}{\partial y}\left(\underline{x},z_{S(\underline{k}_0)}\right)\right)=\underline{i}_{\underline{k}_0}-S(\underline{k}_0)$$ for any $\underline{k}\geq_{\mathrm{grlex}} \underline{k}_0$. So, $w\left(\displaystyle\frac{\partial P}{\partial y}(\underline{x},y_0)\right)=w\left(\displaystyle\frac{\partial P}{\partial y}(\underline{x},z_{S(\underline{k}_0)})\right)$. 
Moreover,  by minimality of $\underline{k}_0$, the sequence $(\underline{i}_{\underline{k}}-S(\underline{k}))_{\underline{k}}$ is strictly increasing up to $\underline{k}_0$, so by Formula (\ref{equ:val-DP}): $$w\left( \displaystyle\frac{\partial P}{\partial y}(\underline{x},y_0)\right)=w\left(\displaystyle\frac{\partial P}{\partial y}(\underline{x},z_{S(\underline{k}_0)})\right)=w\left(\left(\displaystyle\frac{\partial P}{\partial y}\right)_{S(\underline{k}_0)}(\underline{x},0)\right)\geq_{\mathrm{grlex}} w\left(\left(\displaystyle\frac{\partial P}{\partial y}\right)_{S(\underline{k}_0)}\right)\geq_{\mathrm{grlex}} \underline{k}_0.$$
So:
$$\left|\underline{k}_0\right|\leq \left|w\left( \displaystyle\frac{\partial P}{\partial y}(\underline{x},y_0)\right)\right|=\mathrm{ord}_{\underline{x}} \displaystyle\frac{\partial P}{\partial y}(\underline{x},y_0).$$
 Hence we obtain by (\ref{equ:derivee}) as desired that:
$$|\underline{k}_0| \leq 2d_{\underline{x}}d_y.$$ \\
In the case where  $P$  has only  simple roots, as in the proof of   Lemma \ref{lemme:ordreQ},  $\mathrm{ord}_{\underline{x}} \displaystyle\frac{\partial P}{\partial y}(\underline{x},y_0)$ is bounded by the degree of the resultant of $P$ and $\displaystyle\frac{\partial P}{\partial y}$, say the discriminant $\Delta_P$  of $P$, which is bounded by $d_x(2\,d_y-1)$. 
 \end{proof}

\begin{notation}\label{nota:omega_0}
	Resuming Notation \ref{nota:P_k} and the content of Lemma \ref{lemme:double-simple}, we set: $$\omega_0:=(\pi^P_{\underline{k}_0,\underline{i}_{\underline{k}_0}})'(c_{S(\underline{k}_0)}).$$
	By Formula (\ref{equ:partialP}), we note that $$
	\left(\displaystyle\frac{\partial P}{\partial y}\right)(\underline{x},y_0)=\omega_0\,x^{\underline{i}_{\underline{k}_0}-S(\underline{k}_0)}+\cdots.$$
	Thus, $\omega_0$ is the  initial coefficient of $\left(\displaystyle\frac{\partial P}{\partial y}\right)(\underline{x},y_0)$ with respect to $\leq_{\mathrm{grlex}}$, hence $\omega_0\neq 0$.
\end{notation}

\begin{theo}\label{theo:FS}
Consider the following nonzero polynomial in $K\left[\underline{x},y\right]$ of total degree in $\underline{x}$ bounded by $d_x$ and degree in $y$ bounded by $d_y$: 
$$ P(\underline{x},y)=\displaystyle\sum_{|\underline{i}|=0}^{d_x}\displaystyle\sum_{j=0}^{d_y}a_{\underline{i},j}\underline{x}^{\underline{i}}y^j = \displaystyle\sum_{\underline{i}\geq_{\mathrm{grlex}} \underline{0}} \pi^P_{\underline{i}}(y)\underline{x}^{\underline{i}},$$
and a formal power series which is a simple root:
$$y_0=\displaystyle\sum_{\underline{n}>_{\mathrm{grlex}} \underline{0}}c_{\underline{n}}\underline{x}^{\underline{n}}\ \in K\left[\left[\underline{x}\right]\right],\   c_{(0,\ldots,0,1)}\neq 0.$$
Resuming Notations \ref{nota:P_k} and \ref{nota:omega_0} and the content of Lemma \ref{lemme:double-simple}, recall that \\ $\omega_0:=(\pi^P_{\underline{k}_0,\underline{i}_{\underline{k}_0}})'(c_{S(\underline{k}_0)})\neq 0$.
Then, for any $\underline{k}>_{\mathrm{grlex}}\underline{k}_0$:
\begin{itemize}
\item either the polynomial $z_{S(\underline{k})}=\displaystyle\sum_{\underline{n}=(0,\ldots,0,1)}^{S(\underline{k})}c_{\underline{n}}\underline{x}^{\underline{n}}$ is a solution of $P(\underline{x},y)=0$;
\item or the multivariate Laurent polynomial $ _{\underline{k}}R(\underline{x},y):=\displaystyle\frac{P_{\underline{k}}(\underline{x},y+c_{S(\underline{k})})}{-\omega_0\underline{x}^{\underline{i}_{\underline{k}}}}=-y+\, _{\underline{k}}Q(\underline{x},y)$  defines a strongly reduced Henselian equation:
$$ y=\, _{\underline{k}}Q(\underline{x},y)$$
as in Definition \ref{defi:equ-hensel-red} and satisfied by:
$$t_{S(\underline{k})}:=\frac{y_0-z_{S(\underline{k})}}{\underline{x}^{S(\underline{k})}}=c_{S^2(\underline{k})}\underline{x}^{S^2(\underline{k})-S(\underline{k})}+c_{S^3(\underline{k})}\underline{x}^{S^3(\underline{k})-S(\underline{k})}+\cdots.$$
\end{itemize} 
\end{theo}
\begin{proof}
We show by induction on $\underline{k}\in(\mathbb{N}^r,\leq_{\mathrm{grlex}})$, $\underline{k}>_{\mathrm{grlex}}\underline{k}_0$,  that $_{\underline{k}}R(x,y)=-y+ \, _{\underline{k}}Q(x,y)$ with $_{\underline{k}}Q(\underline{x},y)\in  K\left[x_1,x_1^{-1},\ldots,x_r,x_r^{-1}\right]\left[y\right],$
such that $w\left( _{\underline{k}}Q\left(\underline{x},y\right)\right) >_{\textrm{grlex}}\underline{0}$.  Let us apply  Formula  (\ref{equ:p_{k+1}}) with parameter $\underline{k}=\underline{k}_0$. Since $\underline{i}_{S(\underline{k}_0)}=\underline{i}_{\underline{k}_0}+S^2(\underline{k}_0)-S(\underline{k}_0)$, we have that  $\pi^P_{\underline{k}_0,\underline{i}}(c_{S(\underline{k}_0)})=0$ for $\underline{i}_{\underline{k}_0}\leq_{\mathrm{grlex}} \underline{i}<_{\mathrm{grlex}}i_{\underline{k}_0}+S^2(\underline{k}_0)-S(\underline{k}_0)$,  and accordingly:
$$P_{S(\underline{k}_0)}(x,y)=\left[\omega_0\,y+\pi^P_{\underline{k}_0,\underline{i}_{\underline{k}_0}+S^2(\underline{k}_0)-S(\underline{k}_0)}(c_{S(\underline{k}_0)})\right]x^{\underline{i}_{\underline{k}_0}+S^2(\underline{k}_0)-S(\underline{k}_0)}+\, _{S(\underline{k}_0)}T(\underline{x},y)$$
where $_{S(\underline{k}_0)}T(\underline{x},y)\in K[\underline{x},y]$ with $w\left( _{S(\underline{k}_0)}T(\underline{x},y)\right)>_{\mathrm{grlex}}\underline{i}_{\underline{k}_0}+S^2(\underline{k}_0)-S(\underline{k}_0)$.
Since $\underline{i}_{S^2(\underline{k}_0)}=\underline{i}_{\underline{k}_0}+S^3(\underline{k}_0)-S(\underline{k}_0)>_{\mathrm{grlex}}i_{\underline{k}_0}+S^2(\underline{k}_0)-S(\underline{k}_0)$, we obtain that: $$\pi^P_{S(\underline{k}_0),i_{\underline{k}_0}+S^2(\underline{k}_0)-S(\underline{k}_0)}(y)=\omega_0\,y+\pi^P_{\underline{k}_0,i_{\underline{k}_0}+S^2(\underline{k}_0)-S(\underline{k}_0)}(c_{S(\underline{k}_0)})$$ vanishes at $c_{S^2(\underline{k}_0)}$, which implies that $$c_{S^2(\underline{k}_0)}= \displaystyle\frac{-\pi^P_{\underline{k}_0,\underline{i}_{\underline{k}_0}+S^2(\underline{k}_0)-S(\underline{k}_0)}(c_{S(\underline{k}_0)})}{\omega_0}.$$ Computing $_{S(\underline{k}_0)}R(\underline{x},y)$, it follows that:
\begin{center}
$_{S(\underline{k}_0)}R(\underline{x},y)=-y+\, _{S(\underline{k}_0)}Q(\underline{x},y) $,\end{center} \textrm{ with }$_{S(\underline{k}_0)}Q(\underline{x},y)=\displaystyle\frac{_{S(\underline{k}_0)}T(\underline{x},y +c_{S^2(\underline{k}_0)})}{-\omega_0\underline{x}^{i_{\underline{k}_0}+S^2(\underline{k}_0)-S(\underline{k}_0)}}$.
So $_{S(\underline{k}_0)}Q(\underline{x},y)\in K\left[x_1,x_1^{-1},\ldots,x_r,x_r^{-1}\right]\left[y\right]$ with $w\left( _{S(\underline{k}_0)}Q(\underline{x},y)\right)>_{\mathrm{grlex}} \underline{0}$. 

Now suppose that the property holds true at a rank $\underline{k}\geq_{\mathrm{grlex}}S(\underline{k}_0)$, which means that $_{\underline{k}}R(\underline{x},y):=\displaystyle\frac{P_{\underline{k}}(\underline{x},y+c_{S(\underline{k})})}{-\omega_0\underline{x}^{\underline{i}_{\underline{k}}}}=-y+\, _{\underline{k}}Q(\underline{x},y)$. Therefore, for $_{\underline{k}}\tilde{Q}(\underline{x},y)=-\omega_0\, _{\underline{k}}Q(\underline{x},y-c_{S(\underline{k})})\in K\left[x_1,x_1^{-1},\ldots,x_r,x_r^{-1}\right]\left[y\right]$ which is such that $w\left( _{\underline{k}}\tilde{Q}\right)  >_{\textrm{grlex}}\underline{0}$, we can write:
$$\begin{array}{lcl}
 P_{\underline{k}}(\underline{x},y)&=&\omega_0(y-c_{S(\underline{k})})\underline{x}^{\underline{i}_{\underline{k}}}+ \underline{x}^{\underline{i}_{\underline{k}}} \cdot\, _{\underline{k}}\tilde{Q}(\underline{x},y)\\
 &=&\pi^P_{\underline{k},\underline{i}_{\underline{k}}}(y)\underline{x}^{\underline{i}_{\underline{k}}}+\pi^P_{\underline{k},S(\underline{i}_{\underline{k}})}(y)x^{S(\underline{i}_{\underline{k}})}+ \cdots.
\end{array}$$
Since $P_{S(\underline{k})}(\underline{x},y)= P_{\underline{k}}\left(\underline{x},c_{S(\underline{k})}+\underline{x}^{S^2(\underline{k})-S(\underline{k})}y\right)$ and $\underline{i}_{S(\underline{k})}=\underline{i}_{\underline{k}}+S^2(\underline{k})-S(\underline{k})$, we have that:
$$ P_{S(\underline{k})}(\underline{x},y)=\left[\omega_0\,y+\pi^P_{\underline{k},\underline{i}_{\underline{k}}+S^2(\underline{k})-S(\underline{k})}(c_{S(\underline{k})})\right]\underline{x}^{\underline{i}_{\underline{k}}+S^2(\underline{k})-S(\underline{k})}+\pi^P_{S(\underline{k}),S(\underline{i}_{S(\underline{k})})}(y)x^{S(\underline{i}_{S(\underline{k})})}+\cdots.$$
But $i_{S^2(\underline{k})}=\underline{i}_{S(\underline{k})}+S^3(\underline{k})-S^2(\underline{k}) >_{\mathrm{grlex}} \underline{i}_{S(\underline{k})}=\underline{i}_{\underline{k}}+S^2(\underline{k})-S(\underline{k})$. So we must have $\pi^P_{S(\underline{k}),\underline{i}_{S(\underline{k})}}(c_{S^2(\underline{k})})=0$, i.e. $c_{S^2(\underline{k})}=\displaystyle\frac{-\pi^P_{\underline{k},\underline{i}_{\underline{k}}+S^2(\underline{k})-S(\underline{k})}(c_{S(\underline{k})})}{\omega_0}$. It follows that:
$$P_{S(\underline{k})}(\underline{x},y)=\omega_0\left(y-c_{S^2(\underline{k})}\right)\underline{x}^{\underline{i}_{\underline{k}}+S^2(\underline{k})-S(\underline{k})}+\pi^P_{S(\underline{k}),S(\underline{i}_{S(\underline{k})})}(y)x^{S(\underline{i}_{S(\underline{k})})}+\cdots,$$
Hence:
$$\begin{array}{lcl}
_{S(\underline{k})}R(\underline{x},y)&=&-y-
\displaystyle\frac{\pi^P_{S(\underline{k}),S(\underline{i}_{S(\underline{k})})}(y+c_{S^2(\underline{k})})}{\omega_0}x^{S(\underline{i}_{S(\underline{k})})-\underline{i}_{S(\underline{k})}}+ \cdots\\
&=&-y+\, _{S(\underline{k})}Q(\underline{x},y),\ \ \ \ _{S(\underline{k})}Q\in K\left[x_1,x_1^{-1},\ldots,x_r,x_r^{-1}\right]\left[y\right],
\end{array} $$
with $w\left( _{\underline{k}}Q\left(\underline{x},y\right)\right) >_{\textrm{grlex}}\underline{0}$ as desired.

To conclude the proof, it suffices to note that the equation $ _{\underline{k}}R(\underline{x},y)=0$ is strongly reduced Henselian if and only if  $ _{\underline{k}}Q\left(\underline{x},0\right)\nequiv 0$, which is equivalent to $z_{S(\underline{k})}$ not being a root of $P$.
\end{proof}

 We will need the following lemma:
 \begin{lemma}\label{lemme:coincidence}
 Let $y_0\in K[[\underline{x}]]$ be a simple root of a nonzero polynomial $P(\underline{x},y)$ of degrees $\deg_x(P)\leq d_x$ and $\deg_y(P)\leq d_y$. For any other root $y_1\neq y_0$ of $P$, one has that:
 $$\mathrm{ord}_{\underline{x}}\, (y_0-y_1)\leq 2\,d_xd_y.$$
 \end{lemma}
\begin{proof}
Note that the hypothesis imply that $d_y\geq 2$. Let us write $y_1-y_0=\delta_{1,0}$ and $\underline{k}:=w(y_1-y_0)=w(\delta_{1,0})\in \mathbb{N}^r$. By Taylor's Formula, we have:
$$\begin{array}{lcl}
P(\underline{x},y_0+\delta_{1,0})&=&0\\
&=& P(\underline{x},y_0)+\displaystyle\frac{\partial P}{\partial y}(\underline{x},y_0) \delta_{1,0}+\cdots+\displaystyle\frac{1}{d_y!}\displaystyle\frac{\partial^{d_y} P}{\partial y^{d_y}}(\underline{x},y_0){\delta_{1,0}}^{d_y}\\
&=&\delta_{1,0}\left(\displaystyle\frac{\partial P}{\partial y}(\underline{x},y_0)+\cdots+\displaystyle\frac{1}{d_y!}\displaystyle\frac{\partial^{d_y} P}{\partial y^{d_y}}(\underline{x},y_0){\delta_{1,0}}^{d_y-1}\right).
\end{array}$$
Since $\delta_{1,0}\neq 0$ and $\displaystyle\frac{\partial P}{\partial y}(\underline{x},y_0)\neq 0$, one has that:
$$\displaystyle\frac{\partial P}{\partial y}(\underline{x},y_0)=-\delta_{1,0}\left(\displaystyle\frac{1}{2}\displaystyle\frac{\partial^2 P}{\partial y^2}(\underline{x},y_0)+\cdots+\displaystyle\frac{1}{d_y!}\displaystyle\frac{\partial^{d_y} P}{\partial y^{d_y}}(\underline{x},y_0){\delta_{1,0}}^{d_y-2}\right)$$
The valuation of the right hand side being at least $\underline{k}$, we obtain that:
$$ w\left(\displaystyle\frac{\partial P}{\partial y}(\underline{x},y_0)\right)\geq_{\mathrm{grlex}} \underline{k}.$$
But, by Lemma \ref{lemme:ordreQ}, we must have $\mathrm{ord}_x\left(\displaystyle\frac{\partial P}{\partial y}(\underline{x},y_0)\right)\leq 2d_xd_y$. So $|\underline{k}|\leq 2d_xd_y$.
\end{proof}

For the courageous reader, in the case where $y_0$ is a series which is not a polynomial, we deduce from Theorem \ref{theo:FS} and from the generalized  Flajolet-Soria's Formula \ref{theo:formule-FS} a closed-form expression for the coefficients of $y_0$ in terms of the coefficients $a_{\underline{i},j}$ of $P$ and of the coefficients of an initial part  $z_{\underline{k}}$ of $y_0$ sufficiently large, in particular for any $\underline{k}\in\mathbb{N}^r$ such that $|\underline{k}|\geq 2d_xd_y+1$. Recall that $i_{\underline{k}}=w\left( P_{\underline{k}}(\underline{x},y)\right)$. Note that for such $\underline{k}$, since $y_0$ is not a polynomial, by Lemma \ref{lemme:coincidence}, $z_{S(\underline{k})}$ cannot be a root of $P$.

\begin{coro}\label{coro:FS}
Let $\underline{k}\in\mathbb{N}^r$ be such that $|\underline{k}|\geq 2d_xd_y+1$. For any $\underline{p}>_{\mathrm{grlex}} S(\underline{k})$, consider $\underline{n}:=\underline{p}-S(\underline{k})$. Then:
$$c_{\underline{p}}=c_{S(\underline{k})+\underline{n}}=\displaystyle\sum_{q=1}^{\mu_{\underline{n}}} \displaystyle\frac{1}{q}\left(\displaystyle\frac{-1}{\omega_0}\right)^q\displaystyle\sum_{|\underline{S}|=q,\ \|\underline{S}\|\geq q-1}\underline{A}^{\underline{S}}\left(\displaystyle\sum_{|\underline{T}_{\underline{S}}|=\|\underline{S}\|-q+1 \atop  G(\underline{T}_{\underline{S}})=\underline{n}+q\underline{i}_{\underline{k}}-(q-1)S(\underline{k})-G(\underline{S})}e_{\underline{T}_{\underline{S}}}\underline{C}^{\underline{T}_{\underline{S}}}\right),$$
where $\mu_{\underline{n}}$ is as in Theorem \ref{theo:formule-FS} for the equation $ y=\, _{\underline{k}}Q(\underline{x},y)$ of Theorem \ref{theo:FS}, $S=(s_{\underline{i},j})$, $\underline{A}^{\underline{S}}=\displaystyle\prod_{|\underline{i}|=0,\ldots,d_x\, ,\ j=0,\ldots,d_y}a_{\underline{i},j}^{s_{\underline{i},j}}$, $\underline{T}_{\underline{S}}=(t_{\underline{S},\underline{i}})$,  $\underline{C}^{\underline{T}_{\underline{S}}}=\displaystyle\prod_{\underline{i}=(0,\ldots,0,1)}^{S(\underline{k})}c_{\underline{i}}^{t_{\underline{S},\underline{i}}}$, 
 and $e_{\underline{T}_{\underline{S}}}\in\mathbb{N}$ is of the form:\\

\noindent $\ \ e_{\underline{T}_{\underline{S}}}=$\\

$\displaystyle\sum_{\left(n^{\underline{l},m}_{\underline{i},j,\underline{L}}\right)} \displaystyle\frac{q!}{\displaystyle\prod_{{\underline{l} =S(\underline{i}_{\underline{k}})-\underline{i}_{\underline{k}},\ldots,\atop d_yS(\underline{k})+(d_x,0,\ldots,0)-\underline{i}_{\underline{k}}}  \atop m=0,\ldots,m_{\underline{l}}}\displaystyle\prod_{|\underline{i}|=0,\ldots,d_x \atop j=m,\ldots,d_y}\displaystyle\prod_{|\underline{L}|=j-m \atop G(\underline{L})=\underline{l}+\underline{i}_{\underline{k}}-mS(\underline{k})-\underline{i}}n^{\underline{l},m}_{\underline{i},j,\underline{L}}!} \displaystyle\prod_{{\underline{l}=S(\underline{i}_{\underline{k}})-\underline{i}_{\underline{k}},\ldots,\atop d_y S(\underline{k})+(d_x,0,\ldots,0)-\underline{i}_{\underline{k}}} \atop m=0,\ldots,m_{\underline{l}}}\displaystyle\prod_{|\underline{i}|=0,\ldots,d_x \atop j=m,\ldots,d_y}\displaystyle\prod_{|\underline{L}|=j-m \atop G(\underline{L})=\underline{l}+\underline{i}_{\underline{k}}-mS(\underline{k})-\underline{i}}\left(\displaystyle\frac{j!}{m!\,L!}\right)^{n^{\underline{l},m}_{\underline{i},j,\underline{L}}},$\\
where we denote $m_{\underline{l}}:=\min\left\{d_y,\ \max\left\{m\in\mathbb{N}\, / \, mS(\underline{k})\leq_{\mathrm{grlex}}\underline{l }+\underline{i}_{\underline{k}}\right\}\right\}$,\\ $\underline{L}=\underline{L}_{\underline{i},j}^{\underline{l},m}=\left(l_{\underline{i},j,(0,\ldots,0,1)}^{\underline{l},m},\ldots,l_{\underline{i},j,S(\underline{k})}^{\underline{l},m}\right)$, 
 and where the sum is taken over the set of tuples  $\left(n^{\underline{l},m}_{\underline{i},j,\underline{L}}\right)_{\underline{l}= S(\underline{i}_{\underline{k}})-\underline{i}_{\underline{k}},\ldots,d_yS(\underline{k})+(d_x,0,\ldots,0)-\underline{i}_{\underline{k}},\ m=0,\ldots,m_{\underline{l}} \atop |\underline{i}|=0,\ldots,d_x,\ j=m,\ldots,d_y,\ |\underline{L}|=j-m,\  G(\underline{L})=\underline{l}+\underline{i}_{\underline{k}}-mS(\underline{k})-\underline{i}}$ such that:
\begin{center}
 $\displaystyle\sum_{\underline{l},m}\displaystyle\sum_{\underline{L}} n^{\underline{l},m}_{\underline{i},j,\underline{L}}=s_{\underline{i},j}$,  $\ \ \displaystyle\sum_{\underline{l},m}\displaystyle\sum_{\underline{i},j} \displaystyle\sum_{\underline{L}}n^{\underline{l},m}_{\underline{i},j,\underline{L}}=q\ \ \ $ and $\ \ \ \displaystyle\sum_{\underline{l},m}\displaystyle\sum_{\underline{i},j} \displaystyle\sum_{\underline{L}}n^{\underline{l},m}_{\underline{i},j,\underline{L}}\underline{L}= \underline{T}_{\underline{S}}$.
\end{center}
\end{coro}
\begin{remark}\label{rem:coeff-e_T}
Note that the coefficients $e_{\underline{T}_{\underline{S}}}$ are indeed natural numbers, since they are sums of products of multinomial coefficients because  $\displaystyle\sum_{\underline{l},m}\displaystyle\sum_{\underline{i},j}\displaystyle\sum_{\underline{L}} n^{\underline{l},m}_{\underline{i},j,\underline{L}}=q$ and $m+|\underline{L}|=j$. In fact,  $\displaystyle\frac{1}{q}e_{\underline{T}_{\underline{S}}}\in\mathbb{N}$ by Remark \ref{rem:entier} as we will see along the proof.
\end{remark}
\begin{proof}
We get started by computing the coefficients of $\omega_0\underline{x}^{\underline{i}_{\underline{k}}}\, _{\underline{k}}R$,  in order to get those of $\, _{\underline{k}}Q$:
$$\begin{array}{lcl}
-\omega_0\underline{x}^{\underline{i}_{\underline{k}}}\, _{\underline{k}}R&=&P_{\underline{k}}(\underline{x},\,y+c_{S(\underline{k})})\\
&=&P(\underline{x},z_{S(\underline{k})}+x^{S(\underline{k})}y)\\
&=& \displaystyle\sum_{|\underline{i}|=0,\ldots,d_x\, ,\ j=0,\ldots,d_y}a_{\underline{i},j}\underline{x}^{\underline{i}}\left(z_{S(\underline{k})}+\underline{x }^{S(\underline{k})}y\right)^{j}\\
&=& \displaystyle\sum_{|\underline{i}|=0,\ldots,d_x\, ,\ j=0,\ldots,d_y}a_{\underline{i},j}\underline{x}^{\underline{i}}\displaystyle\sum_{m=0}^{j}\displaystyle\frac{j!}{m!\,(j-m)!}z_{S(\underline{k})}^{j-m}\underline{x}^{mS(\underline{k})}y^m.
\end{array}$$
For $L=(l_{(0,\ldots,0,1) },\cdots,l_{S(\underline{k})})$, we denote 
$\underline{C}^{\underline{L}}:=c_{(0,\ldots,0,1)}^{l_{(0,\ldots,0,1)}}\cdots c_{S(\underline{k})}^{l_{S(\underline{k})}}$. One  has that:
$$z_{S(\underline{k})}^{j-m}=\displaystyle\sum_{|\underline{L}|=j-m}\displaystyle\frac{(j-m)!}{\underline{L}!}\underline{C}^{\underline{L}}\underline{x}^{G(\underline{L})}.$$
So:
$$ -\omega_0\underline{x}^{\underline{i}_{\underline{k}}}\, _{\underline{k}}R=\displaystyle\sum_{m=0}^{d_y} \displaystyle\sum_{|\underline{i}|=0,\ldots,d_x \atop j=m,\ldots,d_y}a_{\underline{i},j}\displaystyle\sum_{|\underline{L}|=j-m}\displaystyle\frac{j!}{m!\,\underline{L}!}\underline{C}^{\underline{L}}\underline{x}^{G(\underline{L}) +mS(\underline{k})+\underline{i}}\,y^m.$$
We set $\underline{\hat{l}}=G(\underline{L})+mS(\underline{k})+\underline{i}$, which ranges between $mS(\underline{k})$ and $(d_y-m)S(\underline{k})+mS(\underline{k})+(d_x,0,\ldots,0)=d_yS(\underline{k})+(d_x,0,\ldots,0)$. Thus: 
$$-\omega_0\underline{x}^{\underline{i}_{\underline{k}}}\, _{\underline{k}}R=\displaystyle\sum_{m=0,\ldots,d_y \atop \hat{l}=mS(\underline{k}),\ldots,d_yS(\underline{k})+(d_x,0,\ldots,0)} \displaystyle\sum_{|\underline{i}|=0,\ldots,d_x \atop j=m,\ldots,d_y}a_{\underline{i},j}\displaystyle\sum_{|\underline{L}|=j-m \atop G(\underline{L})=\underline{\hat{l}}-mS(\underline{k})-\underline{i}}\displaystyle\frac{j!}{m!\,\underline{L}!}\underline{C}^{\underline{L}}\underline{x}^{\underline{\hat{l}}}y^m.$$
Since $_{\underline{k}}R(\underline{x},y)=-y+\, _{\underline{k}}Q(\underline{x},y)$ with $w(\, _{\underline{k}}Q(\underline{x},y))>_{\mathrm{grlex}} \underline{0}$, the coefficients of $_{\underline{k}}Q$ are obtained for $\underline{\hat{l}}=S(\underline{i}_{\underline{k}}),\ldots,d_yS(\underline{k})+(d_x,0,\ldots,0)$ \footnote{Note that our assumptions ensure that $ _{\underline{k}}Q(\underline{x},y)\neq 0$, so $S(\underline{i}_{\underline{k}}) \leq_{\mathrm{grlex}} d_yS(\underline{k})+(d_x,0,\ldots,0)$.}. We set $\underline{l}:=\underline{\hat{l}}-\underline{i}_{\underline{k}}$
 and  $$m_{\underline{l}}:=\min\left\{d_y,\ \max\left\{m\in\mathbb{N}\, / \, mS(\underline{k})\leq_{\mathrm{grlex}}\underline{l }+\underline{i}_{\underline{k}}\right\}\right\}.$$
We obtain: $$_{\underline{k}}Q(\underline{x},y)=\displaystyle\sum_{\underline{l}=S(\underline{i}_{\underline{k}})-\underline{i}_{\underline{k}},\ldots,d_yS(\underline{k})+(d_x,0,\ldots,0)-\underline{i}_{\underline{k}} \atop m=0,\ldots,m_{\underline{l}}}b_{\underline{l},m}\underline{x}^{\underline{l}}y^m,$$
with:
$$b_{\underline{l},m}=\displaystyle\frac{-1}{\omega_0}\displaystyle\sum_{|\underline{i}|=0,\ldots,d_x \atop j=m,\ldots,d_y}a_{\underline{i},j}\displaystyle\sum_{|\underline{L}|=j-m \atop G(\underline{L})=\underline{l}+\underline{i}_{\underline{k}}-mS(\underline{k})-\underline{i}}\displaystyle\frac{j!}{m!\,\underline{L}!}\underline{C}^{\underline{L}}.$$
According to Lemma \ref{lemme:partie-princ}, Theorem \ref{theo:FS} and Lemma \ref{lemme:coincidence}, we are in position to apply  the generalized Flajolet-Soria's Formula of Theorem \ref{theo:formule-FS}  in order to compute the coefficients of the solution $t_{S(\underline{k})}=c_{S^2(\underline{k})}\underline{x}^{S^2(\underline{k})-S(\underline{k})}+c_{S^3(\underline{k})}\underline{x}^{S^3(\underline{k})-S(\underline{k})}+\cdots$.  Thus, denoting $\underline{B}:=(b_{\underline{l},m})$, $\underline{Q}:=(q_{\underline{l},m})$ and $\underline{B}^{\underline{Q}}:=\displaystyle\prod_{\underline{l},m} b_{\underline{l},m}^{q_{\underline{l},m}}$  for $\underline{l}=S(\underline{i}_{\underline{k}})-\underline{i}_{\underline{k}},\ldots,d_yS(\underline{k})+(d_x,0,\ldots,0)-\underline{i}_{\underline{k}}$ and $m=0,\ldots,m_{\underline{l}}$, we obtain: 
$$c_{S(\underline{k})+\underline{n}}=\displaystyle\sum_{q=1}^{\mu_{\underline{n}}}\displaystyle\frac{1}{q}\displaystyle\sum_{|\underline{Q}|=q,\,\|\underline{Q}\|=q-1 ,\, G(\underline{Q})=\underline{n}}\displaystyle\frac{q!}{\underline{Q}!}\underline{B}^{\underline{Q}}.$$
As in Remark \ref{rem:entier}, we have $\displaystyle\frac{1}{q}\cdot\displaystyle\frac{q!}{\underline{Q}!}\in\mathbb{N}$.
Let us compute:
\begin{equation}\label{equ:b_lm}
\begin{array}{lcl}
b_{\underline{l},m}^{q_{\underline{l},m}}&=&\left(\displaystyle\frac{-1}{\omega_0}\right)^{q_{\underline{l},m}}\left(\displaystyle\sum_{|\underline{i}|=0,\ldots,d_x \atop j=m,\ldots,d_y}a_{\underline{i},j}\displaystyle\sum_{|\underline{L}|=j-m \atop G(\underline{L})=\underline{l} +\underline{i}_{\underline{k}}-mS(\underline{k})-\underline{i}}\displaystyle\frac{j!}{m!\,\underline{L}!}\underline{C}^{\underline{L}}\right)^{q_{\underline{l},m}}\\
&=&\left(\displaystyle\frac{-1}{\omega_0}\right)^{q_{\underline{l},m}}\displaystyle\sum_{|\underline{M}_{\underline{l},m}|=q_{\underline{l},m}}\displaystyle\frac{q_{\underline{l},m}!}{\underline{M}_{\underline{l},m}!} \underline{A}^{\underline{M}_{\underline{l},m}} \displaystyle\prod_{|\underline{i}|=0,\ldots,d_x \atop j=m,\ldots,d_y}\left(\displaystyle\sum_{|\underline{L}|=j-m \atop G(\underline{L})=\underline{l}+\underline{i}_{\underline{k}}-mS(\underline{k})- \underline{i}}\displaystyle\frac{j!}{m!\,\underline{L}!}\underline{C}^{\underline{L}}\right)^{m^{\underline{l},m}_{\underline{i},j}}\\
&& \textrm{ where } \underline{M}_{\underline{l},m}=(m^{\underline{l},m}_{\underline{i},j})\textrm{ for } |\underline{i}|=0,\ldots,d_x,\  j=0,\ldots,d_y\textrm{ and }m^{\underline{l},m}_{\underline{i},j}=0\textrm{ for }j<m.
\end{array}
\end{equation}
Note that, in the previous formula, $\left(-\omega_0\right)^{q_{\underline{l},m}}b_{\underline{l},m}^{q_{\underline{l},m}}$ is the evaluation at $\underline{A}$ and $\underline{C}$ of a polynomial with coefficients in $\mathbb{N}$. Since  $\displaystyle\frac{1}{q}\cdot\displaystyle\frac{q!}{\underline{Q}!}\in\mathbb{N}$, the expansion of $\left(-\omega_0\right)^{q}\displaystyle\frac{1}{q}\cdot\displaystyle\frac{q!}{\underline{Q}!}\underline{B}^{\underline{Q}}$ as a polynomial in $\underline{A}$ and $\underline{C}$ will only have natural numbers as coefficients.

Let us expand the expression $\displaystyle\prod_{|\underline{i}|=0,\ldots,d_x \atop j=m,\ldots,d_y}\left(\displaystyle\sum_{|\underline{L}|=j-m \atop G(\underline{L})=\underline{l}+\underline{i}_{\underline{k}}-mS(\underline{k})-\underline{i}}\displaystyle\frac{j!}{m!\,\underline{L}!}\underline{C}^{\underline{L}}\right)^{m^{\underline{l},m}_{\underline{i},j}}$.
For each $(\underline{l},m,\underline{i},j)$, we enumerate the terms $\displaystyle\frac{j!}{m!\,\underline{L}!}\underline{C}^{\underline{L}}$ with $u=1,\ldots,\alpha_{\underline{i},j}^{\underline{l},m}$. Subsequently:
$$\begin{array}{lcl}
\left(\displaystyle\sum_{|\underline{L}|=j-m \atop G(\underline{L})=\underline{l}+\underline{i}_{\underline{k}}-mS(\underline{k})-\underline{i}}\displaystyle\frac{j!}{m!\,\underline{L}!}\underline{C}^{\underline{L}}\right)^{m^{\underline{l},m}_{\underline{i},j}}&=& \left(\displaystyle\sum_{u=1}^{\alpha_{\underline{i},j}^{\underline{l},m}}\displaystyle\frac{j!}{m!\,\underline{L}_{\underline{i},j,u}^{\underline{l},m}!} \underline{C}^{\underline{L}_{\underline{i},j,u}^{\underline{l},m}}\right)^{ m^{\underline{l},m}_{\underline{i},j}}\\
&=& \displaystyle\sum_{|\underline{N}^{\underline{l},m}_{\underline{i},j}|=m^{\underline{l},m}_{\underline{i},j}}\displaystyle\frac{m^{\underline{l},m}_{\underline{i},j}!}{\underline{N}^{\underline{l},m}_{\underline{i},j}!}\left( \displaystyle\prod_{u=1}^{\alpha_{\underline{i},j}^{\underline{l},m}} \left(\displaystyle\frac{j!}{m!\,\underline{L}_{\underline{i},j,u}^{\underline{l},m}!}\right)^{ n^{\underline{l},m}_{\underline{i},j,u}}\right) \underline{C}^{\sum_{u=1}^{\alpha^{\underline{l},m}_{\underline{i},j}} n^{\underline{l},m}_{\underline{i},j,u} \underline{L}_{\underline{i},j,u}^{\underline{l},m}},
\end{array} $$
where $\underline{N}^{\underline{l},m}_{\underline{i},j}= \left(n^{\underline{l},m}_{\underline{i},j,u}\right)_{u=1,\ldots,\alpha_{\underline{i},j}^{\underline{l},m}}$, %
 $\, \underline{N}^{\underline{l},m}_{\underline{i},j}!= \displaystyle\prod_{u=1}^{\alpha_{\underline{i},j}^{\underline{l},m}} n^{\underline{l},m}_{\underline{i},j,u}!$. 
 Denoting $ \underline{U}_{\underline{l},m}:= \displaystyle\sum_{|\underline{i}|=0,\ldots,d_x \atop j=m,\ldots,d_y}\displaystyle\sum_{u=1}^{ \alpha_{\underline{i},j}^{\underline{l},m}}n^{\underline{l},m}_{\underline{i},j,u} \underline{L}_{\underline{i},j,u}^{\underline{l},m}$, one computes:
\begin{equation}\label{equ:Ulm}
\begin{array}{lcl}
|\underline{U}_{\underline{l},m}|&=&\displaystyle\sum_{|\underline{i}|=0,\ldots,d_x \atop j=m,\ldots,d_y}\displaystyle\sum_{u=1}^{\alpha_{\underline{i},j}^{\underline{l},m}} n^{\underline{l},m}_{\underline{i},j,u}|\underline{L}_{\underline{i},j,u}^{\underline{l},m}|\\
&=&\displaystyle\sum_{|\underline{i}|=0,\ldots,d_x \atop j=m,\ldots,d_y}\left(\displaystyle\sum_{u=1}^{\alpha_{\underline{i},j}^{\underline{l},m}} n^{\underline{l},m}_{\underline{i},j,u}\right)(j-m)\\
&=&\displaystyle\sum_{|\underline{i}|=0,\ldots,d_x \atop j=m,\ldots,d_y}m^{\underline{l},m}_{\underline{i},j}(j-m)\\
&=& \|\underline{M}_{\underline{l},m}\|-m\,q_{\underline{l},m}.
\end{array} \end{equation} 
Likewise, one computes:
\begin{equation}\label{equ:Ulm2}\begin{array}{lcl}
G(\underline{U}_{\underline{l},m})&=&\displaystyle\sum_{|\underline{i}|=0,\ldots,d_x \atop j=m,\ldots,d_y}\displaystyle\sum_{u=1}^{\alpha_{\underline{i},j}^{\underline{l},m}} n^{\underline{l},m}_{\underline{i},j,u}G(\underline{L}_{\underline{i},j,u}^{\underline{l},m})\\
&=&\displaystyle\sum_{|\underline{i}|=0,\ldots,d_x \atop j=m,\ldots,d_y}\left(\displaystyle\sum_{u=1}^{\alpha_{\underline{i},j}^{\underline{l},m}} n^{\underline{l},m}_{\underline{i},j,u}\right)(\underline{l}+\underline{i}_{\underline{k}}-mS(\underline{k})-\underline{i})\\
&=&\displaystyle\sum_{|\underline{i}|=0,\ldots,d_x \atop j=m,\ldots,d_y}m^{\underline{l},m}_{\underline{i},j}(\underline{l}+\underline{i}_{\underline{k}}-mS(\underline{k})-\underline{i})\\
&=& q_{\underline{l},m}[\underline{l}+\underline{i}_{\underline{k}}-mS(\underline{k})]-G(\underline{M}_{\underline{l},m}).
\end{array} \end{equation} 
So, according to Formula (\ref{equ:b_lm}) and the new way of writing the expression\\ $\displaystyle\prod_{|\underline{i}|=0,\ldots,d_x \atop j=m,\ldots,d_y}\left(\displaystyle\sum_{|\underline{L}|=j-m \atop G(\underline{L})=\underline{l}+\underline{i}_{\underline{k}}-mS(\underline{k})-\underline{i}}\displaystyle\frac{j!}{m!\,\underline{L}!}\underline{C}^{\underline{L}}\right)^{m^{\underline{l},m}_{\underline{i},j}}$, we obtain:
\begin{center}
$\begin{array}{lcl}
b_{\underline{l},m}^{q_{\underline{l},m}}&=& \left(\displaystyle\frac{-1}{\omega_0}\right)^{q_{\underline{l},m}}\displaystyle\sum_{|\underline{M}_{\underline{l},m}|=q_{\underline{l},m}}\underline{A}^{\underline{M}_{\underline{l},m}}\displaystyle\sum_{|\underline{U}_{\underline{l},m}|=\|\underline{M}_{\underline{l},m}\|-m\,q_{\underline{l},m} \atop G(\underline{U}_{\underline{l},m})=q_{\underline{l},m}[\underline{l}+\underline{i}_{\underline{k}}-mS(\underline{k})]-G(\underline{M}_{\underline{l},m})} d_{\underline{U}_{\underline{l},m}}\underline{C}^{\underline{U}_{\underline{l},m}}\\
&& \textrm{ with }d_{\underline{U}_{\underline{l},m}}:=\displaystyle\sum_{ \left(\underline{N}^{\underline{l},m}_{\underline{i},j}\right)}\displaystyle\frac{q_{\underline{l},m}!}{\displaystyle\prod_{|\underline{i}|=0,\ldots,d_x \atop j=m,\ldots,d_y}\underline{N}^{\underline{l},m}_{\underline{i},j}!}\displaystyle\prod_{|\underline{i}|=0,\ldots,d_x \atop j=m,\ldots,d_y}\displaystyle\prod_{u=1}^{\alpha_{\underline{i},j}^{\underline{l},m}} \left(\displaystyle\frac{j!}{m!\,\underline{L}_{\underline{i},j,u}^{\underline{l},m}!}\right)^{n^{\underline{l},m}_{\underline{i},j,u}},
 \end{array}$
 \end{center} 
 \textrm{ where the sum is taken over }$$\left\{\left(\underline{N}^{\underline{l},m}_{\underline{i},j}\right)_{|\underline{i}|=0,\ldots,d_x \atop j=m,\ldots,d_y}\textrm{ such that }|\underline{N}^{\underline{l},m}_{\underline{i},j}|=m^{\underline{l},m}_{\underline{i},j}\textrm{ and }\displaystyle\sum_{|\underline{i}|=0,\ldots,d_x \atop j=m,\ldots,d_y}\sum_{u=1}^{\alpha_{\underline{i},j}^{\underline{l},m}} n^{\underline{l},m}_{\underline{i},j,u} \underline{L}_{\underline{i},j,u}^{\underline{l},m}=\underline{U}_{\underline{l},m}\right\}.$$
(Note that, if the latter set is empty, then $d_{\underline{U}_{\underline{l},m}}=0$.)\\

We deduce that:
$$\begin{array}{lcl}
\underline{B}^{\underline{Q}}&=&\displaystyle\prod_{\underline{l}=S(\underline{i}_{\underline{k}})- \underline{i}_{\underline{k}},\ldots,d_yS(\underline{k})+(d_x,0,\ldots,0)-\underline{i}_{\underline{k}}\atop m=0,\ldots,m_{\underline{l}}}b_{\underline{l},m}^{q_{\underline{l},m}}\\
&=& \left(\displaystyle\frac{-1}{\omega_0}\right)^{q}\displaystyle\prod_{\underline{l},m} \left[\displaystyle\sum_{|\underline{M}_{\underline{l},m}|=q_{\underline{l},m}}\underline{A}^{\underline{M}_{\underline{l},m}}\displaystyle\sum_{|\underline{U}_{\underline{l},m}|=\|\underline{M}_{\underline{l},m}\|-m\,q_{\underline{l},m} \atop \|\underline{U}_{\underline{l},m}\|=q_{\underline{l},m}\left(\underline{l}+\underline{i}_{\underline{k}}-mS(\underline{k})\right)-G(\underline{M}_{\underline{l},m})}d_{\underline{U}_{\underline{l},m}} \underline{C}^{\underline{U}_{\underline{l},m}}\right].
\end{array}$$ 
Now, in order to expand the latter product of sums, we consider the corresponding sets: $$\mathcal{S}_{\underline{Q}}:= \left\{\displaystyle\sum_{\underline{l},m}\underline{M}_{\underline{l},m}\ \ / \ \   \exists (\underline{M}_{\underline{l},m})\ \mathrm{s.t.}\ |\underline{M}_{\underline{l},m}|=q_{\underline{l},m}\textrm{ and }\forall \underline{l},m,\  m^{\underline{l},m}_{\underline{i},j}=0\textrm{ for }j<m\right\}$$ and, for any   $\underline{S}\in\mathcal{S}_{\underline{Q}}$,
\begin{center} $ \mathcal{U}_{\underline{Q},\underline{S}}:=\left\{\left(\underline{U}_{\underline{l},m}\right)\ \  / \ \  \exists (\underline{M}_{\underline{l},m})\ \mathrm{s.t.}\ |\underline{M}_{\underline{l},m}|=q_{\underline{l},m}\textrm{ and }\forall \underline{l},m,\  m^{\underline{l},m}_{\underline{i},j}=0\textrm{ for }j<m,\  \displaystyle\sum_{\underline{l},m}\underline{M}_{\underline{l},m}=\underline{S},\right.$\\
	$\ \ \ \ \ \ \ \ \    \ \ \  \ \  \    \left. |\underline{U}_{\underline{l},m}|=\|\underline{M}_{\underline{l},m}\|-m\,q_{\underline{l},m} \textrm{ and }G(\underline{U}_{\underline{l},m})=q_{\underline{l},m}\left(\underline{l}+\underline{i}_{\underline{k}}-mS(\underline{k})\right)-G(\underline{M}_{\underline{l},m})  \displaystyle\frac{}{}\right\}$\end{center}
and
\begin{center} $\mathcal{T}_{\underline{Q},\underline{S}}:=\left\{\displaystyle\sum_{\underline{l},m} \underline{U}_{\underline{l},m} \ \  / \ \  \left(\underline{U}_{\underline{l},m}\right)\in  \mathcal{U}_{\underline{Q},\underline{S}} \right\}.$\end{center}
 We have:
\begin{equation}\label{equ:BQ}
 \begin{array}{lcl}
\underline{B}^{\underline{Q}}&=& \left(\displaystyle\frac{-1}{\omega_0}\right)^{q}\displaystyle\sum_{\underline{S}\in\mathcal{S}_{\underline{Q}}} \underline{A}^{\underline{S}}\displaystyle\sum_{\underline{T}_{\underline{S}} \in\mathcal{T}_{\underline{Q},\underline{S}}} \left(\displaystyle\sum_{ \left(\underline{U}_{\underline{l},m}\right)\in  \mathcal{U}_{\underline{Q},\underline{S}} \atop \sum_{\underline{l},m} \underline{U}_{\underline{l},m}=\underline{T}_{\underline{S}}}\displaystyle\prod_{\underline{l},m} d_{\underline{U}_{\underline{l},m}}\right)       \underline{C}^{\underline{T}_{\underline{S}}}\\
&=&\left(\displaystyle\frac{-1}{\omega_0}\right)^{q}\displaystyle\sum_{\underline{S}\in\mathcal{S}_{\underline{Q}}} \underline{A}^{\underline{S}}\displaystyle\sum_{\underline{T}_{\underline{S}} \in\mathcal{T}_{\underline{Q},\underline{S}}}e_{\underline{Q},\underline{T}_{\underline{S}}} \underline{C}^{\underline{T}_{\underline{S}}}.\end{array}\end{equation}
\textrm{ where }: $$e_{\underline{Q},\underline{T}_{\underline{S}}}:= \displaystyle\sum_{\left(\underline{N}^{\underline{l},m}_{\underline{i},j}\right)} \displaystyle\frac{\displaystyle\prod_{\underline{l},m}q_{\underline{l},m}!}{\displaystyle\prod_{\underline{l},m}\displaystyle\prod_{\underline{i},j} \underline{N}^{\underline{l},m}_{\underline{i},j}!} \displaystyle\prod_{\underline{l},m} \displaystyle\prod_{\underline{i},j}\displaystyle\prod_{u}\left(\displaystyle\frac{j!}{m!\,L_{\underline{i},j,u}^{\underline{l},m}!}\right)^{n^{\underline{l},m}_{\underline{i},j,u}}$$
 and  where the previous sum is taken over:  
 \begin{center}$ \mathcal{E}_{\underline{Q},\underline{T}_{\underline{S}}}:=\left\{\left( \underline{N}^{\underline{l},m}_{\underline{i},j}\right)_{\underline{l}= S(\underline{i}_{\underline{k}})-\underline{i}_{\underline{k}},\ldots,d_yS(\underline{k})+(d_x,0,\ldots,0)-\underline{i}_{\underline{k}},\, m=0,\ldots,m_{\underline{l}} \atop |\underline{i}|=0,\ldots,d_x,\ j=m,\ldots,d_y}\ \ /\ \  \forall \underline{i},j,\ \displaystyle\sum_{\underline{l},m} \sum_{u=1}^{\alpha_{\underline{i},j}^{\underline{l},m}}n^{\underline{l},m}_{\underline{i},j,u}=s_{\underline{i},j}, \right.$\\
$\left. \ \ \ \ \ \  \ \ \ \ \  \ \ \ \ \ \ \ \  \  \ \  \ \ \ \ \ \ \ \  \ \ \ \ \ \  \ \ \ \ \ \ \ \ \ \  \forall \underline{l},m,\ \displaystyle\sum_{\underline{i},j}|\underline{N}^{\underline{l},m}_{\underline{i},j}|=q_{\underline{l},m}, \textrm{ and }  \displaystyle\sum_{\underline{l},m} \displaystyle\sum_{\underline{i}, j}\sum_{u=1}^{\alpha_{\underline{i},j}^{\underline{l},m}} n^{\underline{l},m}_{\underline{i},j,u} \underline{L}_{\underline{i},j,u}^{\underline{l},m} =\underline{T}_{\underline{S}} \right\}.$\end{center}
(Note that, if the latter set is empty, then $e_{\underline{Q},\underline{T}_{\underline{S}}}=0$.)\\

Observe that $\displaystyle\frac{1}{q}\displaystyle\frac{q!}{Q!}e_{\underline{Q},\underline{T}_{\underline{S}}}$ lies in $\mathbb{N}$ as a coefficient of $(-{\omega_0}) ^{q}\displaystyle\frac{1}{q}\displaystyle\frac{q!}{Q!}B^{Q}$ as seen before.
Note also that, for any $\underline{Q}$ and for any $\underline{S}\in\mathcal{S}_{\underline{Q}}$,  $|\underline{S}|=\displaystyle\sum_{\underline{l},m}q_{\underline{l},m}=q$ and $\|\underline{S}\|\geq \displaystyle\sum_{\underline{l},m}mq_{\underline{l},m}=\|\underline{Q}\|=q-1$. Moreover, for any $\underline{T}_{\underline{S}}\in\mathcal{T}_{\underline{Q},\underline{S}}$: 
$$\begin{array}{lcl}
|\underline{T}_{\underline{S}}|&=&\displaystyle\sum_{\underline{l},m} \|\underline{M}_{\underline{l},m}\|-m\,q_{\underline{l},m}\\
 &=&\|\underline{S}\|-\|\underline{Q}\|\\
&=&\|\underline{S}\|-q+1
\end{array}$$ and:
$$\begin{array}{lcl}
G(\underline{T}_{\underline{S}})&=& \displaystyle\sum_{l,m}q_{l,m}\left(\underline{l}+\underline{i}_{\underline{k}}-mS(\underline{k})\right)-G(\underline{M}_{\underline{l},m})\\
&=&G(\underline{Q})+|\underline{Q}|\,\underline{i}_{\underline{k}}-\|\underline{Q}\|\,S(\underline{k})-G(\underline{S})\\
&=&\underline{n}+q\,\underline{i}_{\underline{k}}-(q-1)\,S(\underline{k})-G(\underline{S}).
\end{array}$$ Let us show that:
\begin{equation}\label{equ:BQAS}
  \begin{array}{lcl}
\displaystyle\sum_{|\underline{Q}|=q,\,\|\underline{Q}\|=q-1,\, G(\underline{Q})=\underline{n}}\displaystyle\frac{q!}{\underline{Q}!}\underline{B}^{\underline{Q}}&=&  \left(\displaystyle\frac{-1}{\omega_0}\right)^{q}\displaystyle\sum_{|\underline{S}|=q,\, \|\underline{S}\|\geq q-1}\underline{A}^{\underline{S}}\displaystyle\sum_{|\underline{T}_{\underline{S}}|=\|\underline{S}\|-q+1 \atop G(\underline{T}_{\underline{S}})=\underline{n}+q\underline{i}_{\underline{k}}-(q-1)S(\underline{k})-G(\underline{S})}e_{\underline{T}_{\underline{S}}}\underline{C}^{\underline{T}_{\underline{S}}},\end{array} 
\end{equation} 
\textrm{ where } $e_{\underline{T}_{\underline{S}}}:=\displaystyle\sum_{ \left(\underline{N}^{\underline{l},m}_{\underline{i},j}\right)}\displaystyle\frac{q!}{\displaystyle\prod_{\underline{l},m}\displaystyle\prod_{\underline{i},j} \underline{N}^{\underline{l},m}_{\underline{i},j}!}\displaystyle\prod_{\underline{l},m} \displaystyle\prod_{\underline{i},j}\displaystyle\prod_{u}\left(\displaystyle\frac{j!}{m!\,\underline{L}_{\underline{i},j,u}^{\underline{l},m}!}\right)^{n^{\underline{l},m}_{\underline{i},j,u}}$  and  \textrm{ where the sum is taken over }

\begin{center}
 $\mathcal{E}_{\underline{T}_{\underline{S}}}:=\left\{\left(\underline{N}^{\underline{l},m}_{\underline{i},j}\right)_{\underline{l}= S(\underline{i}_{\underline{k}})-\underline{i}_{\underline{k}},\ldots,d_yS(\underline{k})+(d_x,0,\ldots,0)-\underline{i}_{\underline{k}},\, m=0,\ldots,m_{\underline{l}} \atop |\underline{i}|=0,\ldots,d_x,\ j=m,\ldots,d_y}\textrm{ s.t. }\displaystyle\sum_{\underline{l},m}\displaystyle\sum_{u} n^{\underline{l},m}_{\underline{i},j,u}=s_{\underline{i},j},\ 
\displaystyle\sum_{\underline{l},m}\displaystyle\sum_{\underline{i},j}|\underline{N}^{\underline{l},m}_{\underline{i},j}|=q\right.$\\
$ \  \ \ \ \ \ \ \ \  \  \ \  \ \ \ \ \ \ \ \  \ \ \ \ \ \  \ \ \ \ \ \ \ \ \ \  \  \ \ \ \ \ \ \ \  \  \ \  \ \ \ \ \ \ \ \  \ \ \ \ \ \  \ \ \ \ \ \ \  \ \ \ \ \ \ \  \ \ \  \left.\textrm{ and } \displaystyle\sum_{\underline{l},m}\displaystyle\sum_{\underline{i},j} \sum_{u=1}^{\alpha_{\underline{i},j}^{\underline{l},m}}n^{\underline{l},m}_{\underline{i},j,u} \underline{L}_{\underline{i},j,u}^{\underline{l},m}=\underline{T}_{\underline{S}}\right\}.$	 
\end{center} 
(Note that, if the latter set is empty, then $e_{ \underline{T}_{\underline{S}}}=0$.)\\
Recall that  $\underline{N}^{\underline{l},m}_{\underline{i},j}!= \displaystyle\prod_{u=1}^{\alpha_{\underline{i},j}^{\underline{l},m}} n^{\underline{l},m}_{\underline{i},j,u}!$ and that the $\underline{L}^{\underline{l},m}_{\underline{i},j,u}$'s enumerate the $\underline{L}$'s such that $|\underline{L}|=j-m$ and $G(\underline{L})=\underline{l}+\underline{i}_{\underline{k}}-m\,S(\underline{k})-\underline{i}$ for given $\underline{l},m,\underline{i},j$.

Let us consider  $\underline{S}$ and  $\underline{T}_{\underline{S}}$ such that $|\underline{S}|=q,\, \|\underline{S}\|\geq q-1$, $|\underline{T}_{\underline{S}}|=\|\underline{S}\|-q+1,\  G(\underline{T}_{\underline{S}})=\underline{n}+q\underline{i}_{\underline{k}}-(q-1)S(\underline{k})-G(\underline{S})$ and such that $\mathcal{E}_{T_S}\neq \emptyset$. Take an element $( n^{l,m}_{i,j,u})\in \mathcal{E}_{\underline{T}_{\underline{S}}}$. Define $ m^{\underline{l},m}_{\underline{i},j}:=\displaystyle\sum_{u=1}^{\alpha_{\underline{i},j}^{\underline{l},m}} n^{\underline{l},m}_{\underline{i},j,u}$ for each $\underline{i},\,j,\,\underline{l},\,m$ with $j\geq m$, and $m^{\underline{l},m}_{\underline{i},j}:=0$ if $j<m$. Set $\underline{M}_{\underline{l},m}:=(m^{\underline{l},m}_{\underline{i},j})_{\underline{i},j}$ for each $\underline{l},\,m$. So, $\displaystyle\sum_{\underline{l},m}m^{\underline{l},m}_{\underline{i},j}=\displaystyle\sum_{\underline{l},m}\sum_{u=1}^{\alpha_{\underline{i},j}^{\underline{l},m}} n^{\underline{l},m}_{\underline{i},j,u}=s_{\underline{i},j}$, and $\underline{S}=\displaystyle\sum_{\underline{l},m}\underline{M}_{\underline{l},m}$. Define $q_{\underline{l},m}:=\displaystyle\sum_{\underline{i},j}m^{\underline{l},m}_{\underline{i},j}=|\underline{M}_{\underline{l},m}|$ for each $\underline{l},\,m$, and $\underline{Q}:=(q_{\underline{l},m})$. Let us show that $|\underline{Q}|=q$, $ G(\underline{Q})=\underline{n}$ and $\|\underline{Q}\|=q-1$. By definition of $\mathcal{E}_{\underline{T}_{\underline{S}}}$, $$|\underline{Q}|:=\displaystyle\sum_{\underline{l},m}q_{\underline{l},m}= \displaystyle\sum_{\underline{l},m}\displaystyle\sum_{\underline{i},j} \sum_{u=1}^{\alpha_{\underline{i},j}^{\underline{l},m}}n^{\underline{l},m}_{\underline{i},j,u}=q.$$
Recall that $\|\underline{Q}\|:=\displaystyle\sum_{\underline{l},m}mq_{\underline{l},m}$. We have:
$$\begin{array}{ll}
&|\underline{T}_{\underline{S}}|= \left|\displaystyle\sum_{\underline{l},m}\displaystyle\sum_{\underline{i},j} \sum_{u=1}^{\alpha_{\underline{i},j}^{\underline{l},m}}n^{\underline{l},m}_{\underline{i},j,u}\underline{L}_{\underline{i},j,u}^{\underline{l},m}\right|=\|\underline{S}\| -q+1\\
\Leftrightarrow & \displaystyle\sum_{\underline{l},m}\displaystyle\sum_{\underline{i},j} \sum_{u=1}^{\alpha_{\underline{i},j}^{\underline{l},m}}n^{\underline{l},m}_{\underline{i},j,u}|\underline{L}_{\underline{i},j,u}^{\underline{l},m}|=\displaystyle\sum_{\underline{i},j}js_{\underline{i},j}-q+1\\
\Leftrightarrow & \displaystyle\sum_{\underline{l},m}\displaystyle\sum_{\underline{i},j} \sum_{u=1}^{\alpha_{\underline{i},j}^{\underline{l},m}}n^{\underline{l},m}_{\underline{i},j,u}(j-m)= \displaystyle\sum_{\underline{i},j}js_{\underline{i},j}-q+1\\
\Leftrightarrow &\displaystyle\sum_{\underline{i},j} j\displaystyle\sum_{\underline{l},m} \sum_{u=1}^{\alpha_{\underline{i},j}^{\underline{l},m}}n^{\underline{l},m}_{\underline{i},j,u}- \displaystyle\sum_{\underline{l},m}m\displaystyle\sum_{\underline{i},j} \sum_{u=1}^{\alpha_{\underline{i},j}^{\underline{l},m}}n^{\underline{l},m}_{\underline{i},j,u}= \displaystyle\sum_{\underline{i},j}js_{\underline{i},j}-q+1\\
\Leftrightarrow &\displaystyle\sum_{\underline{i},j} js_{\underline{i},j}-\displaystyle\sum_{\underline{l},m}mq_{\underline{l},m} =\displaystyle\sum_{\underline{i},j}js_{\underline{i},j}-q+1\\
\Leftrightarrow & \|\underline{Q}\|=q-1.
\end{array} $$
Recall that $ G(\underline{Q}):=\displaystyle\sum_{\underline{l},m}q_{\underline{l},m}\underline{l}$. We have:
$$\begin{array}{ll}
&G(\underline{T}_{\underline{S}})= G\left(\displaystyle\sum_{\underline{l},m}\displaystyle\sum_{\underline{i},j} \sum_{u=1}^{\alpha_{\underline{i},j}^{\underline{l},m}}n^{\underline{l},m}_{\underline{i},j,u}\underline{L}_{\underline{i},j,u }^{\underline{l},m}\right)=\underline{n}+q\,\underline{i}_{\underline{k}} -(q-1)S(\underline{k}) -G(\underline{S})\\
\Leftrightarrow & \displaystyle\sum_{\underline{l},m}\displaystyle\sum_{\underline{i},j} \sum_{u=1}^{\alpha_{\underline{i},j}^{\underline{l},m}}n^{\underline{l},m}_{\underline{i},j,u}G(\underline{L}_{\underline{i},j,u}^{\underline{l},m})=\underline{n}+q\,\underline{i}_{\underline{k}} -(q-1)S(\underline{k}) -G(\underline{S})\\
\Leftrightarrow & \displaystyle\sum_{\underline{l},m}\displaystyle\sum_{\underline{i},j} \sum_{u=1}^{\alpha_{\underline{i},j}^{\underline{l},m}}n^{\underline{l},m}_{\underline{i},j,u}(\underline{l}+\underline{i}_{\underline{k}} -mS(\underline{k}) -\underline{i})= \underline{n}+q\, \underline{i}_{\underline{k}} -(q-1)S(\underline{k}) - G(\underline{S})\\
\Leftrightarrow & \begin{array}{l}
\displaystyle\sum_{\underline{l},m}\underline{l}\displaystyle\sum_{\underline{i},j} \sum_{u=1}^{\alpha_{\underline{i},j}^{\underline{l},m}}n^{\underline{l},m}_{\underline{i},j,u}+
\underline{i}_{\underline{k}} \displaystyle\sum_{\underline{l},m}\displaystyle\sum_{\underline{i},j} \sum_{u=1}^{\alpha_{\underline{i},j}^{\underline{l},m}}n^{\underline{l},m}_{\underline{i},j,u}
-S(\underline{k}) \displaystyle\sum_{\underline{l},m}m\displaystyle\sum_{\underline{i},j} \sum_{u=1}^{\alpha_{\underline{i},j}^{\underline{l},m}}n^{\underline{l},m}_{\underline{i},j,u}
 \\     
\ \ \ \ \ \ \ \ \ \ \ \ \ \ \  \ \ \ \ \ \ \ \ \ \ \ \ \ \ \ \ \  \ \ \ \ \ \ \ \ \ \ \ \ \ \ \  -\displaystyle\sum_{\underline{i},j}\underline{i} \displaystyle\sum_{\underline{l},m} \sum_{u=1}^{\alpha_{\underline{i},j}^{\underline{l},m}}n^{\underline{l},m}_{\underline{i},j,u}=\underline{n}+q\, \underline{i}_{\underline{k}} -(q-1)S(\underline{k}) -G(\underline{S}) \\

\end{array}\\
\Leftrightarrow & \displaystyle\sum_{\underline{l},m}q_{\underline{l},m}\underline{l}+q\,\underline{i}_{\underline{k}}-S(\underline{k}) \displaystyle\sum_{\underline{l},m}m\, q_{\underline{l},m}-\displaystyle\sum_{\underline{i},j} s_{i,j}\underline{i}= \underline{n}+q\, \underline{i}_{\underline{k}} -(q-1)S(\underline{k}) -G(\underline{S}) \\
\Leftrightarrow & G(\underline{Q})+q\, \underline{i}_{\underline{k}}-\|Q\|S(\underline{k})-G(\underline{S})=\underline{n}+q\, \underline{i}_{\underline{k}} -(q-1)S(\underline{k}) -G(\underline{S}).
\end{array} $$
Since $\|\underline{Q}\|=q-1$, we deduce that $G(\underline{Q})=\underline{n}$ as desired. So, $\underline{S}\in \mathcal{S}_{\underline{Q}}$ for $\underline{Q}$ as in the left-hand side of (\ref{equ:BQAS}).\\
Now, set $\underline{U}_{\underline{l},m}:=\displaystyle\sum_{\underline{i},j} \sum_{u=1}^{\alpha_{\underline{i},j}^{\underline{l},m}}n^{\underline{l},m}_{\underline{i},j,u} \underline{L}_{\underline{i},j,u}^{\underline{l},m}$, so $\displaystyle\sum_{\underline{l},m}\underline{U}_{\underline{l},m}=\underline{T}_{\underline{S}}$. Let us show that $(\underline{U}_{\underline{l},m})\in \mathcal{U}_{\underline{Q},\underline{S}}$, which implies that $\underline{T}_{\underline{S}}\in \mathcal{T}_{\underline{Q},\underline{S}}$ as desired. The existence of $(\underline{M}_{\underline{l},m})$ such that $|\underline{M}_{\underline{l},m}|=q_{\underline{l},m}\textrm{ and } \  m^{\underline{l},m}_{\underline{i},j}=0\textrm{ for }j<m$ and $\displaystyle\sum_{\underline{l},m}\underline{M}_{\underline{l},m}=\underline{S}$ follows by construction. Conditions $|\underline{U}_{\underline{l},m}|=\|\underline{M}_{\underline{l},m}\|-m\,q_{\underline{l},m}  \textrm{ and }G(\underline{U}_{\underline{l},m})=q_{\underline{l},m}[\underline{l}+\underline{i}_{\underline{k}}-mS(\underline{k})]-G(\underline{M}_{\underline{l},m})$ are obtained exactly as in (\ref{equ:Ulm}) and (\ref{equ:Ulm2}). This shows that $(n^{\underline{l },m}_{\underline{i},j,u}) \in \mathcal{E}_{\underline{Q},\underline{T}_{\underline{S}}}$, so:
$$ \mathcal{E}_{\underline{T}_{\underline{S}}}\subseteq \bigcupdot_{|\underline{Q}|=q,\, G(\underline{Q})=\underline{n},\,\|Q\|=q-1}  \mathcal{E}_{\underline{Q},\underline{T}_{\underline{S}}}. $$
The reverse inclusion holds trivially since $|\underline{Q}|=q$, so:
$$ \mathcal{E}_{\underline{T}_{\underline{S}}} = \bigcupdot_{|\underline{Q}|=q,\, G(\underline{Q})=\underline{n},\,\|\underline{Q}\|=q-1}  \mathcal{E}_{\underline{Q},\underline{T}_{\underline{S}}}. $$
We deduce that:
$$ e_{\underline{T}_{\underline{S}}}=\displaystyle\sum_{|\underline{Q}|=q,\, G(\underline{Q})=\underline{n},\,\|\underline{Q}\|=q-1}\displaystyle\frac{q!}{\underline{Q}!} e_{\underline{Q},\underline{T}_{\underline{S}}}.$$
We conclude that any term occuring in the right-hand side of (\ref{equ:BQAS}) comes from a term from the left-hand side. 

Conversely, for any $\underline{Q}$ as in the  left-hand side of  Formula (\ref{equ:BQAS}), $\underline{S}\in\mathcal{S}_{\underline{Q}}$ and $\underline{T}_{\underline{S}} \in \mathcal{T}_{\underline{Q},\underline{S}}$ verify the following conditions:
$$|\underline{S}|=q,\ \ \|\underline{S}\|\geq q-1,\ \ |\underline{T}_{\underline{S}}|=\|\underline{S}\|-q+1 ,\ \ \|\underline{T}_{\underline{S}}\|=\underline{n}+q\,\underline{i}_{\underline{k}}-(q-1)S(\underline{k})-G(\underline{S})$$
and
$$ \mathcal{E}_{\underline{T}_{\underline{S}}} = \bigcupdot_{|\underline{Q}|=q,\, G(\underline{Q})=\underline{n},\,\|\underline{Q}\|=q-1}  \mathcal{E}_{\underline{Q},\underline{T}_{\underline{S}}},\ \ \ \  e_{\underline{T}_{\underline{S}}}=\displaystyle\sum_{|\underline{Q}|=q,\, G(\underline{Q})=\underline{n},\,\|\underline{Q}\|=q-1}\displaystyle\frac{q!}{\underline{Q}!} e_{\underline{Q},\underline{T}_{\underline{S}}}. $$
Hence, any term occuring in the expansion of $\underline{B}^{\underline{Q}}$ contributes to the right hand side of Formula (\ref{equ:BQAS}).

Thus we obtain  Formula (\ref{equ:BQAS}) from which the statement of Corollary \ref{coro:FS} follows. Note also that:
$$ \displaystyle\frac{1}{q}e_{\underline{T}_{\underline{S}}}=\displaystyle\sum_{|\underline{Q}|=q,\, G(\underline{Q})=\underline{n},\,\|\underline{Q}\|=q-1}\displaystyle\frac{1}{q}\displaystyle\frac{q!}{\underline{Q}!} e_{\underline{Q},\underline{T}_{\underline{S}}},$$
so $ \displaystyle\frac{1}{q}e_{\underline{T}_{\underline{S}}}\in\N$.
\end{proof}

\begin{remark}\label{rem:omega_0}
We have seen in Theorem \ref{theo:FS} and its proof (see Formula (\ref{equ:p_{k+1}}) with $\underline{k}=\underline{k}_0$) that $\omega_0=(\pi^P_{\underline{k}_0,\underline{i}_{\underline{k}_0}})'(c_{S(\underline{k}_0)})$ is the coefficient of the monomial $\underline{x}^{\underline{i}_{S(\underline{k}_0)}}y$ in the expansion of $P_{S(\underline{k}_0)}(\underline{x},y)=P(\underline{x},c_{(0,\ldots,0,1)}x_r+\cdots+c_{S(\underline{k}_0)}\underline{x}^{S(\underline{k}_0)}+ \underline{x}^{S^2(\underline{k}_0)}y)$, and that  $c_{S^2(\underline{k}_0)}=\displaystyle\frac{-\pi^P_{\underline{k}_0,\underline{i}_{S(\underline{k}_0)}}(c_{S(\underline{k}_0)})}{\omega_0}$ where $\pi^P_{\underline{k}_0,\underline{i}_{S(\underline{k}_0)}}(c_{S(\underline{k}_0)})$ is the coefficient of $\underline{x}^{\underline{i}_{S(\underline{k}_0)}}$ in the expansion of $P_{S(\underline{k}_0)}(\underline{x},y)$. Expanding $P_{S(\underline{k}_0)}(\underline{x},y)$, having done the whole computations, we deduce that:
$$\left\{\begin{array}{lcl}
\omega_0&=&\displaystyle\sum_{|\underline{i}|=0,..,d_x,\ j=1,..,d_y}\ \ \displaystyle\sum_{|\underline{L}|=j-1,\  G(\underline{L})=\underline{i}_{\underline{k}_0}-S(\underline{k}_0)-\underline{i}}\displaystyle\frac{j!}{\underline{L}!}a_{\underline{i},j}\underline{C}^{\underline{L}}\ ;\\
c_{S^2(k_0)}&=& \displaystyle\frac{-1}{\omega_0}\displaystyle\sum_{|\underline{i}|=0,..,d_x,\ j=0,..,d_y}\ \ \displaystyle\sum_{|\underline{L}|=j,\  G(\underline{L})=\underline{i}_{S(\underline{k}_0)}-\underline{i}}\ \ \displaystyle\frac{j!}{L!}a_{\underline{i},j}\underline{C }^{\underline{L}},
\end{array}\right. $$
where $\underline{C}:=\left(c_{(0,\ldots,0,1)},\ldots,c_{S(\underline{k}_0)}\right)$ and $\underline{L}:=\left(l_{(0,\ldots,0,1)},\ldots,l_{S(\underline{k}_0)}\right)$.
\end{remark}

\begin{ex}\label{ex:FS}
In order to illustrate Corollary \ref{coro:FS} and its proof, we resume the polynomial of  Example \ref{ex:wilc} with $a_{0,0,2}\neq 0$:
$$\begin{array}{lcl}
P(\underline{x},y)&=&a_{0,0,2}y^2+\left(a_{0,2,0}+a_{0,2,1}y+a_{0,2,2}y^2\right) {x_2}^2+\left(a_{2,2,0}   +a_{2,2,1}y+a_{2,2,2}y^2\right)x_1^2{x_2}^2\\
P_{\underline{0}}(\underline{x},y)&=&\left(a_{0,2,0}+a_{0,0,2}y^2\right)  {x_2}^2+a_{0,2,1}y{x_2}^3+a_{0,2,2}y^2{x_2}^4+
a_{2,2,0} {x_1}^2{x_2}^2  \\
&&+a_{2,2,1}y{x_1}^2{x_2}^3+a_{2,2,2}y^2 {x_1}^2{x_2}^4\\
P_{0,1}(\underline{x},y)&=&  2a_{0,0,2}c_{0,1}yx_1x_2+a_{0,0,2}{x_1}^2 y^2 + a_{0,2,1}c_{0,1}{x_2}^3 +a_{0,2,1}x_1{x_2}^2y+ a_{0,2,2}c_{0,1}^2{x_2}^4\\
&&+2a_{0,2,2}c_{0,1}x_1{x_2}^3y+\left(a_{2, 2, 0}+ a_{0,2,2}y^2\right){x_1}^2 {x_2}^2
 +a_{2,2,1}c_{0,1}{x_1}^2{x_2}^3\\
&&+a_{2,2,1}y{x_1}^3{x_2}^2+ a_{2,2,2}c_{0,1}^2{x_1}^2{x_2}^4 +2a_{2,2,2}c_{0,1}y{x_1}^3{x_2}^3+a_{2,2,2}y^2{x_1}^4{x_2}^2\\
&&\textrm{with }a_{0,2,0}+a_{0,0,2}{c_{0,1}}^2=0\,\Leftrightarrow\, c_{0,1}= \pm\sqrt{\displaystyle\frac{-a_{0,2,0}}{a_{0,0,2}}}.
\end{array}$$
 Thus, $ \underline{i}_{\underline{0}}=(0,2)$, $\underline{i}_{0,1}=(1,1)=\underline{i}_{\underline{0}}+(1,0)-(0,1)$, so $\underline{k}_0=\underline{0}$, $\omega_0=2\,a_{0,0,2}\,c_{0,1}$. The coefficient $c_{1,0}$ must verify $2a_{0,2}c_{0,1}c_{1,0} =0\, \Leftrightarrow\, c_{1,0}=0$ since $a_{0,2}c_{0,1}\neq 0$. We obtain that:
\begin{center}
$\begin{array}{lcl}
-\omega_0\cdot \, _{0,1}R&=&\displaystyle\frac{P_{0,1}(\underline{x},y+c_{1,0})}{x_1x_2}\\
&=&  \omega_0y +a_{0,0,2}x_1{x_2}^{-1} y^2 + a_{0,2,1}c_{0,1}{x_1}^{-1}{x_2}^2 +a_{0,2,1}x_2y+ a_{0,2,2}{c_{0,1}}^2{x_1}^{-1}{x_2}^3\\
&&+2a_{0,2,2}c_{0,1}{x_2}^2y+\left(a_{2, 2, 0}+ a_{0,2,2}y^2\right)x_1 x_2 +a_{2,2,1}c_{0,1}x_1{x_2}^2+a_{2,2,1}y{x_1}^2{x_2}\\
&&+ a_{2,2,2}{c_{0,1}}^2x_1{x_2}^3 +2a_{2,2,2}c_{0,1}y{x_1}^2{x_2}^2+a_{2,2,2}y^2{x_1}^3x_2.
\end{array}$
 \end{center}
So the coefficients of the corresponding strongly reduced Henselian equation $y=\, _{0,1}Q(x,y)$ are:
 \begin{center}
 $\begin{array}{lll}
 b_{1,-1,2}=-a_{0,0,2}/\omega_0, &b_{-1,2,0}=-a_{0,2,1}c_{0,1}/\omega_0, & b_{0,1,1}=-a_{0,2,1}/\omega_0,\\
 b_{-1,3,0}=-a_{0,2,2}{c_{0,1}}^2/\omega_0,& b_{0,2,1}= -2\,a_{0,2,2}c_{0,1}/\omega_0,& b_{1,1,0}=-a_{2,2,0} /\omega_0,\\
 b_{1,1,2}=-a_{0,2,2} /\omega_0,& b_{1,2,0}=-a_{2,2,1}c_{0,1}/\omega_0, &b_{2,1,1}=-a_{2,2,1} /\omega_0, \\
b_{1,3,0}=-a_{{2,2,2}}/\omega_0,&  b_{2,2,1}=-2a_{2,2,2}c_{0,1}/\omega_0,& b_{3,1,2}=-a_{2,2,2}/\omega_0.
 \end{array}$\end{center}
According to Example  \ref{ex:FS-gene}, applying the generalized Flajolet-Soria's Formula \ref{theo:formule-FS}, one obtains for the first terms that are  not trivially zero:\\
$\begin{array}{lcl}
c_{0,2}&=&b_{-1,2,0}=\displaystyle\frac{ -a_{0,2,1}c_{0,1} }{2\,a_{0,0,2}\,c_{0,1}}\\
&=&\displaystyle\frac{ -a_{0,2,1} }{2\,a_{0,0,2}};\\
c_{0,3}&=&b_{{-1,3,0}}+b_{{0,1,1}}b_{{-1,2,0}}+b_{{1,-1,2}}{b_{{-1,2,0}}}^{2}\\
&=&-{\displaystyle\frac {a_{{0,2,2}}{c_{0,1}}^{2}}{2\,a_{0,0,2}\,c_{0,1}}}+{\displaystyle\frac {{a_{{0,2,1}}}^{2}c_{0,1}}{\left(2\,a_{0,0,2}\,c_{0,1}\right)^{2}}} -{\displaystyle\frac {a_{{0,0,2}}{a_{{0,2,1}}}^{2}{c_{0,1}}^{2}}{\left(2\,a_{0,0,2}\,c_{0,1}\right)^{3}}} \\
&=&-{\displaystyle\frac {a_{{0,2,2}}{c_{0,1}}}{2\,a_{0,0,2}}}+{\displaystyle\frac {{a_{{0,2,1}}}^{2}}{8\,a_{0,0,2}^2\,c_{0,1}}};\\
c_{2,1}&=&b_{{1,1,0}}=-{\displaystyle\frac {a_{{2,2,0}} }{2\,a_{0,0,2}\,c_{0,1}}};\\
c_{0,4}&=&b_{{0,2,1}}b_{{-1,2,0}}+b_{{0,1,1}}b_{{-1,3,0}}+2\,b_{{1,-1,2}}b_{{-1,2,0}}b_{{-1,3,0}}+{b_{{0,1,1}}}^{
2}b_{{-1,2,0}}\\
&&+3\,b_{{0,1,1}}b_{{1,-1,2}}{b_{
{-1,2,0}}}^{2}+2\,{b_{{1,-1,2}}}^{2}{b_{{-1,2,0}}}^{3}\\
&=& \displaystyle\frac{1}{2}\,{\displaystyle\frac {a_{{0,2,1}}a_{{0,2,2}}}{{a_{{0,0,2}}}^{2}}};\\
&\vdots&\end{array}$
\end{ex}


As a consequence of Theorem \ref{theo:wilc} and Corollary \ref{coro:FS}, we get the following result. Let $d_x$, $d_y$ be some fixed degrees, and a multi-index $\underline{p}\in\mathbb{N}^r$ such that $|\underline{p}|>2d_xd_y+2$. There is a finite number of universal polynomial formulas which compute the coefficient $c_{\underline{p}}$ of any algebraic series $y_0$ of degrees at most $d_x$, $d_y$. These formulas are evaluated at the first coefficients of the terms  of $y_0$ of degree at most  $2d_xd_y+2$, and their number is independent of $\underline{p}$. More precisely:

\begin{coro}\label{coro:param-ratio}
Let $d_x,d_y\in \mathbb{N}^*$. We set $M_1:=\displaystyle\frac{1}{2}d_y(d_y+1)\displaystyle\binom{d_x+r}{r}+d_y-2$, $M_2:=2\left(d_y(2d_yd_x+1)+d_x+1\right)^{r-1}$ and  $\underline{k}:=(0,\ldots,0,1,2d_xd_y)$. There exists a finite set $\Lambda$ and for any $\lambda\in\Lambda$, there exist a polynomial $\Omega^{(\lambda)}(C_{(0,\ldots,0,1)},\ldots,C_{\underline{k}})\in \mathbb{Z}\left[C_{(0,\ldots,0,1)},\ldots,C_{\underline{k}}\right]$, $ \deg \Omega^{(\lambda)}\leq M_1$, and for every $\underline{n}\in\mathbb{Z}^r$,  $\underline{n}:=\underline{p}-S(\underline{k})$ for $\underline{p}\in\mathbb{N}^r$, $\underline{p}>_{\mathrm{grlex}}S(\underline{k})$, a polynomial $\Psi^{(\lambda)}_{\underline{n}}(C_{(0,\ldots,0,1)},\ldots,C_{S(\underline{k})})\in \mathbb{Z}\left[C_{(0,\ldots,0,1)},\ldots,C_{S(\underline{k})}\right]$, $ \deg \Psi^{(\lambda)}_{\underline{n}}\leq M_2|\underline{n}|(M_1+d_x)+1  $, such that for every $y_0=\displaystyle\sum_{\underline{p}>_{\mathrm{grlex}} \underline{0}}c_{\underline{p}}\underline{x }^{\underline{p}}$, $c_{(0,\ldots,0,1)}\neq 0$, algebraic with vanishing polynomial of degrees bounded by $d_x$ in $\underline{x}$ and $d_y$ in $y$, there exists $\lambda\in \Lambda$ such that for every $\underline{n}\in\mathbb{Z}^r$,  $\underline{n}:=\underline{p}-S(\underline{k})$ for $\underline{p}\in\mathbb{N}^r$, $\underline{p}>_{\mathrm{grlex}}S(\underline{k})$:
$$c_{\underline{p}}=c_{S(\underline{k})+\underline{n}}= \displaystyle\frac{\Psi^{(\lambda)}_{\underline{n}}(c_{(0,\ldots,0,1)},\ldots,c_{S(\underline{k})})}{\Omega^{(\lambda)}(c_{(0,\ldots,0,1)},\ldots,c_{\underline{k}})^{M_2|\underline{n}|}}.$$
\end{coro}

\begin{proof}
Let $y_0=\displaystyle\sum_{\underline{p}>_{\mathrm{grlex}} \underline{0}}c_{\underline{p}}\underline{x }^{\underline{p}}$\,, $c_{(0,\ldots,0,1)}\neq 0$, be algebraic with vanishing polynomial of degrees bounded by $d_x$ in $\underline{x}$ and $d_y$ in $y$. According to Theorem \ref{theo:wilc}, for $N:=2d_xd_y$, there is a finite set $\Lambda$ and for every $\lambda\in\Lambda$, there are polynomials $a^{(\lambda)}_{i,j}(C_{(0,\ldots,0,1)},\ldots,C_{(N,0,\ldots,0)})\in \mathbb{Z}[C_{(0,\ldots,0,1)},\ldots,C_{(N,0,\ldots,0)}]$ such that:
$$P^{(\lambda)}=\displaystyle\sum_{|\underline{i}|\leq d_x,j\leq d_y}a_{\underline{i},j}^{(\lambda)}(C_{(0,\ldots,0,1)},\ldots,C_{(N,0,\ldots,0)})\underline{x}^{\underline{i}}y^j$$ 
is a vanishing polynomial for $y_0$ for a certain $\lambda\in\Lambda$. Enlarging the finite set $\Lambda$ by indices corresponding to the various $\displaystyle\frac{\partial^k P^{(\lambda)}}{\partial y^k}$, $k=1,\ldots,d_y-1$, we can assume that there is $\lambda$ such that $y_0$ is a simple root of $P^{(\lambda)}$. So the coefficients of $y_0$ can be computed as in  Corollary \ref{coro:FS}. More precisely, for any $\underline{n}\in\mathbb{Z}^r$,  $\underline{n}:=\underline{p}-S(\underline{k})$ for $\underline{p}\in\mathbb{N}^r$, $\underline{p}>_{\mathrm{grlex}}S(\underline{k})$:
\begin{equation}\label{equ:deg-coeff} c_{\underline{p}}=c_{S(\underline{k})+\underline{n}}= \displaystyle\sum_{q=1}^{\mu_{\underline{n}}} \displaystyle\sum_{\underline{S}\in I_q} \displaystyle\sum_{\underline{T}_{\underline{S}}\in J_{\underline{S}}}\displaystyle\frac{m_{\underline{S },\underline{T}_{\underline{S}}}}{\omega_0^{\mu_{\underline{n}}}} \omega_0^{\mu_{\underline{n}}-q} \underline{A}^{\underline{S}} \underline{C}^{\underline{T}_{\underline{S}}}
\end{equation}
where  $\mu_{\underline{n}}$ is as in Theorem \ref{theo:formule-FS} for the equation $ y=\, _{\underline{k}}Q^{(\lambda)}(\underline{x},y)$ of Theorem \ref{theo:FS}, $I_q=\left\{(s_{\underline{i},j})\ |\ |\underline{S}|=q,\ \|\underline{S}\|\geq q-1\right\}$, $$J_{\underline{S}}=\left\{(t_{\underline{S},\underline{i}})\ |\ |\underline{T}_{\underline{S}}|=\|\underline{S}\|-q+1,\ G(\underline{T}_{\underline{S}})=\underline{n}+q\,\underline{i}_{\underline{k}}-(q-1)S(\underline{k})-G(\underline{S})\right\}$$ and $m_{\underline{S},\underline{T}_{\underline{S}}}\in \mathbb{Z}$. Note that $\underline{C}=(c_{(0,\ldots,0,1)},\ldots,c_{S(\underline{k})})$ and $\underline{A}=\left(a_{i,j}^{(\lambda)}(c_{(0,\ldots,0,1)},\ldots,c_{(N,0,\ldots,0)})\right)$. Let us show that  $\mu_{\underline{n}}\leq M_2|\underline{n}|$. First, note that $\underline{i}_{\underline{k}}$ in Notation \ref{nota:P_k} verifies $|\underline{i}_{\underline{k}}|\leq \left(2d_xd_y+1\right)d_y+d_x$. Indeed: 
$$\begin{array}{lcl}
\underline{i}_{\underline{k}}&:=&w\left(P^{(\lambda)}_{\underline{k}}(\underline{x},y)\right)\\
&=& w\left(\displaystyle\sum_{|\underline{i}|\leq d_x,j\leq d_y} a_{\underline{i},j}\underline{x}^{\underline{i}} (z_{\underline{k}}+\underline{x}^{S(\underline{k})} )^j\right),
\end{array}
$$
so $|\underline{i}_{\underline{k}}|\leq d_x+d_y|S(\underline{k})|$ with $|S(\underline{k})|=2d_xd_y+1$. This implies that $\underline{\iota}_0$ of Theorem \ref{theo:formule-FS} for the equation
$ y=\, _{\underline{k}}Q^{(\lambda)}(\underline{x},y)$ verifies $\iota_{0,s}\leq \left(2d_xd_y+1\right)d_y+d_x$ for $s=1,\ldots,r$. With the notations of Theorem \ref{theo:FS}, we get that $\lambda_1\leq 2\left((2d_xd_y+1)d_y+d_x+1\right)^{r-1}=:M_2$. Finally, by Remark 
\ref{rem:FS-contraint}, we obtain that $\mu_{\underline{n}}\leq M_2|\underline{n}|$.\\
Let us rewrite (\ref{equ:deg-coeff}) as:
$$ c_{\underline{p}}=c_{S(\underline{k})+\underline{n}}= \displaystyle\sum_{q=1}^{\mu_{\underline{n}}} \displaystyle\sum_{\underline{S}\in I_q} \displaystyle\sum_{\underline{T}_{\underline{S}}\in J_{\underline{S}}}\displaystyle\frac{m_{\underline{S},\underline{T}_{\underline{S}}}}{\omega_0^{M_2|\underline{n}|}} \omega_0^{M_2|\underline{n}|-q}\underline{A}^{\underline{S}} \underline{C}^{\underline{T}_{\underline{S}}}$$
It now suffices to bound the degrees of the numerator and denominator in the terms of the previous formula. By  Theorem \ref{theo:wilc}, $\deg a_{i,j}^{(\lambda)}\leq M_1 +1-d_y$ for $j\geq 1$ and $\deg a_{i,0}^{(\lambda)}\leq M_1+1-d_y+d_x$. So by Remark \ref{rem:omega_0} and Theorem \ref{theo:wilc}, we deduce that $\omega_0$ is the evaluation of a polynomial $\Omega^{(\lambda)}(C_{(0,\ldots,0,1)},\ldots,C_{\underline{k}})$ such that  $\deg \Omega^{(\lambda)}\leq M_1$. The degree $d_{q,\underline{S}}$ corresponding to a term $\omega_0^{M_2|\underline{n}|-q}\underline{A}^{\underline{S}} \underline{C}^{\underline{T}_{\underline{S}}}$ is bounded by:
$$(M_2|\underline{n}|-q)M_1+|\underline{S}|(M_1+d_x+1-d_y)+|\underline{T}_{\underline{S}}|=(M_2|\underline{n}|-q)M_1+q(M_1+d_x+1-d_y)+\|\underline{S}\|-q+1.$$
But, $\|\underline{S}\|\leq q\,d_y$ and $1\leq q\leq \mu_{\underline{n}}\leq M_2|\underline{n}|$. So we get that:
$$d_{q,\underline{S}}\leq (M_2|\underline{n}|-q)M_1+q(M_1+d_x+1-d_y)+q\,d_y-q+1\leq M_2|\underline{n}|(M_1+d_x)+1.$$
\end{proof}

\begin{ex}
We resume Example \ref{ex:wilc} for which $2d_xd_y+1=9$, hence $\underline{k}=(1,8)$.  By Corollary \ref{coro:param-ratio}, the coefficients $c_{\underline{p}}$ for $\underline{p}>_{\mathrm{grlex}}S(\underline{k})=(2,7)$ are expressed as values of rational fractions in the first $c_{0,1},\ldots,c_{2,7}$ coefficients. In fact, such rational parametrizations might occur at a lower rank, i.e. in terms of fewer coefficients. E.g., by substituting the results in Example \ref{ex:wilcz2} into the ones in Example \ref{ex:FS}  under Conditions (\ref{equ:exemple}), we obtain the following relations for the first not trivially zero terms. 
$$\begin{array}{lcl}
c_{0,2}&=&\displaystyle\frac{ -a_{0,2,1} }{2\,a_{0,0,2}}=c_{0,2};\\
c_{0,3}&=&-{\displaystyle\frac {a_{{0,2,2}}{c_{0,1}}^{2}}{2\,a_{0,0,2}\,c_{0,1}}}+{\displaystyle\frac {{a_{{0,2,1}}}^{2}}{8\,{a_{0,0,2}}^2\,c_{0,1}}} \\
&=&\displaystyle\frac{1}{2}\,{\displaystyle\frac {2\,{c_{0,1}}c_{{0,3}}-{c_{{0,2}}}^{2}}{c_{0,1}}}+\displaystyle\frac{1}{2}\,{\frac {{c_{{0,2}}}^{2}}{c_{0,1}}}={c_{0,3}};\\
c_{2,1}&=&-{\displaystyle\frac {a_{{2,2,0}} }{2\,a_{0,0,2}\,c_{0,1}}}=c_{{2,1}}
;\\
c_{0,4}&=& \displaystyle\frac{1}{2}\,{\displaystyle\frac {a_{{0,2,1}}a_{{0,2,2}}}{{a_{{0,0,2}}}^{2}}}\\
&=& {\displaystyle\frac {c_{{0,2}} \left( 2\,c_{0,1}\,c_{{0,3}}-{c_{{0,2}}}^{2} \right)}{{c_{0,1}}^{2}}};\\
&\vdots&\end{array}$$
A relation of type $c_{i_1,i_2}=c_{i_1,i_2}$ simply means that $c_{i_1,i_2}$ is a free parameter. Recall that these formulas actually give rational parametrizations for the coefficients of Puiseux series solutions of equations as in Example \ref{ex:eclt} i.e. equations that are of degrees bounded by $1$ in $\underline{x}$ and by $2$ in $y$.
\end{ex}

As an immediate consequence of the previous result, we obtain a proof of the multivariate version  of Eisenstein theorem on algebraic power series due to K. V. Safonov \cite[Theorem 5]{safonov:algebraic-power-series}:

\begin{coro}\label{coro:eisenstein}
Let $y_0=\displaystyle\sum_{\underline{n}\in\mathbb{N}^r}c_{\underline{n}}\underline{x}^{\underline{n}}\ \in  \mathbb{Q}\left[\left[\underline{x}\right]\right]$ be algebraic over $\mathbb{Q}(\underline{x})$. Then there exists $\delta_0,\delta\in\mathbb{N}^*$ such that for every $\underline{n}\in\mathbb{N}^r$: $$\delta_0\cdot\delta^{|\underline{n}|}\cdot c_{\underline{n}}\in\mathbb{Z}.$$
\end{coro}
\begin{proof}

Without loss of generality, we may assume that $c_{\underline{0}}=0$ and $c_{(0,\ldots,0,1)}\neq 0$ as in Section \ref{section:preliminaries}. Consider $d_x$ and $d_y$ such that $y_0$ is a root of a polynomial of degrees bounded by $d_x$ and $d_y$. For $\underline{k}:=(0,\ldots,0,1,2d_xd_y)$ as in Corollary \ref{coro:param-ratio}, by multiplication of the coefficients $c_{\underline{n}}$ by a suitable $\delta_0\in\mathbb{N}^*$, one can assume that $c_{(0,\ldots,0,1)},\ldots,c_{\underline{k}}\in\mathbb{Z}$. Now, take $\lambda\in \Lambda$ as in Corollary \ref{coro:param-ratio} and $\delta$ the absolute value of ${\Omega^{(\lambda)}(c_{(0,\ldots,0,1)},\ldots,c_{\underline{k}})^{M_2}}$.
\end{proof}

\section{Appendix}\label{section:appendice}
For families of multi-indices $\mathcal{F}=\left\{(0,2,1),(2,2,1),(0,0,2),(0,2,2),(2,2,2)\right\}$ \textrm{ and } $\mathcal{G}=\left\{(0,2,0),(2,2,0)\right\}$ as in Example \ref{ex:wilc} (see also  Definition \ref{defi:alg-relative} and Examples \ref{ex:eclt} and \ref{ex:wilcz2}), the corresponding Wilczynski matrix is:\\
 
$ M :=\\
\left[\begin{array}{ccccccc}
0 & 0 & 0 & 0& 0 & 0 & 0 \\
0 & 0 & 0 & 0& 0 & 0 & 0 \\
1 &0 & 0 & 0 & {\mathop{\rm c}_{0,1}}^{2} & 0&0 \\
0 &0 & 0 & 0 & 2~\mathop{\rm c}_{0,1}~\mathop{\rm c}_{1,0} & 0&0 \\
 0 & 0 & 0 & 0& {\mathop{\rm c}_{1,0}}^{2}&0 & 0 \\
0 &0  & \mathop{\rm c}_{0,1}&0&2\,c_{{0,1}}c_{{0,2}}&0&0\\
0 &0 &\mathop{\rm c}_{1,0}&0 &  2\,c_{{0,1}}c_{{1,1}}+2\,c_{{1,0}}c_{{0,2}}  & 0&0 \\
0&0 &0 & 0&   2\,c_{{0,1}}c_{{2,0}}+2\,c_{{1,0}}c_{{1,1}}  & 0&0 \\
0&0 &0 &0 &2\,c_{{1,0}}c_{{2,0}}&0&0 \\ 
0&0 &c_{{0,2}}&0&2\,c_{{0,1}}c_{{0,3}}+{c_{{0,2}}}^{2}& \mathop{\rm c}_{0,1}^{2} &0\\
0&0 &c_{{1,1}}&0&\begin{array}{c}
2\,c_{{0,1}}c_{{1,2}}+2\,c_{{0,2}}c_{{1,1}}\\
+2\,c_{{1,0}}c_{{0,3}}
\end{array} & 2~\mathop{\rm c}_{0,1}~\mathop{\rm c}_{1,0} &0\\
0&1 &c_{{2,0}}&0& \begin{array}{c}
2\,c_{{1,0}}c_{{1,2}}+{c_{{1,1}}}^{2}\\
+2\,c_{{0,2}}c_{{2,0}}+2\,c_{{0,1}}c_{{2,1}} 
\end{array}   &{\mathop{\rm c}_{1,0}}^{2}& 0\\
0&0 &0 & 0 & \begin{array}{c}
 2\,c_{{0,1}}c_{{3,0}}+2\,c_{{1,1}}c_{{2,0}}\\
 +2\,c_{{1,0}}c_{{2,1}}
\end{array} &0&0\\ 
0&0 &0 &0 &{c_{{2,0}}}^{2}+2\,c_{{1,0}}c_{{3,0}}&0&0\\ 
0&0 &c_{{0,3}}&0& c_{{0,5}}^{(2)}& 2\,c_{{0,1}}c_{{0,2}}& 0\\
0&0 &c_{{1,2}}&0& c_{{1,4}}^{(2)}&  2\,c_{{0,1}}c_{{1,1}}+2\,c_{{1,0}}c_{{0,2}} & 0\\
0&0 &c_{{2,1}}&\mathop{\rm c}_{0,1}&c_{{2,3}}^{(2)}&  2\,c_{{0,1}}c_{{2,0}}+2\,c_{{1,0}}c_{{1,1}}& 0\\
0&0 & c_{{3,0}} &\mathop{\rm c}_{1,0}&c_{{3,2}}^{(2)}&2\,c_{{0,1}}c_{{0,2}}&  0\\
0&0 &0&0& c_{{4,1}}^{(2)}& 0&  0\\
0&0 &0&0& c_{{5,0}}^{(2)}& 0&  0\\
0&0 &c_{{0,4}}&0& c_{{0,6}}^{(2)}&  2\,c_{{0,1}}c_{{0,3}}+{c_{{0,2}}}^{2}& 0\\
0&0 &c_{{1,3}}&0& c_{{1,5}}^{(2)}& \begin{array}{c}
 2\,c_{{0,1}}c_{{1,2}}+2\,c_{{0,2}}c_{{1,1}}\\
 +2\,c_{{1,0}}c_{{0,3}}
\end{array}&  0\\
0&0 &c_{{2,2}}&c_{{0,2}}& c_{{2,4}}^{(2)}&  \begin{array}{c}
2\,c_{{1,0}}c_{{1,2}}+{c_{{1,1}}}^{2}\\
+2\,c_{{0,2}}c_{{2,0}}+2\,c_{{0,1}}c_{{2,1}} 
\end{array}&  {\mathop{\rm c}_{0,1}}^{2} \\
0&0 &c_{{3,1}}&\mathop{\rm c}_{1,1} & c_{{3,3}}^{(2)}&  \begin{array}{c}
 2\,c_{{0,1}}c_{{3,0}} +2\,c_{{1,1}}c_{{2,0}}\\
 +2\,c_{{1,0}}c_{{2,1}}
\end{array}& 2 \mathop{\rm c}_{0,1}\mathop{\rm c}_{1,0} \\
0&0 &c_{{4,0}}&\mathop{\rm c}_{2,0}&c_{{4,2}}^{(2)}&  {c_{{2,0}}}^{2}+2\,c_{{1,0}}c_{{3,0}}&  {\mathop{\rm c}_{1,0}}^{2} \\
\vdots & \vdots & \vdots&\vdots & \vdots & \vdots & \vdots \\
\end{array}\right],\\ $

and the reduced matrix is:\\

$ M ^{{\it red} } :=
\left[\begin{array}{ccccc}
 0 & 0& 0 & 0 & 0 \\
 0 & 0& 0 & 0 & 0 \\
 0 & 0 & 2~\mathop{\rm c}_{0,1}~\mathop{\rm c}_{1,0} & 0&0 \\
 0 & 0& {\mathop{\rm c}_{1,0}}^{2}&0 & 0 \\
 \mathop{\rm c}_{0,1}&0&2\,c_{{0,1}}c_{{0,2}}&0&0\\
\mathop{\rm c}_{1,0}&0 &  2\,c_{{0,1}}c_{{1,1}}+2\,c_{{1,0}}c_{{0,2}}  & 0&0 \\
0 & 0&   2\,c_{{0,1}}c_{{2,0}}+2\,c_{{1,0}}c_{{1,1}}  & 0&0 \\
0 &0 &2\,c_{{1,0}}c_{{2,0}}&0&0 \\ 
c_{{0,2}}&0&2\,c_{{0,1}}c_{{0,3}}+{c_{{0,2}}}^{2}& {\mathop{\rm c}_{0,1}}^{2} &0\\
c_{{1,1}}&0& \begin{array}{c}
 2\,c_{{0,1}}c_{{1,2}}+2\,c_{{0,2}}c_{{1,1}}\\
 +2\,c_{{1,0}}c_{{0,3}}
\end{array} & 2~\mathop{\rm c}_{0,1}~\mathop{\rm c}_{1,0} &0\\
0 & 0 &  \begin{array}{c}
2\,c_{{0,1}}c_{{3,0}}+2\,c_{{1,1}}c_{{2,0}}\\
+2\,c_{{1,0}}c_{{2,1}}
\end{array} &0&0\\ 
0 &0 &{c_{{2,0}}}^{2}+2\,c_{{1,0}}c_{{3,0}}&0&0\\ 
c_{{0,3}}&0& c_{{0,5}}^{(2)}& 2\,c_{{0,1}}c_{{0,2}}& 0\\
c_{{1,2}}&0& c_{{1,4}}^{(2)}&  2\,c_{{0,1}}c_{{1,1}}+2\,c_{{1,0}}c_{{0,2}} & 0\\
c_{{2,1}}&\mathop{\rm c}_{0,1}&c_{{2,3}}^{(2)}&  2\,c_{{0,1}}c_{{2,0}}+2\,c_{{1,0}}c_{{1,1}}& 0\\
 c_{{3,0}} &\mathop{\rm c}_{1,0}&c_{{3,2}}^{(2)}&2\,c_{{0,1}}c_{{0,2}}  &  0\\
0&0& c_{{4,1}}^{(2)}& 0&  0\\
0&0& c_{{5,0}}^{(2)}& 0&  0\\
c_{{0,4}}&0& c_{{0,6}}^{(2)}&  2\,c_{{0,1}}c_{{0,3}}+{c_{{0,2}}}^{2}& 0\\
c_{{1,3}}&0& c_{{1,5}}^{(2)}& \begin{array}{c}
2\,c_{{0,1}}c_{{1,2}}+2\,c_{{0,2}}c_{{1,1}} \\
+2\,c_{{1,0}}c_{{0,3}}
\end{array} &  0\\
c_{{2,2}}&c_{{0,2}}& c_{{2,4}}^{(2)}&  \begin{array}{c}
 2\,c_{{1,0}}c_{{1,2}}+{c_{{1,1}}}^{2}\\ +2\,c_{{0,2}}c_{{2,0}}+2\,c_{{0,1}}c_{{2,1}}
\end{array}&  {\mathop{\rm c}_{0,1}}^{2} \\
c_{{3,1}}&\mathop{\rm c}_{1,1} &c_{{3,3}}^{(2)}& \begin{array}{c}
 2\,c_{{0,1}}c_{{3,0}}+2\,c_{{1,1}}c_{{2,0}}\\
 +2\,c_{{1,0}}c_{{2,1}}
\end{array} & 2 \mathop{\rm c}_{0,1}\mathop{\rm c}_{1,0} \\
c_{{4,0}}&\mathop{\rm c}_{2,0}&c_{{4,2}}^{(2)}&  {c_{{2,0}}}^{2}+2\,c_{{1,0}}c_{{3,0}}&  {\mathop{\rm c}_{1,0}}^{2} \\
 \vdots&\vdots & \vdots & \vdots & \vdots \\
\end{array}\right].$




\bibliographystyle{amsalpha}

\begin{thebibliography}{EKM{\etalchar{+}}01}

\bibitem[AI09]{aroca-ilardi:puiseux-multivar}
F.~Aroca and G.~Ilardi, \emph{A family of algebraically closed fields
  containing polynomials in several variables}, Comm. Algebra \textbf{37}
  (2009), no.~4, 1284--1296. \MR{2510985 (2010f:12008)}

\bibitem[EKM{\etalchar{+}}01]{evans-al:tot-ord-commut-monoids}
K.~Evans, M.~Konikoff, J.~J. Madden, R.~Mathis, and G.~Whipple, \emph{Totally
  ordered commutative monoids}, Semigroup Forum \textbf{62} (2001), no.~2,
  249--278.

\bibitem[EP05]{engler-prestel:valued-fields}
A.~J. Engler and A.~Prestel, \emph{Valued fields}, Springer Monographs in
  Mathematics, Springer-Verlag, Berlin, 2005.

\bibitem[FS97]{flajolet-soria:coeff-alg-series}
P.~Flajolet and M.~Soria, \emph{Coefficients of algebraic series}, Algorithms
  seminar 1997-1998, Tech. Report, INRIA, 1997.

\bibitem[GP00]{gonzalez-perez_singul-quasi-ord}
P.~D. Gonz{\'a}lez~P{\'e}rez, \emph{Singularit\'es quasi-ordinaires toriques et
  poly\`edre de {N}ewton du discriminant}, Canad. J. Math. \textbf{52} (2000),
  no.~2, 348--368.

\bibitem[Hah07]{hahn:nichtarchim}
H.~Hahn, \emph{{\"U}ber die nichtarchimedischen {G}r\"ossensystem},
  Sitzungsberichte der Kaiserlichen Akademie der Wissenschaften, Mathematisch -
  Naturwissenschaftliche Klasse (Wien) \textbf{116} (1907), no.~Abteilung IIa,
  601--655.

\bibitem[Hen64]{henrici:lagr-burmann}
P.~Henrici, \emph{An algebraic proof of the {L}agrange-{B}\"urmann formula}, J.
  Math. Anal. Appl. \textbf{8} (1964), 218--224.

\bibitem[HM15]{hickel-matu:puiseux-alg}
M.~Hickel and M.~Matusinski, \emph{On the algebraicity of puiseux series},
  preprint submitted for publication, http://arxiv.org/abs/1503.04965, 2015.

\bibitem[Leg30]{legendre:theorie-nbres}
A.-M. Legendre, \emph{Th\'eorie des nombres t.1}, Firmin-Didot (Paris), 1830.

\bibitem[McD95]{mcdonald_puiseux-multivar}
J.~McDonald, \emph{Fiber polytopes and fractional power series}, J. Pure Appl.
  Algebra \textbf{104} (1995), no.~2, 213--233.

\bibitem[Ray74]{rayner_puiseux-multivar}
F.~J. Rayner, \emph{Algebraically closed fields analogous to fields of
  {P}uiseux series}, J. London Math. Soc. (2) \textbf{8} (1974), 504--506.

\bibitem[Rib92]{rib:series-fields-alg-closed}
P.~Ribenboim, \emph{Fields: algebraically closed and others}, Manuscripta Math.
  \textbf{75} (1992), no.~2, 115--150.

\bibitem[RvdD84]{rib-vdd_ratio-funct-field}
P.~Ribenboim and L.~van~den Dries, \emph{The absolute {G}alois group of a
  rational function field in characteristic zero is a semidirect product},
  Canad. Math. Bull. \textbf{27} (1984), no.~3, 313--315.

\bibitem[Saf00]{safonov:algebraic-power-series}
K.~V. Safonov, \emph{On power series of algebraic and rational functions in
  {${\bf C}^n$}}, J. Math. Anal. Appl. \textbf{243} (2000), no.~2, 261--277.

\bibitem[Sat83]{sathaye:newt-puiseux-exp_abh-moh-semigr}
A.~Sathaye, \emph{Generalized {N}ewton-{P}uiseux expansion and
  {A}bhyankar-{M}oh semigroup theorem}, Inventiones Mathematicae \textbf{74}
  (1983), 149--157, 10.1007/BF01388535.

\bibitem[Sin80]{singmaster:binomial-multinomial}
D.~Singmaster, \emph{Divisibility of binomial and multinomial coefficients by
  primes and prime powers}, A collection of manuscripts related to the
  {F}ibonacci sequence, Fibonacci Assoc., Santa Clara, Calif., 1980,
  pp.~98--113.

\bibitem[Sok11]{sokal:implicit-function}
A.~D. Sokal, \emph{A ridiculously simple and explicit implicit function
  theorem}, S\'em. Lothar. Combin. \textbf{61A} (2009/11), Art. B61Ad, 21.

\bibitem[SV06]{soto-vicente:polyhedral-cones}
M.~J. Soto and J.~L. Vicente, \emph{Polyhedral cones and monomial blowing-ups},
  Linear Algebra Appl. \textbf{412} (2006), no.~2-3, 362--372.

\bibitem[SV11]{soto-vicente_puiseux-multivar}
\bysame, \emph{The {N}ewton procedure for several variables}, Linear Algebra
  Appl. \textbf{435} (2011), no.~2, 255--269. \MR{2782778}

\bibitem[Wal78]{walker_alg-curves}
R.~J. Walker, \emph{Algebraic curves}, Springer-Verlag, New York, 1978, Reprint
  of the 1950 edition.

\bibitem[{Wil}19]{wilczynski:alg-power-series}
E.~J. {Wilczynski}, \emph{{On the form of the power series for an algebraic
  function.}}, {Am. Math. Mon.} \textbf{26} (1919), 9--12 (English).

\end{thebibliography}

\newcommand{\etalchar}[1]{$^{#1}$}
\def\cprime{$'$} \def\cprime{$'$}
\providecommand{\bysame}{\leavevmode\hbox to3em{\hrulefill}\thinspace}
\providecommand{\MR}{\relax\ifhmode\unskip\space\fi MR }
\providecommand{\MRhref}[2]{%
  \href{http://www.ams.org/mathscinet-getitem?mr=#1}{#2}
}
\providecommand{\href}[2]{#2}


%
\end{document}